\documentclass[11pt]{amsart}

\usepackage{amssymb,latexsym}
\usepackage{amsthm}
\usepackage{amsmath}
\usepackage{amsfonts}
\usepackage{mathrsfs}
\usepackage{mathdots}
\usepackage{graphicx}
\usepackage[all]{xy}
\usepackage{chngcntr}
\usepackage{fancyhdr}
\usepackage[top=1in,bottom=1in,left=1.25in,right=1.25in]{geometry}

\newtheorem{theorem}{Theorem}[section]
\newtheorem{definition}[theorem]{Definition}
\newtheorem{remark}[theorem]{Remark}
\newtheorem{conjecture}[theorem]{Conjecture}
\newtheorem{proposition}[theorem]{Proposition}
\newtheorem{corollary}[theorem]{Corollary}
\newtheorem{lemma}[theorem]{Lemma}
\newtheorem{example}[theorem]{Example}

\def\Q{\mathbb{Q}}
\def\F{\mathbb{F}}
\def\R{\mathbb{R}}
\def\Z{\mathbb{Z}}
\def\A{\mathbb{A}}

\def\C{\mathbb{C}}
\def\G{\mathbb{G}}

\def\Ac{\mathscr{A}}

\def\D{\mathbb{D}}
\def\E{\mathcal{E}}

\def\Fc{\mathcal{F}}
\def\Fl{\mathscr{F}}

\def\Ll{\mathcal{L}}
\def\M{\mathcal{M}}
\def\N{\mathcal{N}}
\def\Ol{\mathcal{O}}

\def\S{\mathcal{S}}

\def\W{\mathcal{W}}
\def\Y{\mathcal{Y}}
\def\Zm{\mathcal{Z}}
\def\V{\mathcal{V}}

\def\Bun{\mathrm{Bun}}
\def\ch{\mathrm{ch}}

\def\Inv{\mathrm{Inv}}
\def\Isom{\mathrm{Isom}}
\def\Fil{\mathrm{Fil}}
\def\Hom{\mathrm{Hom}}
\def\Lie{\mathrm{Lie}}
\def\Mm{\mathrm{M}}
\def\Mod{\mathrm{Mod}}
\def\Nilp{\mathrm{Nilp}}
\def\GL{\mathrm{GL}}
\def\PGL{\mathrm{PGL}}

\def\Proj{\mathrm{Proj}}
\def\Rep{\mathrm{Rep}}
\def\Sh{\mathrm{Sh}}
\def\Sht{\mathrm{Sht}}
\def\GSp{\mathrm{GSp}}
\def\SO{\mathrm{SO}}
\def\GSpin{\mathrm{GSpin}}
\def\RZ{\mathrm{RZ}}
\def\Spin{\mathrm{Spin}}
\def\Spa{\mathrm{Sp}}
\def\Spa{\mathrm{Spa}}
\def\Spf{\mathrm{Spf}}
\def\Spec{\mathrm{Spec}}
\def\Adm{\mathrm{Adm}}

\def\lan{\langle}
\def\ran{\rangle}

\def\lra{\longrightarrow}
\def\ra{\rightarrow}
\def\ov{\overline}
\def\ul{\underline}
\def\wh{\widehat}
\def\wt{\widetilde}
\def\st{\stackrel}
\def\tr{\textrm}

\usepackage[pdftex]{hyperref}

\begin{document}

\title{On some generalized Rapoport-Zink spaces}
\author{Xu Shen}
\date{}
\address{Morningside Center of Mathematics, Academy of Mathematics and Systems Science, Chinese Academy of Sciences\\
	No. 55, Zhongguancun East Road\\
	Beijing 100190, China}\email{shen@math.ac.cn}
	
\renewcommand\thefootnote{}
\footnotetext{2010 Mathematics Subject Classification. Primary: 11G18; Secondary: 14G35.}

\renewcommand{\thefootnote}{\arabic{footnote}}
\keywords{Rapoport-Zink spaces, shtukas, Shimura varieties, K3 surfaces}

\begin{abstract}
We enlarge the class of Rapoport-Zink spaces of Hodge type by modifying the centers of the associated $p$-adic reductive groups. These such-obtained Rapoport-Zink spaces are called of abelian type.  The class of Rapoport-Zink spaces of abelian type is strictly larger than the class of Rapoport-Zink spaces of Hodge type, but the two type spaces are closely related as having isomorphic connected components. The rigid analytic generic fibers of Rapoport-Zink spaces of abelian type  can be viewed as moduli spaces of local $G$-shtukas in mixed characteristic in the sense of Scholze.

We prove that Shimura varieties of abelian type can be uniformized by the associated Rapoport-Zink spaces of abelian type. We construct and study the Ekedahl-Oort stratifications for the special fibers of Rapoport-Zink spaces of abelian type.  As an application, we deduce a Rapoport-Zink type uniformization for the supersingular locus of the moduli space of polarized K3 surfaces in mixed characteristic. Moreover, we show that the Artin invariants of supersingular K3 surfaces are related to some purely local invariants.

\end{abstract}

\maketitle
\setcounter{tocdepth}{1}
\tableofcontents

\section{Introduction}

The theory of Rapoport-Zink spaces finds its origin in the work of Drinfeld in \cite{Dr}. Let $E$ be a finite extension of $\Q_p$, and let $\Omega^d_E$ be the complement of all $E$-rational hyperplanes in the $p$-adic projective space $\mathbb{P}^{d-1}$ over $E$. In \cite{Dr} Drinfeld interpreted this rigid-analytic space $\Omega^d_E$ as the generic fibre of a formal scheme over $\Ol_E$ parametrizing certain $p$-divisible groups. He used this formal moduli scheme to $p$-adically uniformize certain Shimura curves and to construct \'etale coverings of $\Omega^d_E$.  In their foundational and seminal work \cite{RZ}, Rapoport and Zink generalized greatly the construction of Drinfeld by introducing general formal moduli spaces of $p$-divisible groups with EL/PEL structures, and proved these spaces $\breve{\M}$ can be used to uniformize certain pieces of general PEL type Shimura varieties. Moreover, Rapoport and Zink constructed \'etale coverings $\M_K$ of the generic fibers of these formal moduli spaces, and  realized these rigid analytic spaces as \'etale coverings of more general non-archimedean period domains.  Besides these importances in arithmetic geometry and $p$-adic Hodge theory, it was conjectured by Kottwitz that  the $\ell$-adic cohomology  of these Rapoport-Zink spaces $\M_K$ realizes the local Langlands correspondence for the related local reductive group $G$, cf. \cite{Ra} section 5.

Recently, in \cite{Kim1} Kim has constructed more general formal moduli spaces of $p$-divisible groups with additional structures. (Here and throughout the rest of this introduction we assume $p>2$.) These formal schemes $\breve{\M}$ are called of Rapoport-Zink spaces of Hodge type, associated to unramified local Shimura data of Hodge type $(G, [b], \{\mu\})$ (see below). The additional structures on $p$-divisible groups are given by the so called crystalline Tate tensors, cf. \cite{Kim1} Definition 4.6, generalizing the EL/PEL structures introduced by Rapoport-Zink (in the unramified case). Kim also constructed a tower $(\M_K)_K$ of rigid analytic spaces (as usual, $K\subset G(\Q_p)$ runs through open compact subgroups of $G(\Q_p)$), when passing to the generic fibers of these formal moduli schemes. These Rapoport-Zink spaces of Hodge type appear as local analogues of the recent work of Kisin \cite{Ki1} on integral canonical models of Shimura varieties of Hodge type. In \cite{Kim2} Kim has proved his Rapoport-Zink spaces of Hodge type can be used to uniformize certain pieces of Shimura varieties of Hodge type. If the unramified local Shimura datum of Hodge type comes from a Shimura datum of Hodge type, Howard and Pappas have given another (global) construction of the associated Hodge type Rapoport-Zink spaces. We refer to \cite{HP} for more details.

In this note, we show that we can in fact go ahead one step further: we will construct some (slightly) more general formal and rigid analytic Rapoport-Zink spaces, and we will show that these spaces can be used to uniformize (pieces of) Kisin's integral canonical models Shimura varieties of abelian type, cf. \cite{Ki1}. Moreover, we will give some interesting applications to the moduli spaces of K3 surfaces in mixed characteristic.

There are several motivations for our work here. In our previous work \cite{Sh1}, we constructed perfectoid Shimura varieties of abelian type. One of the main motivations for this work is to study the local geometric structures of these perfectoid Shimura varieties, and to study the local geometric structures of Kisin's integral models of Shimura varieties of abelian type \cite{Ki1}. Another motivation is the recent developments in the theory of local Shimura varieties. In \cite{RV}, Rapoport-Viehmann conjectured the existence of a rigid analytic tower \[(\M_K)_K\] associated to a local Shimura datum $(G, [b], \{\mu\})$, where\footnote{Here we have followed \cite{RV} to write a local Shimura datum as $(G, [b], \{\mu\})$.}
\begin{itemize}
	\item $G$ is a connected reductive group over $\Q_p$,
	\item $\{\mu\}$ is a conjugacy class of minuscule cocharacters $\mu: \G_m\ra G_{\ov{\Q}_p}$,
	\item $[b]$ is a $\sigma$-conjugacy class in the Kottwitz set $B(G, \mu)$ (see \cite{Kot2}  section 6)
\end{itemize}
These conjectural local Shimura varieties are intended to be generalizations of Rapoport-Zink spaces, and there should be a theory in the local situation as good as the classical theory of Shimura varieties (\cite{D}). Recently, using the theory of perfectoid spaces (\cite{Sch1}), and the developments in $p$-adic Hodge theory due to Fargues, Fargues-Fontaine, and Kedlaya-Liu \cite{F, FF, KL}, Scholze has almost given a solution for Rapoport-Viehmann's conjecture by constructing moduli of local $G$-shtukas in mixed characteristic (cf. \cite{Sch2}) \[(\Sht_K)_K\] as some reasonable geometric objects. These geometric objects are called diamonds there, a generalization of perfectoid spaces and analytic adic spaces. Along the way of construction, we get an infinite level moduli space $\Sht_\infty$, such that as diamonds we have $\Sht_\infty= \varprojlim_K\Sht_K$.

In fact, Scholze proved more: one can allow the conjugacy class of cocharacters $\{\mu\}$ non minuscule, contrary to the original requirement of Rapoport-Viehmann in \cite{RV}, and in fact one can allow several $\{\mu\}$'s. Thus this theory is the mixed characteristic analogue of the theory of moduli of shtukas in the function fields case (\cite{Var, HV}). 

Despite its great success, the method of \cite{Sch2} is purely generic\footnote{We have learnt very recently that Scholze's method also produces integral models of local Shimura varieties as $v$-sheaves, cf. \cite{Sch4, SW17}.}: a priori, one has no information on reduction mod $p$. In the case of EL/PEL Rapoport-Zink spaces $\M_K$, Scholze proved the associated diamonds $\M_K^\diamond$ are isomorphic to his moduli spaces of local $G$-shtukas $\Sht_K$. From the point of view of moduli, this means that one can switch $p$-divisible groups with additional structures to local $G$-shtukas. Thus, in these classical cases, one gets formal integral structures and can talk about reduction mod $p$. From now on, we assume that $G$ is unramified over $\Q_p$ and fix a reductive integral model $G_{\Z_p}$ of $G$ over $\Z_p$. Using Dieudonn\'e theory, one can prove the special fibers of formal Rapoport-Zink spaces (of EL/PEL/Hodge type) are closely related to the corresponding affine Deligne-Lusztig varieties \[X^{G}_\mu(b):=\{g\in G(L)/G(W)|\,g^{-1}b\sigma(g)\in G(W)\mu(p)G(W)\}, \] where $W=W(\ov{\F}_p), L=W_\Q, G(W)=G_{\Z_p}(W)$, and $\sigma$ is the Frobenius. More precisely, in the above definition we have fixed a representative $b\in G(L)$ of the class $[b]$. On $X^{G}_\mu(b)$, we have an action of $J_b(\Q_p)$, where $J_b$ is the $\sigma$-centralizer of $b$. These objects are defined purely group theoretically, and thus make sense for arbitrary $(G, [b], \{\mu\})$ (as in the case of Scholze's moduli of local $G$-shtukas). These affine Deligne-Lusztig varieties play a crucial role in understanding the reduction mod $p$ of Shimura varieties, cf. \cite{Ra1}.

In this paper, we introduce a class of local Shimura data, the so called unramified local Shimura data of abelian type, and for each such datum $(G, [b], \{\mu\})$, we construct a formal scheme $\breve{\M}$, and a tower of rigid analytic spaces $(\M_K)_{K}$ such that
\begin{itemize}
\item the reduced special fiber $\M_{red}(\ov{\F}_p)\simeq X^G_\mu(b)$;
\item the rigid analytic (adic) generic fiber $\breve{\M}^{ad}_\eta=\M_{G(\Z_p)}$;
\item the associated diamonds $\M_K^\diamond\simeq \Sht_K$.
\end{itemize}
Moreover, we can prove that there exists a preperfectoid space $\M_\infty$ over $L$ such that \[\M_\infty\sim \varprojlim_K\M_K,\] where the meaning of $\sim$ is as \cite{SW} Definition 2.4.1.
This class of unramified local Shimura data of abelian type is strictly larger than the class of unramified local Shimura data of Hodge type.
Thus, among all local Shimura data, we find a new and larger class such that
\begin{itemize}
\item there exists a formal model $\breve{\M}$, such that $\breve{\M}^{ad,\diamond}_\eta\simeq \Sht_{G(\Z_p)}, \quad \M_{red}(\ov{\F}_p)\simeq X^G_\mu(b)$;
\item there exists a preperfectoid space $\M_\infty$, such that $\M_\infty^\diamond\simeq \Sht_\infty$.
\end{itemize}
We remark that the analogue of the above two additional structures in the global situation of Shimura varieties of abelian type are known by \cite{Ki1, Sh1}. They are not known yet for general local Shimura data (or local shtuka data).

A local Shimura datum $(G, [b],  \{\mu\})$ is called of unramified Hodge type, if $G$ is unramified, and there exists an embedding $(G, [b],  \{\mu\})\hookrightarrow (\GL(V), [b'],  \{\mu'\})$ of local Shimura data, such that $\{\mu'\}$ corresponds to $(1^{r}, 0^{n-r})$ for some integral $1\leq r \leq n=\dim V$. Roughly, the class of local Shimura data of Hodge type is the largest class for which the associated Rapoport-Zink spaces can be realized as moduli of $p$-divisible groups with additional structures. In this paper we introduce the following notion.
A local Shimura datum $(G, [b], \{\mu\})$ is called of unramified abelian type, if there exists an unramified local Shimura datum of Hodge type $(G_1, [b_1],  \{\mu_1\})$ such that we have an isomorphism of the associated adjoint local Shimura data $(G^{ad}, [b^{ad}],  \{\mu^{ad}\})\simeq (G_1^{ad}, [b_1^{ad}],  \{\mu_1^{ad}\})$. This is the local analogue of a Shimura datum of abelian type\footnote{More precisely, our local Shimura data of abelian type are the local analogues of Shimura data of preabelian type.}. We remark that although by definition we only change the centers of the groups, there does not exist local Hodge embedding any more for a general local Shimura datum of abelian type  $(G, [b], \{\mu\})$ (as in the corresponding global situation of Shimura varieties). This means that the class of local Shimura data of (unramified) abelian type is strictly larger than the class of (unramified) Hodge type. By Serre's classification \cite{Ser}, the groups $G$ in this larger class consist exactly of all classical groups, see section 4. 

Our first main theorem is as follows. See Theorem \ref{T:ab}, Proposition \ref{P:local system}, Corollary \ref{C:local ab}.
\begin{theorem}\label{T:intr ab}
	Let $(G, [b], \{\mu\})$ be an unramified local Shimura datum of abelian type. Fix a representative $b\in G(L)$ of $[b]$. Then there exists a formal scheme $\breve{\M}(G,b,\mu)$, which is formally smooth, formally locally of finite type over $W$, such that \[\ov{\M}(G,b,\mu)^{perf}\simeq X_{\mu}^{G}(b).\] Here $\ov{\M}(G,b,\mu)^{perf}$ is the perfection of the special fiber $\ov{\M}(G,b,\mu)$, and $X_{\mu}^{G}(b)$ is the affine Deligne-Lusztig variety, considered as a perfect scheme by \cite{Zhu, BS}.
	The formal scheme $\breve{\M}:=\breve{\M}(G,b,\mu)$ is equipped with a transitive action of $J_b(\Q_p)$, compatible with the action of $J_b(\Q_p)$ on $X_{\mu}^{G}(b)$. Moreover, there exist a tower of rigid analytic spaces $(\M_K)_K$ and a preperfectoid space $\M_\infty$ over $L$ such that
	\begin{enumerate}
	\item $\breve{\M}^{ad}_\eta=\M_{G(\Z_p)}$,
	\item $\M_\infty\sim \varprojlim_K\M_K$,
	\item $\M_K^\diamond\simeq \Sht_K$,
	\item there exists a compatible system of \'etale morphism $\pi_{dR}: \M_K\ra \Fl\ell_{G,\mu}^{adm}$,
	\item there exists a Hodge-Tate period morphism $\pi_{HT}: \M_\infty\ra \Fl\ell_{G,\mu^{-1}}$.
	\end{enumerate}
\end{theorem}
Here $\Fl\ell_{G,\mu}^{adm}$ is the admissible locus in the $p$-adic flag variety $\Fl\ell_{G,\mu}$ associated to $(G, \{\mu\})$, cf. \cite{Ra2} Definition A.6 or \cite{CFS} Definition 3.1, and $\Fl\ell_{G,\mu^{-1}}$ is the $p$-adic flag variety associated to $(G, \{\mu^{-1}\})$. In fact, we will see in Corollary \ref{C:local ab} that $\pi_{HT}$ also factors through a locally closed subspace $\Fl\ell_{G,\mu^{-1}}^b\subset \Fl\ell_{G,\mu^{-1}}$.

The construction of $\breve{\M}(G,b,\mu)$ associated to $(G, [b], \{\mu\})$ as above is based on the following observations.
Take any unramified local Shimura datum of Hodge type $(G_1, [b_1],  \{\mu_1\})$ such that $(G^{ad}, [b^{ad}],  \{\mu^{ad}\})\simeq (G_1^{ad}, [b_1^{ad}],  \{\mu_1^{ad}\})$. We have the associated formal Rapoport-Zink space
$\breve{\M}(G_1,b_1,\mu_1)$ constructed by Kim \cite{Kim1}, by patching together Faltings's construction of deformation ring for $p$-divisible groups (with crystalline Tate tensors) with Artin's criterion for algebraic spaces. By \cite{Zhu} Proposition 3.11, $\ov{\M}(G_1,b_1,\mu_1)^{perf}\simeq X_{\mu_1}^{G_1}(b_1)$. For any local Shimura datum $(G, [b], \{\mu\})$, we have a $J_b(\Q_p)$-equivariant surjective map
\[\omega_G: X_{ \mu}^G(b)\lra c_{b,\mu}\pi_1(G)^\Gamma,\] which factors through the set of connected components $\pi_0(X_{\mu}^G(b))$. Here $\pi_1(G)$ is the algebraic fundamental algebraic group of $G$ and $\Gamma=\mathrm{Gal}(\ov{\Q}_p/\Q_p)$.  See subsection \ref{subsection:ADL} for the construction of this map and the element $c_{b,\mu}\in \pi_1(G).$  Moreover, by \cite{CKV} Theorem 1.2,
$J_b(\Q_p)$ acts transitively on $\pi_0(X_\mu^G(b))$. For any local Shimura datum $(G, [b], \{\mu\})$,  by \cite{CKV} Corollary 2.4.2, we have a cartesian diagram
\[\xymatrix{X_{\mu}^{G}(b)\ar[r]\ar[d] & X_{\mu^{ad}}^{G_1^{ad}}(b^{ad})\ar[d]\\
	c_{b,\mu}\pi_1(G)^\Gamma\ar[r] & c_{b^{ad},\mu^{ad}}\pi_1(G^{ad})^\Gamma.
}\]
In particular we apply the above diagram to $(G, [b], \{\mu\})$ and $(G_1, [b_1],  \{\mu_1\})$ as above.
Let $X_{\mu_1}^{G_1}(b_1)^+\subset X_{\mu_1}^{G_1}(b_1)$ be a fixed choice of fiber of the map $\omega_{G_1}:  X_{\mu_1}^{G_1}(b_1)\ra c_{b_1,\mu_1}\pi_1(G_1)^\Gamma$. This is isomorphic to the corresponding subset of $X_{\mu}^{G}(b)$. Let \[\breve{\M}(G_1,b_1,\mu_1)^+\subset \breve{\M}(G_1,b_1,\mu_1)\] be the open and closed subspace corresponding to $X_{\mu_1}^{G_1}(b_1)^+$. As $X_{\mu}^{G}(b)=J_b(\Q_p)X_{\mu}^{G}(b)^+$,  we get the formal scheme $\breve{\M}(G,b,\mu)$ whose special fiber satisfies $\ov{\M}(G,b,\mu)^{perf}\simeq X_{\mu}^{G}(b)$. By construction, this formal scheme does not depend on the choice of the Hodge type local Shimura datum $(G_1, [b_1],  \{\mu_1\})$.
The other properties can be proved similarly.

Let $\big(\M(G,b,\mu)_K\big)_{K\subset G(\Q_p)}$ and $\big(\M(G_1,b_1,\mu_1)_{K_1}\big)_{K_1\subset G_1(\Q_p)}$ be the two towers associated to $(G, [b], \{\mu\})$ and $(G_1, [b_1],  \{\mu_1\})$ as above. By construction, the two towers are locally isomorphic in the sense that there exist sub towers\footnote{Here a sub tower $(Y_K)_K$ of a tower $(X_K)_K$ of inverse system of rigid analytic spaces is by definition given by an inverse system of subspaces $Y_K\subset X_K$ with compatible transition maps.} $\big(\M(G,b,\mu)_K^+\big)_{K\subset G(\Q_p)}$  and $\big(\M(G_1,b_1,\mu_1)_{K_1}^+\big)_{K_1\subset G_1(\Q_p)}$ such that 
\[\M(G,b,\mu)^+_\infty=\M(G_1,b_1,\mu_1)^+_\infty,\]where $\M(G,b,\mu)^+_\infty$ is the preperfectoid space over $L$ such that \[\M(G,b,\mu)^+_\infty\sim \varprojlim_{K}\M(G,b,\mu)_K^+,\] and similarly for $\M(G_1,b_1,\mu_1)^+_\infty$.
This implies in particular that $\Fl\ell_{G,\mu}^{adm}=\Fl\ell_{G_1,\mu_1}^{adm}$.  The tower $\big(\M(G,b,\mu)_K\big)_{K\subset G(\Q_p)}$ can be recovered from $\big(\M(G,b,\mu)_K^+\big)_{K\subset G(\Q_p)}$ and $\pi_1(G)^\Gamma$ by the action of either $G(\Q_p)$ or $J_b(\Q_p)$.
We expect that such results hold true for any local shtuka data $(G, [b], \{\mu\})$ and $(G_1, [b_1],  \{\mu_1\})$  with the same adjoint data.

We note that the above construction is simpler than the corresponding global situation, cf. \cite{Sh1, Ki1}, where one has to make a quotient on each geometric connected component of Shimura varieties of Hodge type. 

In subsection 4.3 we will try to find a moduli interpretation for the formal scheme $\breve{\M}(G,b,\mu)$ associated to $(G, [b], \{\mu\})$ as above, cf. Proposition \ref{P:moduli}, which is a priori non canonical, however. It is desirable to find a more canonical moduli interpretation for $\breve{\M}(G,b,\mu)$. After the first version of this paper appeared on line, B\"ultel and Pappas have recently found an intrinsic moduli interpretation for $\breve{\M}(G,b,\mu)$ with $(G, [b], \{\mu\})$ of Hodge type under a certain nilpotent condition, cf. \cite{BP}. They use a notion of $(G,\mu)$-displays, which is purely group theoretical. We can naturally extend B\"ultel and Pappas's moduli interpretation to certain abelian type case $\breve{\M}(G,b,\mu)$ studied in this paper, cf. Theorem \ref{T:moduli}. As mentioned above, the further recent  progress of \cite{SW17} will give a canonical moduli interpretation for the formal scheme $\breve{\M}(G,b,\mu)$ in the general case, as moduli of local shtukas similar to that in section 5, cf. Remark \ref{Remarks} (3).

If the unramified local Shimura datum of abelian type $(G, [b], \{\mu\})$  comes from a Shimura datum of abelian type $(G,X)$, we can prove the following uniformization theorem. Let $K^p\subset G(\A_f^p)$ be a fixed sufficiently small open compact subgroup. Consider $S_K$, the Kisin integral canonical model over $W$ of the Shimura variety $\Sh_K$ with $K=G(\Z_p)K^p$. Let \[\phi: \mathfrak{Q}\ra  \mathfrak{G}_G\] be a Langlands-Rapoport parameter with the associated reductive group $I_\phi$ over $\Q$, such that $[b]=[b(\phi)]$ (see \cite{Ki2} 3.3.6 for the precise meaning of these objects, where a Langlands-Rapoport parameter is called an admissible morphism between the Galois gerbs $\mathfrak{Q}$ and $\mathfrak{G}_G$). Let $\breve{\M}=\breve{\M}(G,b,\mu)$. Fix a Langlands-Rapoport parameter  $\phi_0: \mathfrak{Q}\ra  \mathfrak{G}_{G^{ad}}$ for the adjoint group such that $\phi^{ad}=\phi_0$.  In section 6 we will construct a subspace $\Zm_{\phi_0, K^p}\subset \ov{S}_K$, such that the formal completion of $S_K$ along $\Zm_{\phi_0,K^p}$ can be defined. The following theorem was proved by Rapoport and Zink in the PEL type case (\cite{RZ}), and by Kim in the Hodge type case (\cite{Kim2}, see also \cite{HP}). It can be viewed as the geometric version of the Langlands-Rapoport description for the underlying $\ov{\F}_p$-points, cf. \cite{Ki2}. In fact, it was pointed out in the introduction of \cite{RV} that the works of Kisin \cite{Ki1, Ki2} should yield new Rapoport-Zink spaces (comp. \cite{HP}). Here we construct these spaces locally, and show that they admit global application (comp. \cite{RV} Remark 5.9). See Theorems \ref{Thm:uniformize} and \ref{T:uniformize basic}.
\begin{theorem}\label{thm:unif}
	We have an isomorphism of formal schemes over $W$
	\[ \Theta: \coprod_{[\phi],\phi^{ad}=\phi_0} I_\phi(\Q)\setminus \breve{\M}\times G(\A_f^p)/K^p\st{\sim}{\lra} \wh{S_K}_{/\Zm_{\phi_0,K^p}}, \]where $\phi_0: \mathfrak{Q}\ra  \mathfrak{G}_{G^{ad}}$ is a fixed Langlands-Rapoport parameter for the adjoint group $G^{ad}$, $[\phi]$ runs through the set of isomorphism classes of Langlands-Rapoport parameters $\phi$ for $G$ such that $\phi^{ad}=\phi_0$ and $[b]=[b(\phi)]$.
	When $[b]$ is basic, we have $\Zm_{\phi_0,K^p}=\ov{S}_K^b$ which is the basic locus, and the above isomorphism reduces to
	\[\Theta:  I_\phi(\Q)\setminus \breve{\M}\times G(\A_f^p)/K^p\st{\sim}{\lra} \wh{S_K}_{/\ov{S}_K^b}.\]
\end{theorem}
Unsurprisingly, we apply the tricks of Kisin as in \cite{Ki2} to deduce the theorem from the Hodge type case. One can also deduce rigid analytic and perfectoid versions of the above uniformization theorem.

We consider the examples of basic GSpin and special orthogonal groups Rapoport-Zink spaces. Let $\breve{\M}_1=\breve{\M}(\GSpin, b,\mu), \breve{\M}=\breve{\M}(\SO, b',\mu')$ be the associated basic Rapoport-Zink spaces, where $\GSpin=\GSpin(V,Q), \SO=\SO(V,Q)$ are unramified GSpin and special orthogonal groups associated to a quadratic space $(V,Q)$ over $\Q_p$, with $\dim V=n+2$ for some integer $n\geq 1$. By considering the $G$-zip associated to the universal $p$-divisible group with crystalline Tate tensors on the special fiber $\ov{\M_1}$ of $\breve{\M}_1$, we can define an Ekedahl-Oort stratification on $\ov{\M_1}$, and thus on $\M_{1red}$ (the reduced special fiber), which is the local analogue of the Ekedahl-Oort stratification for Shimura varieties of Hodge type (cf. \cite{Zh}). The index set of this stratification is a subset $^J\W^b$ of the absolute Weyl group of $G_1$, which is thus finite. In fact one can find by computation that, it is in bijection with some explicit finite set of integers.  
For each $w\in\, ^J\W^b$, we have the associated Ekedahl-Oort stratum $\M_{1w}$ of $\M_{1red}$. On the other hand, in \cite{HP}, Howard and Pappas introduced another stratification for the reduced special fiber $\M_{1red}$:
\[\M_{1red}=\coprod_{\Lambda}\M_{1\Lambda}^\circ,\]where $\Lambda$ runs through the set of vertex lattices, see loc. cit. section 5.  By Corollary \ref{C:RZ SO} $\breve{\M}\simeq \breve{\M}_1/p^\Z$, we get the induced Ekedahl-Oort and Howard-Pappas stratifications for $\M_{red}$.
The following theorem is proved in subsection \ref{subsection:HPEO}: see Theorem \ref{Thm:EO} and Corollary \ref{C:EO for SO} for more precise statements.
\begin{theorem}\label{T:EO orth}
Each Ekedahl-Oort stratum $\M_{1w}$ of $\M_{1red}$ is some (disjoint) union of Howard-Pappas strata. Similar result holds for $\M_{red}$.
\end{theorem}
For a similar result in the case of the basic unitary group $GU(1,n-1)$ Rapoport-Zink space, see \cite{VW} Theorem D. 

In fact, in subsection 7.1 we construct the Ekedahl-Oort stratification for the special fibers of arbitrary Rapoport-Zink spaces of abelian type, cf. Theorem \ref{T:RZ EO}. We can compare our geometric construction with the Ekedahl-Oort stratification for affine Deligne-Lusztig varieties (with hyperspecial levels) in \cite{GoHe}, cf. Proposition \ref{P:EO}. In subsection 7.2, we discuss a theorem of similar phenomenon as Theorem \ref{T:EO orth} (cf. Theorem \ref{T:FHN}) for an unramified local Shimura datum of abelian type $(G, [b], \{\mu\})$, with $(G,\{\mu\})$ fully Hodge-Newton decomposable in the sense of \cite{GoHeNi} Definition 2.1. Our discussion in this more general setting is indeed motivated by \cite{GoHeNi} Theorems 2.3 and 2.5, where a posteriori the classification there (for minuscule $\mu$) lies in our class of local Shimura data of abelian type. The basic GSpin and special orthogonal groups Rapoport-Zink spaces are just special cases where one can make things more explicit (by the work of \cite{HP}).

Specializing further to the case of K3 surfaces, we have some interesting applications. Take an integer $d\geq 1$ such that $p\nmid 2d$. Let $\Mm_{2d, K}$ be the moduli spaces of K3 surfaces $f: X\ra S$ together with a primitive polarization $\xi$ of degree $2d$ and a $K$-level structure over $W$. Recall that by the global integral Torelli theorem (cf. \cite{MP2} Corollary 5.15), the integral Kuga-Satake period map
\[\iota: \Mm_{2d,K}\lra S_{K}\] is an open immersion, where $S_K$ is the integral canonical model over $W$ of the Shimura variety $\Sh_K$ for $G=\SO(2,19)$,
see subsection \ref{subsection:mod K3} for more details.
Here, we assume that $K=K_pK^p$ with $K_p=G(\Z_p)$ is the fixed hyperspecial subgroup. 
Let $X$ be a supersingular K3 surface over $\ov{\F}_p$, then the discriminant of its N\'eron-Severi lattice is equal to \[-p^{2\sigma_0(X)}\] for some integer $1\leq \sigma_0(X)\leq 10$. The integer $\sigma_0(X)$ is called the Artin invariant of $X$. The following corollary is a consequence of the above theorems. Note that the group $G$ is adjoint and thus $\phi=\phi_0$.
\begin{corollary}[Corollaries \ref{C:K3 RZ} and \ref{C:K3 Artin} ]
	\begin{enumerate}
\item	Let $\phi$, $[b]$ and $\Zm_{\phi,K^p}$ be as in the above Theorem \ref{thm:unif}, and let $J_\phi$ be the pullback of $\Zm_{\phi,K^p}$ under the open immersion $\ov{\Mm}_{2d, K}\hookrightarrow \ov{S}_{K}$ of special fibers. Then we have the following identity \[\wh{\Mm_{2d, K}}_{/J_\phi} =\coprod_{j\in I}\breve{\N}/\Gamma_j, \]
	where $\breve{\N}\subset \breve{\M}(G, b,\mu)$ is an open subspace, $I$ is certain countable set, and for any $j\in I$, $\Gamma_j\subset J_b(\Q_p)$ is some discrete subgroup. If moreover $[b]=[b_0]$ is basic, then $J_\phi=\ov{\Mm}_{2d, K}^{ss}$ which is the supersingular locus in $\ov{\Mm}_{2d, K}$, and the above disjoint union is finite.
\item Let $x\in\ov{\Mm}_{2d, K}^{ss}(\ov{\F}_p)$ be a point, and $X_x$ the associated supersingular K3 surface over $\ov{\F}_p$. Then we have the identity between the Artin invariant $\sigma_0(X_x)$ and the type $t(\Lambda_x)$: \[\sigma_0(X_x)=\frac{t(\Lambda_x)}{2},\]where $\Lambda_x$ is the vertex lattice attached to the special lattice associated to $(X_x,\xi_x)$, cf. subsection \ref{subsection: K3 app}.
\end{enumerate}
\end{corollary}

We briefly describe the structure of this article. In section 2, we review some basics about affine Deligne-Lusztig varieties which will be used later. In section 3, we first recall the Rapoport-Viehmann conjecture on the theory of local Shimura varieties, then we concentrate on the case of unramified local Shimura datum of Hodge type, and review the construction of Kim \cite{Kim1} on the associated Rapoport-Zink spaces of Hodge type. In section 4, we introduce unramified local Shimura datum of abelian type, and construct the associated formal and rigid analytic Rapoport-Zink spaces. Section 5 is devoted to a review the general framework of moduli of local $G$-shtukas in mixed characteristic due to Scholze, to give a moduli interpretation of the generic fibers of our Rapoport-Zink spaces of abelian type. In section 6, we turn to the global situation of Shimura varieties of abelian type, and prove a Rapoport-Zink type uniformization theorem in this setting. In section 7, motivated by the study of Artin invariants of K3 surfaces, we construct the Ekedahl-Oort stratification for special fibers of Rapoport-Zink spaces. In section 8, we discuss some applications of our theory.  We work on the examples of basic GSpin and special orthogonal groups Rapoport-Zink spaces, and then more specially on the case of moduli spaces of K3 surfaces. These examples are just (related to) special cases of the fully Hodge-Newton decomposable Shimura varieties introduced in \cite{GoHeNi} (see also \cite{SZ}).  Finally, we investigate $p$-adic period domains in the basic orthogonal case in the appendix following Fargues.\\
 \\
\textbf{Acknowledgments.} I would like to thank Laurent Fargues for his constant encouragement and support in mathematics. In particular I thank him for kindly allowing me to include the content of \cite{F4} in the appendix. I wish to express my gratitude to Chao Zhang, from whom I learnt many about the Ekedahl-Oort stratifications. I would like to thank Miaofen Chen, Ruochuan Liu, Sian Nie, Weizhe Zheng and Xinwen Zhu for helpful communications. I should thank the referee for careful reading and useful suggestions.
This work was partially supported by the Chinese Academy of Sciences grants 50Y64198900, 29Y64200900, the Recruitment Program of Global Experts of China, and the NSFC grants No. 11631009 and No. 11688101.

\section{Affine Deligne-Lusztig varieties in mixed characteristic}
In this section, we recall some basic facts about affine Deligne-Lusztig varieties in mixed characteristic, which will be used later.

Fix a prime $p$. Let $G$ be a connected reductive group over $\Q_p$, which we assume to be unramified. Fix $T\subset B$ a maximal torus inside a Borel subgroup of $G$. Let $W=W(\ov{\F}_p)$ be the ring of Witt vectors, and $L=W_\Q$.  Denote $\sigma$ as the Frobenius on $L$ and $W$. In the following we want to fix a hyperspecial subgroup $G(W)\subset G(L)$. To this end,  we fix a reductive model $G_{\Z_p}$ of $G$ over $\Z_p$ and set $G(W)=G_{\Z_p}(W)$. Sometimes by abuse of notation we will also write $G$ as the reductive group $G_{\Z_p}$ over $\Z_p$.

\subsection{Affine Deligne-Lusztig varieties}
For $b\in G(L)$ and a conjugacy class $\{\mu\}$ of cocharacters $\mu: \G_m\ra G_{\ov{\Q}_p}$, we define the affine Deligne-Lusztig sets
\[X_\mu^G(b)=\{g\in G(L)/G(W)|\,g^{-1}b\sigma(g)\in G(W)\mu(p)G(W)\}, \]
and \[X_{\leq\mu}^G(b)=\{g\in G(L)/G(W)|\,g^{-1}b\sigma(g)\in \bigcup_{\mu'\leq\mu}{G(W)\mu'(p)G(W)}\}.\]Here, we assume $\mu\in X_\ast(T)_+$ for the above choice of $B$, and
for dominant elements $\mu,\mu'\in X_\ast(T)$, we say that $\mu'\leq \mu$ if $\mu-\mu'$ is a non-negative integral linear combination of positive coroots.
Let $J_b$ be the reductive group over $\Q_p$ such that for any $\Q_p-$algebra $R$, \[J_b(R)=\{g\in G(L\otimes_{\Q_p} R)|\, gb=b\sigma(g)\}.\] Then $J_b(\Q_p)$ acts naturally on $X_\mu^G(b)$ and $X_{\leq\mu}^G(b)$.
The isomorphism classes of  $X_\mu^G(b)$, $X_{\leq\mu}^G(b)$ and $J_b$ depend only on the $\sigma$-conjugacy class $[b]$ of $b$. 
By \cite{Win}, $X_\mu^G(b)$ and $X_{\leq\mu}^G(b)$  are non empty if and only if $[b]\in B(G, \mu)$. Here $B(G, \mu)$ is the Kottwitz subset (cf. \cite{Kot2} section 6) inside $B(G)$, the set of all $\sigma$-conjugacy classes in $G(L)$. We assume $[b]\in B(G, \mu)$ from now on. The triple $(G, [b], \{\mu\})$ will be called a local shtuka datum in the section 5, cf. Definition \ref{D:local shtuka datum}.
By construction, we have $X_{\mu}^G(b)\subset X_{\leq\mu}^G(b)$.  When $\{\mu\}$ is minuscule, we have $X_{\leq \mu}^G(b)=X_{\mu}^G(b)$.

By the recent work of Zhu \cite{Zhu} and Bhatt-Schoze \cite{BS}, there exist perfect scheme structures on the sets $X_\mu^G(b)$ and $X_{\leq\mu}^G(b)$. More precisely,  $X_\mu^G(b)$ and $X_{\leq\mu}^G(b)$ are the sets of $\ov{\F}_p$-points of some  perfect schemes over $\ov{\F}_p$, which are locally closed subschemes of the Witt vector affine Grassmannian $Gr_G$ (cf. \cite{Zhu, BS}). 
It will be useful to briefly recall the related moduli interpretation. Denote $\E_0$ the trivial $G$-torsor on $W$. For any perfect $\ov{\F}_p$-algebra $R$, we have (cf. \cite{Zhu}1.2 and 3.1)
$Gr_G(R)=\{(\E, \beta)\}/\simeq$, where
\begin{itemize}
\item $\E$ is a $G$-torsor over $W(R)$,
\item  $\beta: \E[1/p]\simeq \E_0[1/p]$ is a trivialization over $W(R)[1/p]$,
\end{itemize}
and
\[\begin{split}X_{\leq\mu}^G(b)(R)&=\{(\E, \beta )\in Gr_G(R)|\, \Inv_x(\beta^{-1}b\sigma(\beta))\leq \mu, \forall x\in \Spec R \},\\
X_{\mu}^G(b)(R)&=\{(\E, \beta )\in Gr_G(R)|\, \Inv_x(\beta^{-1}b\sigma(\beta))= \mu, \forall x\in \Spec R \},\\
\end{split} \]
where $\Inv_x$ is the relative position at $x$, and $\mu$ is the dominant representative in the conjugacy class $\{\mu\}$.
By abuse of notation, we denote also by $X_{\leq\mu}^G(b)$ and $X_\mu^G(b)$ the associated perfect schemes.  By construction, $X_{\mu}^G(b)\subset X_{\leq\mu}^G(b)$ is an open subscheme.

\begin{lemma}
Let $(G_1,[b_1],\{\mu_1\})\ra (G_2,[b_2],\{\mu_2\})$ be a morphism (cf. Definition \ref{D:local Shimura morphism}) It induces a natural map
\[X_{\leq \mu_1}^{G_1}(b_1)\ra X_{\leq \mu_2}^{G_2}(b_2).\] If $G_1\ra G_2$ is a closed immersion, the above map is a closed immersion.
\end{lemma}
\begin{proof} The first statement is clear. For the second statement,
	see \cite{Kim1} Lemma 2.5.4 (1) and \cite{HP} 2.4.4.
\end{proof}

\subsection{Connected components}\label{subsection:ADL}
In \cite{CKV} 2.3.5, Chen, Kisin and Viehmann introduced a notion of connected components for the affine Deligne-Lusztig sets $X_{\leq\mu}^G(b)$ by some ad hoc methods, since the algebro-geometric structure on $X_{\leq\mu}^G(b)$ had not been known by then. We denote by $\pi_0(X_{\leq\mu}^G(b)$ the set of connected components defined by Chen-Kisin-Viehmann in such a way. By resorting on the perfect scheme structure, we have a naturally defined notion of connected components for $X_{\leq\mu}^G(b)$. It is conjectured that the two definitions coincide, cf. \cite{Zhu} Remark 3.2 and \cite{CKV} 2.3.5. This was known in the case of unramified EL/PEL Rapoport-Zink spaces, cf. \cite{CKV} Theorem 5.1.5. Recently this has been proved by He-Zhou in the general case, cf. \cite{HZ} Theorem A.4.

Let $\pi_1(G)$ be the quotient\footnote{We note that $\pi_1(G)$ is finite if $G$ is semisimple.} of $X_\ast(T)$ by the coroot lattice of $G$. There is the Kottwitz homomorphism 
\[\omega_G: G(L)\lra \pi_1(G)\] for which an element $g\in G(W)\mu(p)G(W)\subset G(L)$ is sent to the class of $\mu$. Recall that for our pair $(b, \mu)$ we assume that $[b]\in B(G, \mu)$. Then there is an element $c_{b,\mu}\in \pi_1(G)$ such that $\omega_G(b)-\mu=(1-\sigma)(c_{b,\mu})$. The $\pi_1(G)^\Gamma$-coset of $c_{b,\mu}$ is uniquely determined. Here and the following, $\Gamma=\tr{Gal}(\ov{\Q}_p/\Q_p)$ is the local Galois group. In particular, if $b\in G(W)\mu(p)G(W)$ then we may take $c_{b,\mu}=1$.
As $\omega_G$ is trivial on $G(W)$, when restricting to  $X_{\leq \mu}^G(b)\subset G(L)/G(W)$,  by \cite{CKV} 2.3 (using the theory of Cartan decomposition in families of loc. cit. 2.1) we have a $J_b(\Q_p)$-equivariant morphism (of \'etale sheaves over $\ov{\F}_p$) 
\[\omega_G: X_{\leq \mu}^G(b)\lra c_{b,\mu}\pi_1(G)^\Gamma,\] which factors through $\pi_0(X_{\leq \mu}^G(b))$.  Thus we get a commutative diagram
\[\xymatrix{X_{\leq \mu}^G(b)\ar@{->>}[d] \ar[rd]^{\omega_G} & \\
	\pi_0(X_{\leq \mu}^G(b)) \ar[r] & c_{b,\mu}\pi_1(G)^\Gamma.
	}\]
Therefore, the non empty fibers of the map $\omega_G: X_{\leq \mu}^G(b) \ra c_{b,\mu}\pi_1(G)^\Gamma$ are unions of connected components of $X_{\leq \mu}^G(b)$.  Recall the following main theorem of \cite{CKV}.
\begin{theorem}[\cite{CKV} Theorems 1.2 and 1.1]\label{T:CKV}
	Assume that $\mu$ is minuscule.
	\begin{enumerate}
	\item $J_b(\Q_p)$ acts transitively on $\pi_0(X_\mu^G(b))$.
\item	Assume that $G^{ad}$ is simple, and $(\mu, b)$ is Hodge-Newton indecomposable in $G$. Then $\omega_G$ induces a bijection \[\pi_0(X_\mu^G(b))\simeq c_{b,\mu}\pi_1(G)^\Gamma\] unless $[b]=[\mu(p)]$ with $\mu$ central, in which case \[X_\mu^G(b)\simeq G(\Q_p)/G(\Z_p)\] is discrete.
\end{enumerate}
\end{theorem}
Recently Nie has obtained similar results as above on $\pi_0(X_{\leq\mu}^G(b))$ for general $\mu$ (not necessary minuscule). We refer the reader to \cite{Nie} Theorems 1.1 and 1.2 for the precise statements.

Assume that $\mu$ is minuscule. By (1) of the above theorem, all non empty fibers of $\omega_G: X_{\leq \mu}^G(b)\lra c_{b,\mu}\pi_1(G)^\Gamma$ are isomorphic to each other under the transition induced by the action of $J_b(\Q_p)$.  Fix a point $x_0\in \tr{Im}(\omega_G: X_{\mu}^G(b) \ra c_{b,\mu}\pi_1(G)^\Gamma )$ (soon we will show that $\omega_G$ is surjective). Let \[X_{\mu}^G(b)^+\subset X_{\mu}^G(b)\] be the fiber of $\omega_G$ over $x_0$. By (1) of the above theorem, we have the equality
\[X_{\mu}^G(b)=J_b(\Q_p)X_{\mu}^G(b)^+. \] In the following, we will not need to work on each connected component of $X_{\mu}^G(b)$. The subspace $X_{\mu}^G(b)^+$ and the equality above will be all what we need.

Now let $\mu$ be arbitrary. Since we assume $[b]\in B(G,\mu)$, the set $X_\mu^G(b)\neq \emptyset$. This means that there exists some $g\in G(L)$ such that $b':=g^{-1}b\sigma(g)\in G(W)\mu(p)G(W)$. Thus after replacing $b$ by $b'$, we may assume $c_{b,\mu}=1$.
(We note that the element $c_{b,\mu}$ can be defined for arbitrary $\mu$.)
\begin{lemma}\label{L:surjective}
\begin{enumerate}
	\item The restriction of $\omega_G: G(L)\ra \pi_1(G)$ to $G(\Q_p)$ induces a surjective map
\[\omega_G: G(\Q_p)\ra \pi_1(G)^\Gamma. \]
\item The map $J_b(\Q_p)\ra \pi_1(G)^\Gamma$ is surjective.
\end{enumerate}
\end{lemma}
\begin{proof}
For (1): this is contained in Lemma 1.2.3 of \cite{Ki2}. 

For (2): in the case that $(G, [b],  \{\mu\})$ comes from a Hodge type Shimura datum $(\mathbb{G}, X)$ unramified at $p$ (and  $Z_{\mathbb{G}}$ is a torus), see Lemma 4.6.4 of \cite{Ki2}. The arguments there work also in the general case.
\end{proof}

\begin{proposition}\label{P: surjective}
The map \[\omega_G: X_{\leq\mu}^G(b)\lra c_{b,\mu}\pi_1(G)^\Gamma\] is surjective. In particular we get a surjection 
\[\pi_0(X_{\leq\mu}^G(b))\twoheadrightarrow \pi_1(G)^\Gamma.\]
\end{proposition}
\begin{proof}
By Lemma 2.3.6 of \cite{CKV}, the map $\omega_G$ is compatible with the $J_b(\Q_p)$-actions on both sides. By construction, $J_b(\Q_p)$ acts on $\pi_1(G)^\Gamma$ by left multiplication via the map $J_b(\Q_p)\ra \pi_1(G)^\Gamma$, which is surjective by (2) of Lemma \ref{L:surjective}.
 Thus $\omega_G: X_{\leq\mu}^G(b)\ra c_{b,\mu}\pi_1(G)^\Gamma$ is surjective.
\end{proof}

We continue to assume that $\mu$ can be arbitrary. For a reductive group $G$ over $\Q_p$, we write $Z_G$ as its center.
\begin{proposition}\label{P:cartesian}
Let $(G_1,[b_1],\{\mu_1\})\ra (G_2,[b_2],\{\mu_2\})$ be a morphism.
If $G_2=G_1/Z$ for some central group $Z\subset Z_{G_1}$, we have the following cartesian diagram
\[\xymatrix{X_{\leq \mu_1}^{G_1}(b_1)\ar[r]\ar[d]^{\omega_{G_1}}& X_{\leq \mu_2}^{G_2}(b_2) \ar[d]^{\omega_{G_2}}\\
	c_{b_1,\mu_1}\pi_1(G_1)^\Gamma\ar[r] & c_{b_2,\mu_2}\pi_1(G_2)^\Gamma.
}\]
\end{proposition}
\begin{proof}
This is contained in \cite{CKV} Corollary 2.4.2.
\end{proof}

Let the notations be as in the above proposition. Combined with Proposition \ref{P: surjective}, we get
\begin{corollary}\label{C:local ADLV}
Let $x_1\in c_{b_1,\mu_1}\pi_1(G_1)^\Gamma$ be a point and $x_2\in c_{b_2,\mu_2}\pi_1(G_2)^\Gamma$ be its image under $c_{b_1,\mu_1}\pi_1(G_1)^\Gamma\ra c_{b_2,\mu_2}\pi_1(G_2)^\Gamma$. Let $X_{\leq \mu_1}^{G_1}(b_1)^+$ and $X_{\leq \mu_2}^{G_2}(b_2)^+$ be the fibers of $\omega_{G_1}$ and $\omega_{G_2}$ at $x_1$ and $x_2$ respectively, which are non empty by Proposition \ref{P: surjective}. Then the map $X_{\leq \mu_1}^{G_1}(b_1)\ra X_{\leq \mu_2}^{G_2}(b_2)$ induces a bijection 
\[X_{\leq \mu_1}^{G_1}(b_1)^+ \st{\sim}{\lra}X_{\leq \mu_2}^{G_2}(b_2)^+.\]
\end{corollary}
We still keep the above notations.
\begin{lemma}
	If $\pi_1(G_1)^\Gamma\ra \pi_1(G_2)^\Gamma$ is surjective, then the map $X_{\leq \mu_1}^{G_1}(b_1)\ra X_{\leq \mu_2}^{G_2}(b_2)$ induces an isomorphism \[X_{\leq \mu_1}^{G_1}(b_1)/Z(\Q_p)\simeq X_{\leq \mu_2}^{G_2}(b_2).\]
\end{lemma}
\begin{proof}
	This is implied by the proof of \cite{CKV} Corollaries 2.4.2 and 2.4.3: under the assumption that $\pi_1(G_1)^\Gamma\ra \pi_1(G_2)^\Gamma$ is surjective, all fibers of $X_{\leq \mu_1}^{G_1}(b_1)\ra X_{\leq \mu_2}^{G_2}(b_2)$ are torsors under $X_\ast(Z)^\Gamma$. The group $Z(\Q_p)$ acts on $X_{\leq \mu_1}^{G_1}(b_1)$ via the natural map $Z(\Q_p)\ra X_\ast(Z)^\Gamma$.
\end{proof}

\section{Rapoport-Zink spaces of Hodge type}
Following Rapoport-Viehmann, we first review the general conjecture on the theory of local Shimura varieties in \cite{RV}. Then we concentrate on the Hodge type case, cf. \cite{Kim1, HP,BP}.

\subsection{Local Shimura data and local Shimura varieties}
Recall the following definition of Rapoport-Viehmann.
\begin{definition}[\cite{RV} Definition 5.1]\label{D:RV}
A local Shimura datum over $\Q_p$ is a triple $(G, [b],  \{\mu\})$ where
\begin{itemize}
	\item $G$ is a connected reductive group over $\Q_p$,
	\item $[b]\in B(G)$ is a $\sigma$-conjugacy class,
	\item $\{\mu\}$ is a conjugacy class of cocharacters $\mu: \G_m\ra G_{\ov{\Q}_p}$,
\end{itemize}
such that the following conditions are satisfied
\begin{enumerate}
	\item $[b]\in B(G,\mu)$,
	\item $\{\mu\}$ is minuscule.
\end{enumerate}
\end{definition} 
Associated to a local Shimura datum, we have
\begin{itemize}
	\item the reflex field $E=E(G,\{\mu\})$, which is the field of definition of $\{\mu\}$ inside the fixed algebraic closure $\ov{\Q}_p$,
	\item the flag variety $\Fl\ell_{G, \mu}$, considered as a rigid analytic space (or an adic space) over $\breve{E}$, the completion of the maximal unramified extension of $E$; here the associated parabolic subgroup $P_\mu$ is given by $P_\mu=\{g\in G|\, \lim_{t\ra 0}\mu(t)g\mu(t)^{-1} \,\tr{exists}\}$, 
	\item the reductive group $J_b$ over $\Q_p$, for $b\in [b]$, which up to isomorphism only depends on $[b]$. The group $J_b(\Q_p)$ acts on $\Fl\ell_{G, \mu}$,
	\item  the weakly admissible open subspace $\Fl\ell_{G, \mu}^{wa}\subset \Fl\ell_{G, \mu}$ defined in \cite{RZ} 1.35 and \cite{DOR} Definition 9.5.4. The action of $J_b(\Q_p)$ on $\Fl\ell_{G, \mu}$ stabilizes $\Fl\ell_{G, \mu}^{wa}$.
\end{itemize}
In fact, if $G$ is unramified, we have also (cf. the last section)
\begin{itemize}
	\item the affine Deligne-Lusztig variety $X_\mu^G(b)$ over $\ov{\F}_p$ (which will be expected to be the special fiber of some formal model of the following local Shimura variety $\M_{G(\Z_p)}$, cf. Conjecture \ref{C:RV}).
\end{itemize}

Let $(G, [b],  \{\mu\})$ be a local Shimura datum, with local reflex field $E$. We have the following conjecture (\cite{RV} 5.1):
\begin{conjecture}[Rapoport-Viehmann]\label{C:RV}
	There is a tower of rigid analytic spaces over $\tr{Sp}\breve{E}$,
	\[(\M_K)_{K},\] where $K$ runs through all open compact subgroups of $G(\Q_p)$, with the following properties:
	\begin{enumerate}
		\item the group $J_b(\Q_p)$ acts on each space $\M_K$,
		\item the group $G(\Q_p)$ acts on the tower $(\M_K)_{K}$ as Hecke correspondences,
		\item the tower is equipped with a Weil descent datum over $E$,
		\item there exists a compatible system of \'etale and partially proper period maps \[\pi_K: \M_K\ra \Fl\ell_{G, \mu}^{wa}\] which is equivariant for the action of $J_b(\Q_p)$.
		\end{enumerate}
\end{conjecture}
In fact, in \cite{RV} 5.1 there is a more precise statement on the point (4) of the conjecture. In particular, there should be an open subspace \[\Fl\ell_{G, \mu}^{a}\subset \Fl\ell_{G, \mu}^{wa},\] which should be the image of the period maps $\pi_K$ for all $K$. In fact, Rapoport and Zink conjecture that there exists a $\Q_p$-local system with $G$-structure over $\Fl\ell_{G, \mu}^{a}$ which interpolates the $p$-adic crystalline Galois representations attached to all classical points. Moreover, the tower $(\M_K)_{K\subset G(\Q_p)}$ should be the geometric realization (i.e. spaces of lattices with level structures) of this universal $\Q_p$-local system with $G$-structure over $\Fl\ell_{G, \mu}^{a}$. We refer to \cite{RV} 5.1, \cite{RZ} section 1, \cite{DOR} Conjecture 11.4.4, and \cite{Har} Conjecture 2.3 for more details. This conjecture has been known for the local Shimura data which arise from local EL/PEL data (\cite{RZ}), and the unramified local Shimura datum of Hodge type (\cite{Kim1}). In both cases, these spaces $\M_K$ are finite \'etale covers of the rigid analytic generic fibers of some formal schemes $\breve{\M}$ over $\Spf \Ol_{\breve{E}}$, which are formal moduli spaces of $p$-divisible groups with some additional structures. The special fibers of these formal schemes $\breve{\M}$ are the affine Deligne-Lusztig varieties which we introduced in the last section. In section \ref{section:shtuka} we will discuss a partial solution of the above conjecture due to Scholze, cf. \cite{Sch2,SW17}.

It will be useful to make a definition for morphisms of local Shimura data.
\begin{definition}\label{D:local Shimura morphism}
Let $(G_1, [b_1],\{\mu_1\}), (G_2, [b_2],\{\mu_2\})$ be two local Shimura data. A morphism \[(G_1, [b_1],\{\mu_1\})\ra (G_2, [b_2],\{\mu_2\})\] is a homomorphism of algebraic groups $f: G_1\ra G_2$ sending $([b_1], \{\mu_1\})$ to $([b_2], \{\mu_2\})$.
\end{definition}
If $(G_1, [b_1],\{\mu_1\})\ra (G_2, [b_2],\{\mu_2\})$ is a morphism of local Shimura data, then it is conjectured (\cite{RV} Properties 5.3 (iv)) that for any open compact subgroups $K_1\subset G_1(\Q_p), K_2\subset G_2(\Q_p)$ with $f(K_1)\subset K_2$, there exists a morphism of the associated local Shimura varieties
\[\M(G_1,b,\mu)_{K_1}\lra \M(G_2,b_2,\mu_2)_{K_2}\times \tr{Sp}\breve{E}_1,\]and when $G_1\ra G_2$ is a closed immersion these are closed embeddings for $K_1=K_2\cap G_1(\Q_p)$. 

\subsection{Local Shimura data of Hodge type}
Now we recall the definition of a special class of local Shimura data (cf. \cite{RV} Remark 5.4 (i)):
\begin{definition}\label{D:local Hodge}
	A local Shimura datum $(G, [b],  \{\mu\})$ is called of Hodge type, if there exists an embedding $f: G\hookrightarrow \GL(V)$ and a local Shimura datum $(\GL(V), [b'],  \{\mu'\})$ with $\{\mu'\}$ corresponds to $(1^r, 0^{n-r})$ for some integer $1\leq r\leq n=\dim V$, such that $[b], \{\mu\}$ are mapped to $[b'], \{\mu'\}$ under $f$.
\end{definition}
If $G$ is moreover unramified, by \cite{Ki1} Lemma 2.3.1, we can find some $\Z_p$-lattice $V_{\Z_p}\subset V$ such that $G\hookrightarrow \GL(V)$ is induced by an embedding $G_{\Z_p}\hookrightarrow \GL(V_{\Z_p})$, where $G_{\Z_p}$ is a reductive model of $G$ over $\Z_p$.
\begin{definition}\label{D:local Hodge unramified}
A local Shimura datum of Hodge type $(G, [b],  \{\mu\})$ is called unramified, if $G$ is unramified.
\end{definition}
We note that for an unramified local Shimura datum of Hodge type $(G, [b],  \{\mu\})$, the local reflex field $E$ is an unramified extension of $\Q_p$. Thus $\breve{E}=L, \Ol_{\breve{E}}=\Ol_L=W$ where as before $W=W(\ov{\F}_p), L=W_\Q$. We will fix a reductive model $G_{\Z_p}$ of $G$ over $\Z_p$.
\begin{remark}
The above definition of unramified local Shimura data of Hodge type is more general than that in \cite{HP} Definition 2.3.3. Moreover, for an unramified local Shimura datum of Hodge type $(G, [b],  \{\mu\}) $ in the sense of \cite{HP}, one always has $Z_G\supset \G_m$. 
\end{remark}

We want to classify local Shimura data of Hodge type. Let $(G, [b],  \{\mu\})$ be a given local Shimura datum. Take any faithful representation $V$ of $G$ over $\Q_p$, so that we get an embedding $\rho: G\hookrightarrow \GL(V)$. Therefore we get a conjugacy class $\{\mu'\}$ of cocharacters $\mu'=\rho_{\ov{\Q}_p}\circ \mu: \G_m\ra \GL(V)_{\ov{\Q}_p}$. Let $N(G)$ be the set of Newton points of $G$, cf. \cite{RR} 1.7 (in \cite{Kot2} 4.1, it was denoted by $\ov{C}_\Q$). Recall that the maps \[\nu_G: B(G)\ra N(G),\quad \kappa_G: B(G)\ra \pi_1(G)_\Gamma\]  are functorial in $G$, cf. \cite{Kot1} section 4, \cite{Kot2} 4.9 and 7.5, and \cite{RR} 1.9 and 1.15. We get in particular a map \[B(G, \mu)\ra B(\GL(V), \mu').\]
Let $[b']\in B(\GL(V), \mu')$ be the image of $[b]$ under this map. The triple $(\GL(V), [b'],\{\mu'\})$ is a local Shimura datum if and only if $\{\mu'\}$ is minuscule and corresponds to $(1^r, 0^{n-r})$ for some integer $1\leq r\leq n=\dim V$.  In which case $(G, [b],  \{\mu\})$ is of Hodge type.  As above, let $G$ be a reductive group over $\Q_p$ and $\{\mu\}$ be a conjugacy class of minuscule cocharacters $\mu: \G_m\ra G_{\ov{\Q}_p}$.
In \cite{Ser} Serre classified the pair $(G,\{\mu\})$ for which there exists a faithful representation $V$ of $G$ such that the induced class of cocharacters $\{\mu'\}$ under the embedding $G\hookrightarrow \GL(V)$ corresponds to $(1^r, 0^{n-r})$ for some integer $1\leq r\leq n=\dim V$. It turn out the simple factors of $G^{ad}_{\ov{\Q}_p}$ are groups of type $A, B, C$ or $D$, cf. \cite{Ser} section 3.

The following examples of local Shimura datum of Hodge type are standard.
\begin{example}
	\begin{enumerate}
		\item Let $(G, [b],  \{\mu\})$ be a local Shimura datum which comes from a local EL/PEL datum (cf. \cite{RZ} 1.38), then it is of Hodge type (cf. \cite{Kim1} 4.7).
		\item Let $(G,X)$ be a Shimura datum of Hodge type, i.e. there exists some embedding into the Siegel Shimura datum $(G,X)\hookrightarrow (\GSp, S^{\pm})$. Let $\mu$ be the cocharacter associated to $X$. Take any $[b]\in B(G_{\Q_p},\mu)$. Then the local Shimura datum $(G_{\Q_p}, [b],  \{\mu\})$ is of Hodge type.
	\end{enumerate}
\end{example}

Here is an example of non Hodge type local Shimura datum.
\begin{example}[See \cite{RV} Example 5.5]\label{E:PGLn}
	Let $G=\PGL_n$, $\mu$ be any non trivial minuscule cocharacter, and $[b]\in B(G,\mu)$ be arbitrary. Then the local Shimura datum $(G, [b],  \{\mu\})$ is not of Hodge type. 
\end{example}

\subsection{Rapoport-Zink spaces of Hodge type}
Throughout the rest of this section, we assume that $p>2$.
Let $(G, [b],  \{\mu\})$ be an unramified local Shimura datum of Hodge type. Fix a representative $b\in G(L)$ of $[b]$.  Kim (\cite{Kim1}) constructs a formal moduli scheme $\breve{\M}=\breve{\M}(G,b,\mu)$ over $\Spf W$ parametrizing $p$-divisible groups with crystalline Tate tensors.  We briefly review the related constructions in this subsection. By abuse of notation, we write also $G$ as the fixed associated reductive group scheme over $\Z_p$. Then there exists a faithful representation \[\rho: G\hookrightarrow \GL(\Lambda),\] such that the induced cocharacter $\mu'=\rho_{\ov{\Q}_p}\circ\mu: \G_m\ra \GL(\Lambda\otimes\ov{\Q}_p)$ is minuscule. Let $\Lambda^\vee$ be the dual lattice, and $\Lambda^\otimes$ be the tensor algebra of $\Lambda\oplus \Lambda^\vee$. By Proposition 1.3.2 of \cite{Ki1}, there exists a finite collection of tensors $\{s_\alpha\in \Lambda^\otimes\}_{\alpha\in I}$ such that $\rho: G\subset \GL(\Lambda)$ is the schematic stabilizer of $(s_\alpha)$.  We fix a representative $\mu$. Let $\Lambda\otimes W=\Lambda^0\oplus \Lambda^1$ be the decomposition of $\Lambda\otimes W$ according to the weights of $\mu$, which in turn induces a filtration $\Fil^\bullet\Lambda\otimes W$ with $\Fil^0\Lambda\otimes W=\Lambda\otimes W, \Fil^1\Lambda\otimes W=\Lambda^1$. We assume that $rank\; \Lambda=n, rank\;\Lambda^1=d$. We note that $P_\mu:=\tr{Aut}(\Lambda, {s_\alpha}, \Fil^\bullet\Lambda\otimes W)$ is a parabolic subgroup of $G_W$.

By our assumption and the classical Dieudonn\'e theory, there exists a $p$-divisible group $X_0$ of dimension $d$ and height $n$ over $\ov{\F}_p$, together with an isomorphism $\varepsilon: \D(X_0)\simeq (\Lambda\otimes W, b\sigma)$, where $\D(X_0)$ is the contravariant Dieudonn\'e module of $X_0$. The pair $(X_0, \varepsilon)$ is unique up to a unique isomorphism and we fix it in the sequel. Then we can regard $s_\alpha\otimes 1$ as tensors in $\D(X_0)^\otimes$ via $\varepsilon$. We note that $b\sigma$ fixes $(s_\alpha\otimes 1)$ and $(s_\alpha\otimes 1)$ lie in $\Fil^1\Lambda\otimes W$. Each $s_\alpha\otimes 1$ can be regarded as a map $1:=\D(\Q_p/\Z_p)\ra \D(X_0)^\otimes$, compatible with the filtrations, and such that the induced map $1\ra \D(X_0)^\otimes[\frac{1}{p}]$ is Frobenius-invariant, i.e. $s_\alpha\otimes 1$ is a crystalline Tate tensor of $X_0$, cf. \cite{Kim1} Definition 4.6.

Let $\Nilp_W$ be the category of $W$-algebras on which $p$ is locally nilpotent. Let $R\in \Nilp_W$ and $X$ be a $p$-divisible group on $\Spec R$. Consider the contravariant Dieudonn\'e crystal $\D(X)$ attached to $X$. Then as usual there is a decreasing (Hodge) filtration $\Fil^\bullet \D(X)_R$ on $\D(X)_R$ with locally free graded pieces over $R$. Here $\D(X)_R$ is the value of $\D(X)$ at the trivial PD-thickening $R\st{id}{\ra}R$. Namely, $\Fil^0\D(X)_R=\D(X)_R, \Fil^1\D(X)_R=(\Lie X)^\vee$ and $\Fil^2\D(X)_R=0$. As above, a crystalline Tate tensor of $X$ is a morphism $t_\alpha: 1\ra \D(X)^\otimes$ of crystals, such that $t_{\alpha R}: 1_R\ra \D(X)^\otimes_R$ is compatible with the filtrations, and the induced map $t_\alpha: 1\ra \D(X)^\otimes[\frac{1}{p}]$ is Frobenius-invariant.

Denote by $\Nilp_W^{sm}$ the full subcategory of $\Nilp_W$ consisting of formally smooth formally finitely generated $W/p^m$-algebras for $m\geq 1$. We use the following version of Rapoport-Zink functor, cf. \cite{Zhu} Definition 3.8, which is equivalent to Definition 4.6 of \cite{Kim1}.
\begin{definition}
	The Rapoport-Zink space associated to the unramified local Shimura datum of Hodge type is the functor $\breve{\M}$ on $\Nilp_W^{sm}$ defined by $\breve{\M}(R)=\{(X,(t_\alpha)_{\alpha\in I}, \rho)\}/\simeq$ where
	\begin{itemize}
		\item $X$ is a $p$-divisible group on $\Spec R$,
		\item $(t_\alpha)_{\alpha\in I}$ is a collection of crystalline Tate tensors of $X$,
		\item $\rho: X_0\otimes R/J\ra X\otimes R/J$ is a quasi-isogeny which sends $s_\alpha\otimes 1$ to $t_\alpha$ for $\alpha\in I$, where $J$ is some ideal of definition of $R$,
	\end{itemize}
	such that the following condition holds:\\
	the $R$-scheme 
	\[ \Isom\Big(\big(\D(X)_R, (t_\alpha), \Fil^\bullet(\D(X)_R) \big), \big(\Lambda\otimes R, (s_\alpha\otimes1), \Fil^\bullet\Lambda\otimes R\big)\Big) \]
	that classifies the isomorphisms between locally free sheaves $\D(X)_R$ and $\Lambda\otimes R$ on $\Spec R$ preserving the tensors and the filtrations is a $P_\mu\otimes R$-torsor.
\end{definition}

\begin{theorem}[\cite{Kim1} Theorem 4.9.1]\label{T:hodge}
 The functor $\breve{\M}$ is represented by a separated formal scheme, formally smooth and formally locally of finite type over $W$.
\end{theorem}
 In the classical EL/PEL case ( and with ramification), see \cite{RZ}  Theorem 3.25.  In \cite{Kim1} 4.7, the unramified local EL/PEL data are explained as special examples of unramified Hodge type data. See also \cite{HP} Theorem 3.2.1 for the case that $(G, [b],  \{\mu\})$ comes from a Shimura datum of Hodge type. If $\rho(b)$ has no slope 0, B\"ultel and Pappas have proved the above theorem by a different approach, see \cite{BP}. More precisely, they introduced notions of $(G,\mu)$-displays and quasi-isogenies between such, and they proved that the similar moduli problem of $(G,\mu)$-displays together with quasi-isogenies are representable. In the case $G=\GL_n$, the moduli problem of B\"ultel-Pappas is equivalent to the moduli problem of Rapoport-Zink, by the theorems of Zink (\cite{Zin}) and Lau (\cite{Lau}) that formal $p$-divisible groups over a $p$-adically complete and separated algebra $R$ are classified by the associated nilpotent displays.
 
 We denote also by $\breve{\M}$ the associated formal scheme, and refer it as the formal Rapoport-Zink space of Hodge type attached to $(G, [b],  \{\mu\})$. Let $\M$ be the rigid analytic generic fiber over $L=W_\Q$ of the formal scheme $\breve{\M}$. In the rest of this paper, we will use the following convention: if $G$ is an unramified reductive group over $\Q_p$, we will fix a reductive model over $\Z_p$ and write $G(\Z_p)$ for the associated hyperspecial group.
In \cite{Kim1} 7.4, Kim explained how to construct a tower of rigid analytic spaces \[(\M_K)_{K\subset G(\Z_p)}\] that satisfies the list of properties in Conjecture \ref{C:RV}. Moreover, $\M_{G(\Z_p)}=\M$, and $\M_K\ra \M$ is finite \'etale for any open compact subgroup $K\subset G(\Z_p)$. In particular, for unramified local Shimura data of Hodge type, the Conjecture \ref{C:RV} is true.

Let $\ov{\M}$ be the special fiber over $\ov{\F}_p$ of $\breve{\M}$. Recall that in section 2 attached to $(G,[b], \{\mu\})$, we introduced the affine Deligne-Lusztig variety $X_\mu^G(b)$ over $\ov{\F}_p$, viewed as a perfect scheme. The relation between $\breve{\M}$ and $X_\mu^G(b)$ is as follows.
\begin{proposition}[\cite{Zhu} Proposition 3.11]\label{P:Zhu}
	$X_\mu^G(b)$ is the perfection $\ov{\M}^{perf}$ of $\ov{\M}$. 
\end{proposition}

If $(G, [b],  \{\mu\})\hookrightarrow (\GL_n, [b'],  \{\mu'\})$ is an embedding of unramified local Shimura data of Hodge type, by construction, we have the following embeddings
\[\breve{\M}(G,b,\mu)\hookrightarrow \breve{\M}(\GL_n,b',\mu'),\quad X_\mu^G(b)\hookrightarrow X_{\mu'}^{\GL_n}(b'), \]which are compatible in the sense of the above proposition.

\subsection{Connected components} Let the notations be as above. Recall in subsection 2.2 we have the map
 \[\omega_G: X_{\mu}^G(b)\lra c_{b,\mu}\pi_1(G)^\Gamma.\] By Proposition \ref{P:Zhu} we get an induced map of \'etale sheaves over $W$:
\[\omega_G: \breve{\M}\lra c_{b,\mu}\pi_1(G)^\Gamma.\]Let $G^{der}\subset G$ be the derived subgroup, and $G^{ab}$ the abelian quotient $G/G^{der}$. Consider the exact sequence
\[1\ra G^{der}\ra G\ra G^{ab}\ra 1,\]which induces a map
\[c_{b,\mu}\pi_1(G)^\Gamma\ra c_{b,\mu}\pi_1(G^{ab})^\Gamma=c_{b,\mu}X_\ast(G^{ab})^\Gamma,\]
where $X_\ast(G^{ab})$ is the cocharacter group of the torus $G^{ab}$ over $\ov{\Q}_p$. Let $X^\ast_{\Q_p}(G)$ be the group of $\Q_p$-rational characters of $G$. Then we have \[X^\ast_{\Q_p}(G)=X^\ast(G^{ab})^\Gamma.\] The $\Gamma$-equivariant pairing $X_\ast(G^{ab})\times X^\ast(G^{ab})\ra \Z$ then induces a map
\[c_{b,\mu}X_\ast(G^{ab})^\Gamma\ra \Hom(X^\ast(G^{ab})^\Gamma,\Z)=\Hom(X^\ast_{\Q_p}(G),\Z). \]In summary, we get a map by considering the composition
\[\varkappa_{\breve{\M}}: \breve{\M}\ra c_{b,\mu}\pi_1(G)^\Gamma\ra c_{b,\mu}X_\ast(G^{ab})^\Gamma\ra \Hom(X^\ast_{\Q_p}(G),\Z). \]
In the EL/PEL case, this is just the map constructed in \cite{RZ} 3.52. (See also \cite{CKV} 5.1.3.)

If $(G, [b],  \{\mu\})\hookrightarrow (\GL_n, [b'],  \{\mu'\})$ is an embedding of unramified local Shimura data of Hodge type, we get the following commutative diagram
 \[\xymatrix{X_\mu^G(b)\ar[r]\ar[d]& X_{\mu'}^{\GL_n}(b')\ar[d]\\
 	c_{b,\mu}\pi_1(G)^\Gamma \ar[r] & c_{b',\mu'}\pi_1(\GL_n)^\Gamma.
 	}\]
Moreover, we know $\pi_1(\GL_n)^\Gamma=\pi_1(\GL_n)\simeq \Z$.

Since by Proposition \ref{P:Zhu} $X_\mu^G(b)$ is the perfection $\ov{\M}^{perf}$ of $\ov{\M}$,
we have the isomorphism between the sets of connected components \[\pi_0(\M_{red})\simeq \pi_0^{perf}(X_\mu^G(b)).\] Here $\pi_0^{perf}(X_\mu^G(b))$ denotes the set of connected components of the perfect scheme $X_\mu^G(b)$. On the other hand, we have also the set of connected components $\pi_0(X_\mu^G(b))$ defined in \cite{CKV}. 
\begin{proposition}
With the above notations, there is a bijection
	\[\pi_0(\M_{red})\simeq  \pi_0(X_\mu^G(b)).\]
\end{proposition}
\begin{proof}
	See the Remark 3.2 of \cite{Zhu}. See also \cite{HZ} Theorem A.4.
\end{proof}

Let $\pi_0(\breve{\M})$ be the set of connected components of the formal scheme $\breve{\M}$, which is the same as $\pi_0(\M_{red})$. On the other hand, we have also the set of connected components $\pi_0(\M)$ of the generic fiber $\M$. As $\breve{\M}$ is formally smooth and in particular normal, by \cite{dJ2} Theorem 7.4.1, we have a bijection  \[\pi_0(\M_{red})\simeq  \pi_0(\M).\]

One can also consider the set of connected components $\pi_0(\M_K)$ for the finite \'etale cover $\M_K$ of $\M$. In \cite{RV}, Rapoport and Viehmann made a conjecture on $\pi_0(\M_K\times \C_p)$ under the assumption that $G^{der}$ is \textsl{simply connected}. We refer to \cite{RV} Conjecture 4.26 for the precise statement on the existence of a determinant morphism for the tower $(\M_K)_K$. This conjecture is known in the unramified simple EL/PEL case, cf. Theorem 6.3.1 of \cite{C} (see also \cite{CKV} Theorem 5.1.10 and Remark 5.1.11). It will be interesting to consider the more general Hodge type case studied here.

Fix a point $x_0\in c_{b,\mu}\pi_1(G)^\Gamma$. Let $\M_{red}^+\subset\M_{red}$ be the fiber of $\omega_G$ over $x_0$. Then $\M_{red}^+$ is some union of connected components of $\M_{red}$. Let $\breve{\M}^+\subset \breve{\M}$ be the associated  sub formal scheme, with generic fiber $\M^+$. For any open compact subgroup $K\subset G(\Q_p)$, let $\M_K^+\subset \M_K$ be the pullback of $\M^+\subset\M$. We get a tower \[(\M_K^+)_{K\subset G(\Z_p)}.\] We have the equalities
\[\breve{\M}=J_b(\Q_p)\breve{\M}^+,\quad \M_{red}=J_b(\Q_p)\M_{red}^+, \quad \M= J_b(\Q_p)\M^+\]and
\[\M_K=J_b(\Q_p)\M_K^+. \]

\section{Rapoport-Zink spaces of abelian type}
We enlarge the class of Rapoport-Zink spaces of Hodge type in this section. They are constructed locally from Rapoport-Zink spaces of Hodge type. Throughout this section we assume $p>2$.

\subsection{Local Shimura data of abelian type}
Let $(G, [b],  \{\mu\})$ be a local Shimura datum. Consider the natural projection $G\ra G^{ad}$ from $G$ to its associated adjoint group. We get induced $[b^{ad}], \{\mu^{ad}\}$, so that \[(G^{ad}, [b^{ad}], \{\mu^{ad}\})\] is also a local Shimura datum and $(G, [b],  \{\mu\})\ra (G^{ad}, [b^{ad}], \{\mu^{ad}\})$ is a morphism of local Shimura data.
We introduce the local analogue of a Shimura datum of abelian type (more precisely, of preabelian type) as follows. 
\begin{definition}
A local Shimura datum $(G, [b],  \{\mu\})$  is called  of abelian type, if there exists a local Shimura datum of Hodge type $(G_1, [b_1],  \{\mu_1\})$ such that we have an isomorphism of the associated adjoint local Shimura data $(G^{ad}, [b^{ad}],  \{\mu^{ad}\})\simeq (G_1^{ad}, [b_1^{ad}],  \{\mu_1^{ad}\})$.
\end{definition}

Thus any local Shimura datum of Hodge type is also of abelian type. The later class is strictly larger.
\begin{example}\label{E:PGL}
Let $G=\PGL_n$. Consider a nontrivial minuscule cocharacter $\mu_1: \G_m\ra \GL_n$ and $[b_1]\in B(\GL_n,\mu_1)$. Take $\mu=\mu_1^{ad}, [b]=[b_1^{ad}]$. Then $(G, [b], \{\mu\})$ is of abelian type, but not of Hodge type, cf. Example \ref{E:PGLn}.
\end{example}

Recall that for a local Shimura datum $(G, [b],  \{\mu\})$, if $G_i$ is a simple local factor of $G^{ad}_{\ov{\Q}_p}$ such that the component $\mu_i^{ad}$ of $\mu^{ad}$ is not trivial, then  $G_i$ is a group of one of types $A, B, C, D, E_6, E_7$, cf. \cite{Ser} Annexe.
By Serre's classification (\cite{Ser} section 3) and our definition, simple factors of $G$ appearing in local Shimura data of abelian type consists exactly of local reductive groups of types $A, B, C, D$. This is compatible with Deligne's classification of Shimura data of abelian type in \cite{D}, cf. Example \ref{E:global}.

\subsection{The associated Rapoport-Zink spaces}\label{Subsection:RZab}
To construct Rapoport-Zink spaces, we need the following unramified assumption.
\begin{definition}\label{D:local abelian unramified}
A local Shimura datum of abelian type $(G, [b],  \{\mu\})$ is called unramified, if $G$ is unramified, and there exists an unramified local Shimura datum of Hodge type $(G_1,[ b_1],\{\mu_1\})$ such that $(G^{ad}, [b^{ad}],  \{\mu^{ad}\})\simeq (G_1^{ad}, [b_1^{ad}],  \{\mu_1^{ad}\})$.
\end{definition}
For an unramified local Shimura datum of abelian type, the local reflex field $E$ is an unramified extension of $\Q_p$. Thus $\breve{E}=L, \Ol_{\breve{E}}=\Ol_L=W$ where as before $W=W(\ov{\F}_p), L=W_\Q$.

The following example is one of our main motivations.
\begin{example}\label{E:global}
Let $(G,X)$ be a Shimura datum of abelian type such that $G$ is unramified at $p$ (cf. \cite{D, Ki1}). Take any $[b]\in B(G,\mu)$, the associated triple $(G_{\Q_p}, [b],  \{\mu\})$ is an unramified local Shimura datum of abelian type.
\end{example}

\begin{lemma}
Let $(G, [b], \{\mu\})$ be an unramified local Shimura datum of abelian type. Consider the associated adjoint local Shimura datum $(G^{ad}, [b^{ad}],  \{\mu^{ad}\})$. Fix a representative $b\in G(L)$ of $[b]$ with image $b^{ad}\in G^{ad}(L)$, and identify $Z_G$ as a central subgroup of $J_b$. We have the following isomorphism of reductive groups over $\Q_p$
\[J_b/Z_G\simeq J_{b^{ad}}.\]
\end{lemma}
\begin{proof}
This follows from the definitions of $J_b$ and $J_{b^{ad}}$.
\end{proof}

\begin{theorem}\label{T:ab}
Let $(G, [b], \{\mu\})$ be an unramified local Shimura datum of abelian type. Fix a representative $b\in G(L)$ of $[b]$. Then there exists a formal scheme $\breve{\M}(G,b,\mu)$, which is formally smooth, formally locally of finite type over $W$, such that \[\ov{\M}(G,b,\mu)^{perf}\simeq X_{\mu}^{G}(b),\] where $\ov{\M}(G,b,\mu)$ is the special fiber of  $\breve{\M}(G,b,\mu)$.
The formal scheme $\breve{\M}(G,b,\mu)$ is equipped with a transitive action of $J_b(\Q_p)$, compatible with the action of $J_b(\Q_p)$ on $X_{\mu}^{G}(b)$.
\end{theorem}

\begin{proof}
Take any unramified local Shimura datum of Hodge type $(G_1, [b_1], \{\mu_1\})$ as in Definition \ref{D:local abelian unramified} and fix a representative $b_1\in [b_1]$.
Consider the associated formal Rapoport-Zink space $\breve{\M}(G_1,b_1,\mu_1)$ over $\Spf W$. Then its special fiber $\ov{\M}(G_1,b_1,\mu_1)$ satisfies \[\ov{\M}(G_1,b_1,\mu_1)^{perf}\simeq X_{\mu_1}^{G_1}(b_1).\]
Recall that we have following cartesian diagram (cf. Proposition \ref{P:cartesian})
\[\xymatrix{X_{\mu_1}^{G_1}(b_1)\ar[r]\ar[d]^{\omega_{G_1}} & X_{\mu_1^{ad}}^{G_1^{ad}}(b_1^{ad})\ar[d]^{\omega_{G_1^{ad}}}\\
		c_{b_1,\mu_1}\pi_1(G_1)^\Gamma\ar[r] & c_{b_1^{ad},\mu_1^{ad}}\pi_1(G_1^{ad})^\Gamma.
	}\]
Let $X_{\mu_1}^{G_1}(b_1)^+\subset X_{\mu_1}^{G_1}(b_1)$ be the fiber over $c_{b_1,\mu_1}$ under the map $\omega_{G_1}:  X_{\mu_1}^{G_1}(b_1)\ra c_{b_1,\mu_1}\pi_1(G_1)^\Gamma$. Let $\breve{\M}(G_1,b_1,\mu_1)^+$ be the corresponding formal sub scheme of $\breve{\M}(G_1,b_1,\mu_1)$. On the other hand, we can consider also the fiber  $X_{\mu}^{G}(b)^+\subset X_{\mu}^{G}(b)$ over $c_{b,\mu}$ under $\omega_{G}: X_{\mu}^{G}(b)\ra c_{b,\mu}\pi_1(G)^\Gamma$. Then by Corollary \ref{C:local ADLV}
\[ X_{\mu_1}^{G_1}(b_1)^+\simeq X_{\mu}^{G}(b)^+.\]
We set
\[\breve{\M}(G,b,\mu)^+:=\breve{\M}(G_1,b_1,\mu_1)^+,\]
then  $\ov{\M}G,b,\mu)^{+,perf}\simeq X_{\mu}^{G}(b)^+$. By Theorem \ref{T:CKV} (1), we have
\[X_{\mu}^{G}(b)=J_b(\Q_p)X_{\mu}^{G}(b)^+.\]
Therefore, there exists a formal scheme
$\breve{\M}(G,b,\mu)$, equipped with an action of $J_b(\Q_p)$, such that 
\[\begin{split}\breve{\M}(G,b,\mu)&=J_b(\Q_p)\breve{\M}(G,b,\mu)^+,\\
\ov{\M}(G,b,\mu)^{perf}&\simeq X_{\mu}^{G}(b),\end{split}\]
and the induced action of $J_b(\Q_p)$ on $\M(G,b,\mu)_{red}$ is compatible with that on $X_{\mu}^{G}(b)$ under the above identification.
In fact, we can take
\[\begin{split}\breve{\M}(G,b,\mu)&=[J_b(\Q_p)\times \breve{\M}(G,b,\mu)^+ ]/ J_b(\Q_p)^+\\ &\simeq \coprod_{J_b(\Q_p)/J_b(\Q_p)^+}\breve{\M}(G,b,\mu)^+,\end{split}\]
where $J_b(\Q_p)^+\subset J_b(\Q_p)$ is the stabilizer of $X_{\mu}^{G}(b)^+$ under the action of $J_b(\Q_p)$ on $X_{\mu}^{G}(b)$.

The above construction does not depend on the choice of the unramified local Shimura datum of Hodge type $(G_1, [b_1],\{\mu_1\})$ as in the statement of the theorem, since if $(G_2, [b_2],\{\mu_2\})$ is another such one, then we have a canonical isomorphism
\[\breve{\M}(G_1,b_1,\mu_1)^+\simeq \breve{\M}(G_2,b_2,\mu_2)^+. \]
This follows from the bijection $X_{\mu_1}^{G_1}(b_1)^+\simeq X_{\mu_2}^{G_2}(b_2)^+$, the isomorphism of deformation rings $R_{G_1, x_1}\simeq R_{G_2, x_2}$, where $X_{\mu_1}^{G_1}(b_1)^+ \ni x_1\mapsto x_2 \in X_{\mu_2}^{G_2}(b_2)^+$, cf. \cite{Ki1} 1.5.4 (from the description there, $R_G$ depends only on the adjoint group $G^{ad}$), and the constructions in section 6 of \cite{Kim1}.

\end{proof}

By construction, we have a map of \'etale sheaves $\breve{\M}(G,b,\mu)\ra c_{b,\mu}\pi_1(G)^\Gamma$ over $W$, lifting the map $\omega_G: X_{\leq \mu}^G(b)\lra c_{b,\mu}\pi_1(G)^\Gamma$.
As \cite{Zhu} Corollary 3.12, we have the following dimension formula for the special fibers by applying loc. cit. Theorem 3.1.
\begin{corollary}
Let the notations be as in Theorem \ref{T:ab}. We have $\dim \M_{red}=\lan \rho, \mu-\nu_{[b]}\ran -\frac{1}{2} \tr{def}_G(b)$, where $\rho$ is the half-sum of (absolute) positive roots of $G$, and $\tr{def}_G(b)=rank_{\Q_p}G-rank_{\Q_p}J_b$.
\end{corollary}

Let $(G, [b], \{\mu\})$ be an unramified local Shimura datum of abelian type. Take an embedding $G\hookrightarrow \GL_n$. Then we get an induced triple $(\GL_n,[b'],\{\mu'\})$. If $(G, [b], \{\mu\})$ is not of Hodge type, then $\{\mu'\}$ is not minuscule. In any case, we have the embedding
\[\ov{\M}(G,b,\mu)^{perf}\simeq X_\mu^G(b)\hookrightarrow X_{\leq\mu'}^{\GL_n}(b'). \]

\begin{remark}
In this paper we do not study the Weil descent data on Rapoport-Zink spaces. To define the Weil descent datum on the abelian type Rapoport-Zink space $\breve{\M}(G,b,\mu)$, we just mention that it should be possible to develop a similar theory as that in \cite{D} by dividing the desired Weil descent datum into two parts: one part for $\breve{\M}(G,b,\mu)^+$, and one part for $\pi_1(G)^\Gamma$ so that the morphism $\breve{\M}(G,b,\mu)\ra c_{b,\mu}\pi_1(G)^\Gamma$ is equivariant for the Weil descent data on two sides. The part for $\breve{\M}(G,b,\mu)^+$ is inherited from the Weil descent datum for any associated Hodge type Rapoport-Zink space $\breve{\M}(G_1,b_1,\mu_1)$. For the spaces as in the following Proposition \ref{P:moduli} (1) or Theorem \ref{T:moduli},  the Weil descent datum can be defined quite easily: by quotient from that for $\breve{\M}(G_1,b_1,\mu_1)$ or by moduli methods as that in \cite{RZ} 3.48.
\end{remark}

\subsection{A moduli interpretation}\label{subsection moduli}
Let $(G,[b],\{\mu\})$ be as in Theorem \ref{T:ab}. Then by construction, locally the formal scheme $\breve{\M}(G,b,\mu)$ admits a moduli interpretation. More precisely, take  $(G_1, [b_1],\{\mu_1\})$ as in Definition \ref{D:local abelian unramified}. Then the formal scheme $\breve{\M}(G_1,b_1,\mu_1)$ is a moduli space of $p$-divisible groups with crystalline Tate tensors. In particular, $\breve{\M}(G,b,\mu)^+$ is a moduli space of $p$-divisible groups with crystalline Tate tensors such that under the map $\omega_{G_1}$ the image is fixed. 

Suppose now that there exists a triple $(G_1, [b_1],\{\mu_1\})$ as in Definition \ref{D:local abelian unramified} such that the map 
\[\pi_1(G_1)^\Gamma\ra \pi_1(G_1^{ad})^\Gamma \] is surjective. Then the formal scheme $\breve{\M}(G,b,\mu)$ admits a  global moduli interpretation as follows. 
\begin{proposition}\label{P:moduli}
Under the above assumption, 
\begin{enumerate}
\item we have an isomorphism of formal schemes \[\breve{\M}(G_1^{ad},b_1^{ad},\mu_1^{ad})\simeq\breve{\M}(G_1,b_1,\mu_1)/X_\ast(Z_{G_1})^\Gamma.\]
 \item $\breve{\M}(G,b,\mu)$ is the pullback of $\breve{\M}(G_1,b_1,\mu_1)/X_\ast(Z_{G_1})^\Gamma$ under the morphism 
$\pi_1(G)^\Gamma\ra \pi_1(G^{ad})^\Gamma $.
\end{enumerate}
\end{proposition}
\begin{proof}
We have the following cartesian diagrams (of morphisms between \'etale sheaves over $\ov{\F}_p$)
\[ 
\xymatrix{X_{\mu_1}^{G_1}(b_1)\ar[r]\ar[d] & X_{\mu_1^{ad}}^{G_1^{ad}}(b_1^{ad})\ar[d] & X_\mu^G(b)\ar[l]\ar[d]\\
	c_{b_1,\mu_1}\pi_1(G_1)^\Gamma\ar[r] & c_{b_1^{ad},\mu_1^{ad}}\pi_1(G_1^{ad})^\Gamma & c_{b,\mu}\pi_1(G)^\Gamma\ar[l]
	}\]
inducing the corresponding cartesian diagrams for Rapoport-Zink spaces (as \'etale sheaves over $W$).
All the vertical maps in the above diagram are surjective by Proposition \ref{P: surjective}. The assertions follow by the assumption $\pi_1(G_1)^\Gamma\ra \pi_1(G_1^{ad})^\Gamma $ is surjective.
\end{proof}

\begin{example}
Consider Example \ref{E:PGL} again. As the exact sequence $1\ra \G_m\ra \GL_n\ra \PGL_n\ra 1$ induces a surjection
\[\pi_1(\GL_n)^\Gamma=\pi_1(\GL_n)\ra \pi_1(\PGL_n)^\Gamma, \]
we have \[\breve{\M}(\PGL_n,b,\mu)\simeq \breve{\M}(\GL_n,b_1,\mu_1)/p^\Z.\]
\end{example}
Another example will be given in section 8.

By construction, both the above local moduli interpretation for $\breve{\M}(G,b,\mu)^+$ and the global moduli interpretation in Proposition \ref{P:moduli} are not canonical. Moreover, the formal scheme $\breve{\M}(G,b,\mu)$ associated to a general unramified local Shimura datum of abelian type does not admit a moduli interpretation by $p$-divisible groups with additional structures. Nevertheless, we have
\begin{theorem}\label{T:moduli}
Let $(G,[b],\{\mu\})$ be an unramified local Shimura datum of abelian type. Assume that there exists an unramified local Shimura datum of Hodge type $(G_1,[b_1],\{\mu_1\})$ with a local Hodge embedding $\iota: G_1\hookrightarrow \GL_n$ such that $\iota(b_1)$ has no slope 0, and such that $(G_1^{ad},[b_1^{ad}],\{\mu_1^{ad}\})\simeq (G^{ad},[b^{ad}],\{\mu^{ad}\})$. Then the formal scheme $\breve{\M}(G,b,\mu)$ represents the moduli functor of $(G,\mu)$-displays $\RZ_{G,\mu,b}$ defined in \cite{BP} 4.2.
\end{theorem}
\begin{proof}
We just briefly sketch the arguments: by the proof of the above Proposition 4.8, we have the cartesian diagrams of \'etale sheaves
\[ 
\xymatrix{\breve{\M}(G_1,b_1,\mu_1)\ar[r]\ar[d] & \breve{\M}(G^{ad},b^{ad},\mu^{ad})\ar[d] & \breve{\M}(G,b,\mu)\ar[l]\ar[d]\\
	c_{b_1,\mu_1}\pi_1(G_1)^\Gamma\ar[r] & c_{b_1^{ad},\mu_1^{ad}}\pi_1(G_1^{ad})^\Gamma & c_{b,\mu}\pi_1(G)^\Gamma.\ar[l]
}\]
Consider the B\"ultel-Pappas functors $\RZ_{G_1,\mu_1,b_1}, \RZ_{G_1^{ad},\mu_1^{ad},b_1^{ad}}, \RZ_{G,\mu,b}$ as \'etale sheaves over $W$. By construction their restrictions over $\ov{\F}_p$ are isomorphic to the \'etale sheaves given by the corresponding affine Deligne-Lusztig varieties, cf. \cite{BP} Proposition 4.2.5 and Remark 4.2.6. Thus we have also the cartesian diagrams of \'etale sheaves
\[ 
\xymatrix{\RZ_{G_1,\mu_1,b_1}\ar[r]\ar[d] & \RZ_{G^{ad},\mu^{ad},b^{ad}}\ar[d] & \RZ_{G,\mu,b}\ar[l]\ar[d]\\
	c_{b_1,\mu_1}\pi_1(G_1)^\Gamma\ar[r] & c_{b_1^{ad},\mu_1^{ad}}\pi_1(G_1^{ad})^\Gamma & c_{b,\mu}\pi_1(G)^\Gamma.\ar[l]
}\]
By \cite{BP} Remark 5.2.7, $\breve{\M}(G_1,b_1,\mu_1)$ represents the functor $\RZ_{G_1,\mu_1,b_1}$. Therefore $\breve{\M}(G,b,\mu)$ represents $\RZ_{G,\mu,b}$.
\end{proof}

When passing to the generic fibers, Rapoport-Zink spaces of abelian type are indeed canonical moduli spaces of some objects (local $G$-shtukas in the sense of Scholze): see the next section.

\subsection{Generic fibers and local Shimura varieties of abelian type}\label{subsection:generic}
Let $(G,[b],\{\mu\})$ and $\breve{\M}=\breve{\M}(G,b,\mu)$ be as in Theorem \ref{T:ab}. We consider the rigid analytic fiber $\M=\M(G,b,\mu)$ over $L$, regarded as an adic space. For any open compact subgroup $K\subset G(\Z_p)$, we construct a finite \'etale cover $\M_K$ of $\M$ as follows. If $(G,[b],\{\mu\})$ is of Hodge type, then this is known by \cite{Kim1} 7.4 (see also our subsection 3.3). Now consider the general case.

First, assume that $K=K_n$ for some $n\geq 1$, where $K_n=\ker \Big(G(\Z_p)\ra G(\Z_p/p^n\Z_p)\Big)$. On the component  $\M^+=(\breve{\M}(G,b,\mu)^{+})^{ad}_\eta$, we can construct a finite \'etale cover $\M_n^+$ by taking some unramified local Shimura datum of Hodge type $(G_1, [b_1],\{\mu_1\})$ as in Definition \ref{D:local abelian unramified} and using the moduli interpretation of $\M(G_1,b_1,\mu_1)$.  We  can take \[\M_n=[J_b(\Q_p)\times\M_n^+]/J_b(\Q_p)^+.\]
In this way we get a tower $(\M_n)_n$ on which $G(\Z_p)$ acts. Set $\M_0=\M$. The action of $G(\Z_p)$ on $\M_n$ factors through $G(\Z_p)/K_n=G(\Z_p/p^n\Z_p)$. Now let $K\subset G(\Z_p)$ be arbitrary. Take some sufficiently large $n$ such that $K_n\subset K$. Set
\[\M_K=\M_n/K.\] Then $\M_K$ is a finite \'etale cover of $\M$, and it does not depend on the choice of $n$. When $K\subset G(\Z_p)$ is normal,  $\M_K$ is a Galois cover of $\M$, with Galois group $G(\Z_p)/K$. For any $g\in G(\Q_p)$ and any open compact subgroup $K\subset G(\Z_p)$, we have a natural isomorphism
\[\M_K\st{\sim}{\lra}\M_{gKg^{-1}}.\]As a result, the group $G(\Q_p)$ acts on the tower
\[(\M_K)_{K\subset G(\Z_p)}\]by Hecke correspondences. 

As before, for any open compact $K\subset G(\Z_p)$, let $\M_K^+\subset\M_K$ be the pullback of $\M^+\subset \M$. In this way we get a sub tower $(\M_K^+)_K\subset (\M_K)_K$.
Let $G(\Q_p)^+\subset G(\Q_p)$ be the subgroup which is the stabilizer of the sub tower
\[(\M_K^+)_K\subset (\M_K)_K.\]By Lemma \ref{L:surjective} (1) the map
\[\omega_G: G(\Q_p)\ra\pi_1(G)^\Gamma\] is surjective. By construction we have the induced bijection
\[\omega_G: G(\Q_p)/G(\Q_p)^+\st{\sim}{\lra}\pi_1(G)^\Gamma,\]and moreover,
\[\M_K= J_b(\Q_p)\M_K^+, \quad (\M_K)_K=G(\Q_p) (\M_K^+)_K.\]

Let $(G,[b],\{\mu\})$ be an unramified local Shimura datum of abelian type. In the following we want to construct a period map 
$\pi_{dR}: \M\ra \Fl\ell_{G,\mu}$ on the generic fiber $\M$ of the associated Rapoport-Zink space, and to study some of its properties.
Take any $(G_1,[b_1],\{\mu_1\})$ as in Definition \ref{D:local abelian unramified}. Then we have the canonical identification of the associated $p$-adic flag varieties over $L$
\[\Fl\ell_{G,\mu}=G/P_\mu = \Fl\ell_{G_1,\mu_1}=G_1/P_{\mu_1}.\] Sometimes we will simply write them as $\Fl\ell_\mu$. By \cite{Kim1} 7.5, we have a period map \[\pi_{G_1,dR}: \M(G_1,b_1,\mu_1)\ra \Fl\ell_\mu,\]which is $J_{b_1}(\Q_p)$-equivalent. If $(G, [b],  \{\mu\})\hookrightarrow (\GL_n, [b'],  \{\mu'\})$ is an embedding of unramified local Shimura data of Hodge type, we get an induced embedding of flag varieties $\Fl\ell_{G_1,\mu_1}\hookrightarrow \Fl\ell_{\GL_n,\mu'}$ over $L$. By construction,
we have the following commutative diagram
\[\xymatrix{
\M(G_1,b_1,\mu_1)\ar[d]^{\pi_{G_1,dR}}\ar@{^{(}->}[r]&\M(\GL_n,b',\mu')\ar[d]^{\pi_{\GL_n,dR}}\\
\Fl\ell_{G_1,\mu_1}\ar@{^{(}->}[r]&\Fl\ell_{\GL_n,\mu'}.
}
\]
Let us briefly review the construction of $\pi_{G_1,dR}$.
Let $(s_\alpha)\subset \Lambda^\otimes$ be a finite collection of tensors ($\tr{rank} \Lambda=n$) such that $G_1\subset \GL(\Lambda)$ is the   schematic stabilizer of $(s_\alpha)$. Then the closed embedding $\Fl\ell_{G_1,\mu_1}\hookrightarrow \Fl\ell_{\GL_n,\mu'}$ classifies $\{\mu_1\}$-filtrations of $\Lambda$ with respect to $(s_\alpha)$ cf. \cite{Kim1} Definition 2.2.3 and Lemma 2.2.8. By \cite{Kim1} 7.5, the period morphism $\pi_{G_1,dR}: \M(G_1,b_1,\mu_1)\ra \Fl\ell_{G_1,\mu_1}$ is given by $\big(\Fil^\bullet \D(X^{univ})_{\breve{\M}_1}^{rig}, (t_\alpha^{rig})\big)$ using the induced isomorphism $\rho: \D(X^{univ})_{\breve{\M}_1}^{rig}\simeq \Ol_{\M_1}\otimes \Lambda$ which matches $(t_\alpha^{rig})$ with $1\otimes s_\alpha$, where $\breve{\M}_1=\breve{\M}(G_1,b_1,\mu_1), \M_1=\M(G_1,b_1,\mu_1), (X^{univ}, (t_\alpha), \rho)$ is the universal $p$-divisible group with crystalline Tate tensors and quasi-isogeny over $\breve{\M}_1$. Thus the above diagram is commutative.

Restricting the map $\pi_{G_1,dR}$ to $\M(G_1,b_1,\mu_1)^+=\M^+$, we get a map
\[\pi_{dR}^+: \M^+=\M(G_1,b_1,\mu_1)^+\ra \Fl\ell_\mu.\]
Then applying the group action of $J_b(\Q_p)$, we can define a $J_b(\Q_p)$-equivariant period map for $\M$
\[\pi_{dR}=\pi_{G, dR}: \M=\M(G,b,\mu)\ra \Fl\ell_\mu.\]

Let $\Fl\ell_{G_1,\mu_1}^{adm}\subset \Fl\ell_\mu$ be the open subspace defined by Hartl (using Robba rings) in \cite{Har} section 6, which can be defined equivalently by using the crystalline period ring $B_{cris}$ (cf. \cite{Fal}). In \cite{SW, Ra2},  the subspace $\Fl\ell_{G,\mu}^{adm}$ is described using the Fargues-Fontaine curve, which applies to an arbitrary local Shimura datum $(G, [b], \{\mu\})$. See also Proposition \ref{P:admissible}.
\begin{proposition}\label{P:adm Hodge}
$\Fl\ell_{G_1,\mu_1}^{adm}$ is the image of $\pi_{G_1,dR}$. And we have the following commutative diagram
\[\xymatrix{
	\M(G_1,b_1,\mu_1)\ar@{->>}[d]^{\pi_{G_1,dR}}\ar@{^{(}->}[r]&\M(\GL_n,b',\mu')\ar@{->>}[d]^{\pi_{\GL_n,dR}}\\
	\Fl\ell_{G_1,\mu_1}^{adm}\ar@{^{(}->}[r]&\Fl\ell_{\GL_n,\mu'}^{adm}.
}
\]
\end{proposition}
\begin{proof}
By the above construction, the composition $\M(G_1,b_1,\mu_1)\hookrightarrow  \M(\GL_n,b',\mu') \st{\pi_{\GL_n,dR}}{\lra} \Fl\ell_{\GL_n,\mu'}$ factors through $\Fl\ell_{G_1,\mu_1}$.
 By \cite{Har} Theorem 7.3 and \cite{Fal} section 4, we have $\tr{Im}\;\pi_{\GL_n,dR}=\Fl\ell_{\GL_n,\mu'}^{adm}$. On the other hand,  by \cite{Har} Proposition 6.2, we have \[\Fl\ell_{G_1,\mu_1}^{adm}=\Fl\ell_{\GL_n,\mu'}^{adm}\cap\Fl\ell_{G_1,\mu_1}.\]Thus the above diagram commutes. To show $\tr{Im}\;\pi_{G_1,dR}=\Fl\ell_{G_1,\mu_1}^{adm}$, 
 it suffices to show that for any algebraically closed complete extension $C|L$, the induced map on $C$-valued points is surjective. Let $(x, (s_{x\alpha}))\in \Fl\ell_{G_1,\mu_1}(C, \Ol_C)$ with image $x\in \Fl\ell_{\GL_n,\mu'}(C,\Ol_C)$, such that there exists a point $(X/\Ol_C, \rho)\in \M(\GL_n,b',\mu')(C, \Ol_C)$ maps to $x$ under $\pi_{\GL_n,dR}$. By definition, we have the isomorphism
\[\rho: \D(X_{k_C})_\Q\simeq \Ol_{\M(\GL_n,b',\mu')}\otimes \Lambda,\]and $x=\rho(\Fil^1\D(X_{k_C})_\Q)$ considered as a filtration on the right hand side. Via the rigidification $\rho$, there exists an element $g\in G(\Q_p)/G(\Z_p)\subset \GL_n(\Q_p)/\GL_n(\Z_p)$ such that $\D(X_{k_C})\simeq (g\Lambda\otimes W, g^{-1}b'\sigma(g)\sigma)$.
Therefore, each tensor $s_\alpha$ on $\Lambda$ induces a crystalline Tate tensor $t_\alpha$ on $X$. We get a point $(X/\Ol_C, (t_{\alpha}), \rho)\in \M(G_1,b_1,\mu_1)(C, \Ol_C)$, which by construction maps to $(X/\Ol_C, \rho)\in \M(\GL_n,b',\mu')(C, \Ol_C)$ and $(x, (s_{x\alpha}))\in \Fl\ell_{G_1,\mu_1}(C, \Ol_C)$  under the embedding $\M(G_1,b_1,\mu_1)$ $\hookrightarrow \M(\GL_n,b',\mu')$ and the period map $\pi_{G_1,dR}$ respectively. 
\end{proof}

For any open compact subgroup $K\subset G(\Z_p)$, we have the finite \'etale map $\M^+_K=\M(G_1,b_1,\mu_1)^+_K\ra \M^+=\M(G_1,b_1,\mu_1)^+$, thus we get a morphism
\[\M^+_K\ra \Fl\ell_\mu.\]From this we can define a $J_b(\Q_p)$-equivalent period map for $\M_K$
\[\pi_{G, dR}: \M_K\ra \Fl\ell_\mu.\]
When $K$ varies, these period maps are compatible with the Hecke action of $G(\Q_p)$ on $(\M_K)_K$. Thus we may think that there exists a 
$G(\Q_p)$-invariant map $(\M_K)_K\ra \Fl\ell_\mu$.

Recall that we also have $\Fl\ell_{G_1,\mu_1}^{wa}$ and $\Fl\ell_{G,\mu}^{wa}$.
By construction, we have $\Fl\ell_{G_1,\mu_1}^{adm}\subset\Fl\ell_{G_1,\mu_1}^{wa}$, and similarly $\Fl\ell_{G,\mu}^{adm}\subset\Fl\ell_{G,\mu}^{wa}$.
\begin{lemma}\label{L:period}
We have \[\Fl\ell_{G_1,\mu_1}^{wa}=\Fl\ell_{G,\mu}^{wa},\quad
\Fl\ell_{G_1,\mu_1}^{adm}= \Fl\ell_{G,\mu}^{adm}.\]
\end{lemma}
\begin{proof}
The equality $\Fl\ell_{G_1,\mu_1}^{wa}=\Fl\ell_{G,\mu}^{wa}$ follows by \cite{DOR} Proposition 9.5.3 (iv). 
The second equality follows by the definition using $G$-bundles on the Fargues-Fontaine curve.
\end{proof}

\begin{corollary}\label{C:image ab}
$\Fl\ell_{G,\mu}^{adm}\subset \Fl\ell_\mu$ is the image of the above period map $\pi_{G,dR}$.
\end{corollary}
\begin{proof}
Let $\Fl\ell_{G_1,\mu_1}^{adm,+}\subset\Fl\ell_{G_1,\mu_1}$ be the image of $\pi_{dR}^+$. 

Since $(\M(G_1,b_1,\mu_1)_K)_K=G_1(\Q_p)(\M(G_1,b_1,\mu_1)_K^+)_K$ and the map $\M(G_1,b_1,\mu_1)_K\ra \Fl\ell_{G_1,\mu_1}^{adm}$ is $G_1(\Q_p)$-invariant, we get
\[\tr{Im}\,\pi_{G_1,dR}=\Fl\ell_{G_1,\mu_1}^{adm,+}.\]

We have also $(\M(G,b,\mu)_K)_K=G(\Q_p)(\M(G,b,\mu)_K^+)_K$, and by our construction the map $(\M(G,b,\mu)_K)_K\ra \Fl\ell_\mu$ is $G(\Q_p)$-invariant, we get also  \[\tr{Im}\,\pi_{G,dR} =\Fl\ell_{G_1,\mu_1}^{adm,+}.\]
Thus $\tr{Im}\,\pi_{G_1,dR}=\tr{Im}\,\pi_{G,dR} $. By Proposition \ref{P:adm Hodge} and Lemma \ref{L:period}, we have 
\[\tr{Im}\,\pi_{G,dR} =\tr{Im}\,\pi_{G_1,dR}=\Fl\ell_{G_1,\mu_1}^{adm}= \Fl\ell_{G,\mu}^{adm}.\]
\end{proof}

\begin{remark}
	We always have $\Fl\ell_{G,\mu}^{adm}\subset\Fl\ell_{G,\mu}^{wa}$. In \cite{Har} section 9 and \cite{Ra2} Question A. 20, Hartl and Rapoport asked that when is $\Fl\ell_{G,\mu}^{adm}=\Fl\ell_{G,\mu}^{wa}$? For $G=\GL_n$, in Theorem 9.3 of \cite{Har} Hartl gave a complete solution of this question. For arbitrary $G$ and minuscule $\mu$, Fargues and Rapoport conjecture that this holds true with $[b]$ basic if and only if $(G, \{\mu\})$ is fully Hodge-Newton decomposable in the sense of \cite{GoHeNi} Definition 2.1 (2), cf. \cite{GoHeNi} Conjecture 0.1.
 In the appendix we will see that $\Fl\ell_{G,\mu}^{adm}=\Fl\ell_{G,\mu}^{wa}$ in the case $[b]$ is basic and $G$ is the special orthogonal group. For a solution of the Fargues-Rapoport conjecture, see \cite{CFS}.
\end{remark}

Recall that by Lemma \ref{L:surjective} (1)  the  map
\[\omega_G: G(\Q_p)/G(\Z_p)\ra \pi_1(G)^\Gamma\]is surjective.
\begin{lemma}\label{L:cartesian}
\begin{enumerate}
\item The following diagram is cartesian: \[\xymatrix{G(\Q_p)/G(\Z_p) \ar[r]^{\omega_{G}}\ar[d]& \pi_1(G)^\Gamma\ar[d]\\
	G^{ad}(\Q_p)/G^{ad}(\Z_p)\ar[r]^-{\omega_{G^{ad}}} & \pi_1(G^{ad})^\Gamma.
}\]
\item  In particular, for $G$ and $G_1$ as above we have $G(\Q_p)^+\simeq G_1(\Q_p)^+$.
\end{enumerate}
\end{lemma}
\begin{proof}
Note that non empty fibers of both vertical maps are torsors under $X_\ast(Z_G)^\Gamma$. By \cite{Ki2} Lemma 1.2.4, if $g^{ad}\in G^{ad}(\Q_p)/G^{ad}(\Z_p)$ and $\omega_{G^{ad}}(g^{ad})$ lifts to an element of $\pi_1(G)^\Gamma$, then $g^{ad}$ lies in the image of $G(\Q_p)/G(\Z_p) \ra G^{ad}(\Q_p)/G^{ad}(\Z_p)$. Therefore the above diagram is cartesian.

In particular, we have the bijection $G(\Q_p)^+\simeq G_1(\Q_p)^+$ from (1) for $G$ and $G_1$ as above.
\end{proof}

Let $X$ be a rigid analytic space over a local field $k|\Q_p$.  By \cite{dJ1} section 5 and \cite{KL} 8.4, we have the categories of $\Z_p$-local systems and $\Q_p$-local systems on $X$. Denote them by $\Z_p-Loc_{X}$ and  $\Q_p-Loc_{X}$ respectively. Let $G$ be a reductive group over $\Q_p$. Denote by $\Rep G$ the category of rational representations of $G$. Recall that a $\Q_p$-$G$-local system on $X$ is an exact tensor functor $\Rep G\ra \Q_p-Loc_{X}$ (see \cite{Han} 4.3 for example). If $G$ is moreover unramified, and fix a reductive model $G_{\Z_p}$ of $G$ over $\Z_p$,  then we can define similarly $\Z_p$-$G$-local systems (or better notion: $G_{\Z_p}$-local systems) on $X$. In the following we will take $X=\Fl\ell_{G,\mu}^{adm}$ or $X=\M$.
By construction, we have
\begin{proposition}\label{P:local system}
There exists a $J_b(\Q_p)$-equivariant $\Q_p$-$G$-local system $\mathbb{V}$ on $\Fl\ell_{G,\mu}^{adm}$ such that for any affinoid algebra $(R, R^+)$ over $(L, \mathcal{O}_L)$, $\M(R,R^+)$ is the set of $G_{\Z_p}$-local systems in $\mathbb{V}_{\Spa(R,R^+)}$. In particular, there exists a $J_b(\Q_p)$-equivariant $G_{\Z_p}$-local system $\mathbb{L}$ on $\M$, and the tower $(\M_K)_{K\subset G(\Z_p)}$ is obtained by trivializing $\mathbb{L}$.
\end{proposition}
\begin{proof}
Under the identity $\Fl\ell_{G_1,\mu_1}^{adm}=\Fl\ell_{G,\mu}^{adm}$, we have a $\Q_p-G_1$-local system $\mathbb{V}_1$ on $\Fl\ell_{G,\mu}^{adm}$. Indeed, let $V_p(X^{univ})$ be the rational Tate module of the universal $p$-divisible group $X^{univ}$ over $\M_1$. We have the \'etale Tate tensors \[t_{\alpha,et}: 1\ra V_p(X^{univ})^\otimes\] corresponding to each $t_\alpha$ under the comparison theorem, cf. \cite{Kim1} Theorem 7.1.6. $V_p(X^{univ})$ descends to a $\Q_p$-local system $\V_1$ on $\Fl\ell_{G,\mu}^{adm}$, equipped with the induced  \'etale Tate tensors $t_{\alpha,et}$.  Fix any geometric point $\ov{x}\ra\Fl\ell_{G,\mu}^{adm}$. Let \[\rho_{\V_1, \ov{x}}: \pi_1(\Fl\ell_{G,\mu}^{adm}, \ov{x}) \ra \GL_n(\Q_p)\] be the $p$-adic representation of the (de Jong's) fundamental group $\pi_1(\Fl\ell_{G,\mu}^{adm}, \ov{x}) $ corresponding to $\V_1$, cf. \cite{dJ1} Theorem 4.2. Then as $t_{\alpha,et}$ is invariant under $\pi_1(\Fl\ell_{G,\mu}^{adm}, \ov{x}) $, cf. \cite{Kim1} Theorem 7.1.6, we get a morphism 
\[\rho_{\mathbb{V}_1, \ov{x}}: \pi_1(\Fl\ell_{G,\mu}^{adm}, \ov{x}) \ra G_1(\Q_p)\] 
which thus defines a $\Q_p-G_1$-local system $\mathbb{V}_1$ on $\Fl\ell_{G,\mu}^{adm}$. Moreover, as in the proof of Proposition \ref{P:adm Hodge},
$\M_1(R,R^+)$ can be identified with the set of $\Z_p$-lattices together with tensors $(t_\alpha)$ in $\V_{1\Spa(R,R^+)}$, or equivalently, $[\Isom_{\Fl\ell_{G,\mu}^{adm}}(\ul{G_1}, \mathbb{V}_1)/G_1(\Z_p)](R,R^+)$,  where $\ul{G_1}$ is the trivial $\Q_p$-$G_1$-local system on $\Fl\ell_{G,\mu}^{adm}$.
The tower $(\M_{1K})_{K\subset G_1(\Z_p)}$ is the geometric realization of $\Q_p$-$G_1$-local system $\mathbb{V}_1$ on $\Fl\ell_{G,\mu}^{adm}$ in the sense that \[ \M_{1K}\simeq \Isom_{\Fl\ell_{G,\mu}^{adm}}(\ul{G_1}, \mathbb{V}_1)/K.\] This identification preserves the Hecke actions of $G_1(\Q_p)$ and the actions of $J_{b_1}(\Q_p)$, cf. \cite{Har} Remark 2.7 and the proof of loc. cit. Theorem 7.3 (c) and (d).

The group $\pi_1(\Fl\ell_{G,\mu}^{adm}, \ov{x}) $ acts on $G_1(\Q_p)$ through $\rho_{\mathbb{V}_1, \ov{x}}$. The group $J_{b_1}(\Q_p)$ acts on $G_1(\Q_p)$  as the $\Q_p$-local system $\mathbb{V}_1$ on $\Fl\ell_{G,\mu}^{adm}$ is $J_{b_1}(\Q_p)$-equivariant.

Fix a point $x_0\in \pi_1(G_1)^\Gamma$. Then we have the associated $\breve{\M}_1^+$ and $(\M_{1K}^+)_K$. The tower $(\M_{1K}^+)_K$ defines a subgroup $G_1(\Q_p)^+\subset G_1(\Q_p)$ and a morphism \[\rho_{\mathbb{V}_1, \ov{x}}^+: \pi_1(\Fl\ell_{G,\mu}^{adm}, \ov{x}) \ra G_1(\Q_p)^+.\] By Lemma \ref{L:cartesian} (2), we have $G(\Q_p)^+\simeq G_1(\Q_p)^+$. Therefore, we can define an action of $\pi_1(\Fl\ell_{G,\mu}^{adm}, \ov{x}) $ on $G(\Q_p)$, which commutes with the natural action of $J_b(\Q_p)$. Thus we get a $p$-adic representation  \[\rho_{ \ov{x}}: \pi_1(\Fl\ell_{G,\mu}^{adm}, \ov{x}) \ra G(\Q_p),\] which defines the desired 
$\Q_p$-$G$-local system $\mathbb{V}$ on $\Fl\ell_{G,\mu}^{adm}$. Moreover, for any $K\subset G(\Z_p)$, we have the identification
\[ \M_{K}\simeq \Isom_{\Fl\ell_{G,\mu}^{adm}}(\ul{G}, \mathbb{V})/K,\] where $\ul{G}$ is the trivial $\Q_p$-$G$-local system on $\Fl\ell_{G,\mu}^{adm}$. As above, this identification preserves the Hecke actions of $G(\Q_p)$ and the actions of $J_{b}(\Q_p)$.
\end{proof}
We note Corollary \ref{C:image ab} and Proposition \ref{P:local system} generalize \cite{Har} Theorem 8.4 (EL/PEL type case, but there one can allow ramification) to the abelian type case.

Let $(G,[b],\{\mu\})$ be an unramified local Shimura datum of abelian type. For each open compact subgroup $K\subset G(\Z_p)$, we get the associated Rapoport-Zink space
\[\M_K\simeq \coprod_{\pi_1(G)^\Gamma}\M^+_K.\]
Let $\Delta_G$ be the image of $\pi_1(G)^\Gamma\ra \pi_1(G^{ad})^\Gamma$. This is a finite group. We have an exact sequence
\[1\ra X_\ast(Z_G)^\Gamma\ra \pi_1(G)^\Gamma\ra \Delta_G\ra 1.\]We have the Hecke action of $G(\Q_p)$ on the tower $(\M_K)_K$.
The Hecke action of the central subgroup $Z_G(\Q_p)\subset G(\Q_p)$ stabilizes each $\M_K$. This action of $Z_G(\Q_p)$ is the same of that induced from $J_b(\Q_p)$ when we view $Z_G(\Q_p)\subset J_b(\Q_p)$. This action on \[\begin{split} \M_K&\simeq\coprod_{\pi_1(G)^\Gamma}\M^+_K \\
&=\coprod_{\Delta_G}\coprod_{X_\ast(Z_G)^\Gamma}\M_K^+
\end{split}\] is through the map $Z_G(\Q_p)\ra X_\ast(Z_G)^\Gamma$ and the injection $X_\ast(Z_G)^\Gamma\ra \pi_1(G)^\Gamma$. 

In summary, the tower $(\M_K)_{K\subset G(\Z_p)}$ associated to an unramified local Shimura datum of abelian type can be viewed as the local Shimura varieties thought of in Conjecture \ref{C:RV}. In the next section, we will put these spaces in a more general framework to get some moduli interpretation for each $\M_K$.

\subsection{Infinite level and the Hodge-Tate period map}
Let $(G,[b],\{\mu\})$ be an unramified local Shimura datum of abelian type, and $(\M_K)_K$ be associated tower of Rapoport-Zink spaces of abelian type. Let $\Fl\ell_{G, \mu^{-1}}$ be the $p$-adic flag variety over $L$ associated to $(G, \{\mu^{-1}\})$.
\begin{proposition}\label{P:local perfectoid}
	There exists a preperfectoid space $\M_\infty$ over $L$ such that\[\M_\infty\sim \varprojlim_K\M_K,\] cf. \cite{SW} Definition 2.4.1 for the precise meaning of such formula. Moreover, there exists a Hodge-Tate period map
	\[\pi_{HT}: \M_\infty \ra \Fl\ell_{G, \mu^{-1}},\]which agrees with the period map previously defined in the EL/PEL cases in \cite{SW, CS}.
\end{proposition}
\begin{proof}
If $(G,[b],\{\mu\})$ is of Hodge type, the existence of the preperfectoid space $\M_\infty$ over $L$ such that $\M_\infty\sim \varprojlim_K\M_K$ is proved in \cite{Kim1} Proposition 7.6.1. Fix an embedding $(G,[b],\{\mu\})\hookrightarrow (\GL_n,[b'],\{\mu'\})$ with $\{\mu'\}$ minuscule. We have the associated preperfectoid space $\M(\GL_n,b',\mu')_\infty$ over $L$ such that $\M(\GL_n,b',\mu')_\infty\sim \varprojlim_{K'}\M(\GL_n,b',\mu')_{K'}$. The Hodge-Tate period map \[\pi_{HT}: \M(\GL_n,b',\mu')_\infty\ra \Fl\ell_{\GL_n,(\mu')^{-1}} \] is defined in \cite{SW} 7.1. Arguing as \cite{CS} section 2, we get that the composition \[ \M_\infty\hookrightarrow \M(\GL_n,b',\mu')_\infty\ra \Fl\ell_{\GL_n,(\mu')^{-1}} \]factors through $\Fl\ell_{G,\mu^{-1}}$. In particular we get
\[\pi_{HT}: \M_\infty \ra \Fl\ell_{G,\mu^{-1}}.\]
Now assume that we are in the general case.
As $J_b(\Q_p)$ acts on $|\M_\infty|:=\varprojlim_{K}|\M_K|$, it suffices to prove that there exist a preperfectoid space $\M_\infty^+$ over $L$ such that\[\M_\infty^+\sim \varprojlim_K\M_K^+,\]and a Hodge-Tate period map \[\pi_{HT}^+: \M_\infty^+\ra \Fl\ell_{G,\mu^{-1}}.\] This follows from the Hodge type case.
\end{proof}
The following corollary is clear now.
\begin{corollary}
There exists a sub preperfectoid space $\M_\infty^+\subset \M_\infty$ over $L$, which is stable under $G(\Q_p)^+$, such that\[\M_\infty^+\sim \varprojlim_K\M_K^+, \quad \M_\infty=G(\Q_p)\M_\infty^+.\]
\end{corollary}

\section{Generic fibers of Rapoport-Zink spaces as moduli of local $G$-shtukas}\label{section:shtuka}
In this section, we work mainly on generic fibers.
We want to explain that the generic fibers of the formal schemes $\breve{\M}(G,b,\mu)$, associated to unramified local Shimura data of abelian type $(G,[b],\{\mu\})$, can be viewed as moduli spaces fo local $G$-shtukas in mixed characteristic in the sense of Scholze\footnote{By \cite{Sch4, SW17}, the same should be true even for the formal schemes $\breve{\M}(G,b,\mu)$!}, cf. \cite{Sch2}. We will work in the more general context of Conjecture \ref{C:RV}. The first few subsections will be a brief review of works of Fargues \cite{F3, FF} and Scholze \cite{Sch2}. The reader familiar with these can go directly to subsection \ref{Section:diamond}. 

\subsection{The Fargues-Fontaine curve and $G$-bundles}
The Fargues-Fontaine curve $X_{F, E}$ is associated to a datum $(F, E)$, where $E$ is  a local field with finite residue field $\F_q$ and $F|\F_q$ is a perfectoid field of characteristic $p$. For our purpose, we set $E=\Q_p$, and denote simply $X_{F, \Q_p}$ as $X_F$. It has several incarnations.
\subsubsection{The adic curve}
The adic curve $X_F$ admits the following adic uniformization \[X_F=Y_F/\phi^\Z,\]where $Y_F=\Spa(W(\Ol_F))\setminus V(p[\varpi_F])$, with $\varpi_F\in F$ satisfying $0<|\varpi_F|<1$. The action of the Frobenius $\phi$ on the Witt vectors is given by
\[\phi(\sum_n[x_n]p^n)=\sum_n[x_n^p]p^n, \quad \forall\, \sum_n[x_n]p^n\in W(\Ol_F).\]It induces a totally discontinuous action on $Y_F$.

 Suppose now that $F$ is algebraically closed. Then there is a unique non analytic point $x_k\in \Spa(W(\Ol_F))$. Set $\Y=\Y_F=\Spa(W(\Ol_F))\setminus\{x_k\}$. There exists a surjective continuous map $\kappa: \Y\ra \R_{\geq 0}\cup\{\infty\}$ defined by
 \[\kappa(x)=\frac{\log|[\varpi_F](\wt{x})|}{\log |p(\wt{x})|},\]where $\wt{x}$ is the unique maximal generalization of $x$, cf. \cite{Sch2} 12.2. For any $I\subset \R_{\geq 0}\cup\{\infty\}$, we denote $\Y_I=\kappa^{-1}(I)$. Then $Y:=Y_F=\Y_{(0,\infty)}$. 
 
 Let $I\subset [0,\infty]$ be an interval of the form $[r,\infty)$ or $[r,\infty]$.
 Recall that a $\phi$-module over $\Y_I$ is a pair $(\E, \phi_\E)$, where $\E$ is a vector bundle over $\Y_I$ and $\phi_\E: \phi^\ast\E|_{\Y_I}\ra \E$ is an isomorphism, cf. \cite{Sch2} Definition 13.2.1.
 It follows that $\phi$-modules over $\Y_{(0,\infty)}$ are the same as vector bundles over $X:=X_F$.

\subsubsection{The algebraic curve}
There is a natural line bundle $\Ol(1)$ on $X$, corresponding to the $\phi$-module on $\Y_{(0,\infty)}$ whose underlying line bundle is trivial and for which $\phi$ is $p^{-1}\phi$. Set $\Ol(n)=\Ol(1)^{\otimes n}$, and \[P=\bigoplus_{n\geq 0}H^0(X,\Ol(n)).\]We have 
\[H^0(X, \Ol(n))=\Ol(Y)^{\phi=p^n}.\]Let \[X^{sch}=\Proj(P).\] By \cite{FF}, this is a one dimensional noetherian regular scheme over $\Q_p$. 
There exists a morphism of ringed spaces
\[X\lra X^{sch},\]and $X$ may be viewed as the analytification of $X^{sch}$ in some generalized sense.

\begin{remark}
Using the theory of diamond developed in \cite{Sch2}, the curve admits yet another version: the diamond curve
\[X^\diamond=(\Spa (F)\times \Spa(\Q_p)^\diamond)/\phi^\Z,\]where $\phi=Frob_F\times Id$. We will not use this version in the following.
\end{remark}

Let $\Bun_{X^{sch}}$ and $\Bun_X$ be the categories of vector bundles on $X^{sch}$ and $X$ respectively. The morphism $X\lra X^{sch}$ induces a GAGA functor
\[\Bun_{X^{sch}}\lra \Bun_X.\]
\begin{theorem}[\cite{KL, F2}]
The GAGA functor induces an equivalence of categories \[\Bun_{X^{sch}}\st{\sim}{\lra} \Bun_X.\]
\end{theorem}

There is another way to describe vector bundles on $X$. Consider the Robba ring
\[\wt{\mathcal{R}}_F=\varinjlim_r H^0(\Y_{(0,r]},\Ol_{\Y_{(0,r]}} ).\]The Frobenius $\phi$ induces an action on $\wt{\mathcal{R}}_F$.
Recall a $\phi$-module over $\wt{\mathcal{R}}_F$ is a finite free $\wt{\mathcal{R}}_F$-module $M$ equipped with a $\phi$-linear automorphism.
\begin{theorem}[\cite{KL}, Theorem 6.3.12]
There is an equivalence of categories 
\[\Bun_X\simeq \{\phi-\tr{modules over}\, \wt{\mathcal{R}}_F \}.\]
\end{theorem}

The idea for the proof is that any $\phi$-module over $\wt{\mathcal{R}}_F$ is defined over \[\wt{\mathcal{R}}_F^r:=H^0(\Y_{(0,r]},\Ol_{\Y_{(0,r]}} )\] for some $r$ small enough. This can be spread to a $\phi$-module over $Y_F=\Y_{(0,\infty)}$ via pullback under Frobenius. Giving a $\phi$-module over $\Y_{(0,\infty)}$ is the same giving a vector bundle over $X_F$ by the uniformization $X_F=\Y_{(0,\infty)}/\phi^\Z$.

Let $\phi-\Mod_L$ be the category of $F$-isocrystals over $\ov{\F}_p$, where as before $L=W(\ov{\F}_p)_\Q$. For any $(D,\phi)\in \phi-\Mod_L$, we can construct a vector bundle $\E(D,\phi)$ on $X^{sch}$ by
\[ \E(D,\phi)=\Proj\Big(\bigoplus_{n\geq 0}(D\otimes_L \Ol(Y))^{\phi\otimes \phi=p^n}\Big). \]
\begin{theorem}[\cite{FF}]
The functor $\E(-): \phi-\Mod_L\ra \Bun_{X^{sch}}$ is essentially surjective.
\end{theorem}
Therefore, the composite $\E(-): \phi-\Mod_L\ra \Bun_{X^{sch}}\ra \Bun_X$ is also essentially surjective.

Let $G$ be a connected reductive group over $\Q_p$. We have the following equivalent definitions of a $G$-bundle on $X$ (or equivalently on $X^{sch}$):
\begin{enumerate}
	\item an exact tensor functor $\Rep G\ra \Bun_X$, where as before $\Rep G$ is the category of rational algebraic representations of $G$,
	\item a $G$-torsor on $X$ locally trivial for the \'etale topology.
\end{enumerate}

Recall that an $F$-isocrystal with $G$-structure over $\ov{\F}_p$ is an exact tensor functor \[\Rep G\lra \phi-\Mod_L.\] If $b\in G(L)$, it then defines an $F$-isocystal with a $G$-structure \[\begin{split} M_b: \Rep G&\lra \phi-\Mod_L\\ V&\longmapsto (V_L, b\sigma).\end{split}\] Its isomorphism class only depends on the $\sigma$-conjugacy class $[b]\in B(G)$ of $b$. Conversely, by Steinberg's theorem any $F$-isocrystal with $G$-structure arises in this way. Thus $B(G)$ is the set of isomorphism classes of $F$-isocrystals with $G$-structure, cf. \cite{RR} Remarks 3.4 (i). For $b\in G(L)$, let $\E_b$ be the composition of the above functor $M_b$ and \[\E(-):  \phi-\Mod_L\ra \Bun_{X^{sch}}\simeq \Bun_X.\] In this way, the set $B(G)$ also classifies $G$-bundles on $X$. In fact, we have
\begin{theorem}[\cite{F3}]\label{T: G-bundles}
	Assume that $F$ is algebraically closed. Then
	there is a bijection of sets
	\[\begin{split} B(G)&\st{\sim}{\lra} H^1_{et}(X, G) \\ 
  [b]&\longmapsto [\E_b]. \end{split}	\]
\end{theorem}

We discuss briefly the relative version of the above theory.
Let $(R,R^+)$ be a perfectoid affinoid $\F_p$-algebra, and $S=\Spa(R,R^+)$ be the associated perfectoid space. We have an adic space over $\Q_p$:
\[X_S=Y_S/\phi^\Z,\]with $Y_S=Y_{R,R^+}=\Spa(A, A^+)\setminus V(p[\varpi_R])$, where
\[A=W(R^\circ)=\{\sum_{n\geq 0}[x_n]p^n | x_n\in R^\circ\},\quad A^+=\{\sum_{n\geq 0}[x_n]p^n\in A| x_0\in R^+ \},\]and $\varpi_R$ be a pseudo-uniformizer of $R$.
The adic space $X_S$ is the relative version of the Fargues-Fontaine curve. We can also define the scheme
\[X^{sch}_S=\Proj\Big(\bigoplus_{d\geq 0} H^0(X_S, \Ol_{X_S}(d)) \Big).\]Then there exists a map of locally ringed spaces
$X_S\ra X^{sch}_S$. We can define vector bundles on $X_S, X_S^{sch}$ as above, and the relative Robba ring $\wt{\mathcal{R}}_R$. Moreover, we have
\begin{theorem}[\cite{F2, KL} ]
\[Bun_{X^{sch}_S}\simeq Bun_{X_S}\simeq \{\phi-\tr{modules over}\, \wt{\mathcal{R}}_R \}. \]
\end{theorem}
Let $S=\Spa(R,R^+)$ be an affinoid perfectoid space over $\F_p$, and $\varpi_R$ be a pseudo-uniformizer of $R$. We denote
\[\Y_{[0,\infty)}(R,R^+)=\Spa W(R^+)\setminus V([\varpi_R]).\]Then we have a continuous map
\[\kappa: \Y_{[0,\infty)}(R,R^+)\lra [0,\infty),\]the relative version of the map defined previously. With the same notation there, we have
\[Y_{S}=\Y_{(0,\infty)}(R,R^+).\] Let $G$ be a connected reductive group over $\Q_p$. Then as above we can define $G$-bundles on $X_S$, $Y_S=\Y_{ (0,\infty)}(R,R^+)$.  If $G$ is unramified over $\Q_p$, after fixing a reductive model $G_{\Z_p}$ of $G$ over $\Z_p$ we can further define $G$-bundles on $\Y_{[0,\infty)}(R,R^+)$.

If we start with a perfectoid space $S$ over $\Q_p$, then there exits a canonical closed embedding \[x_S: S\hookrightarrow Y_{S^\flat},\] which in turn induces a closed embedding \[x_S: S\hookrightarrow X_{S^\flat},\] cf. \cite{F} 1.4. Here $S^\flat$ is the tilt of $S$ over $\F_p$ in the sense of \cite{Sch1}. Thus we can view $S$ as a Cartier divisor on $X_{S^\flat}$. If $S=\Spa(R,R^+)$ is perfectoid affinoid over $\Q_p$, by \cite{F3} 1.6 we have a corresponding Cartier divisor $D$ on $X_{S^\flat}^{sch}$. The formal completion of $X_{S^\flat}^{sch}$ along $D$ is \[\Spf B^+_{dR, R},\] cf. Proposition 1.33 of \cite{F}.

\subsection{Local $G$-shtukas in mixed characteristic}
Let the notations be as above. From now on,  we assume that $G$ is unramified\footnote{This is not necessary by the methods of \cite{Sch2, Sch3, Sch4, SW17}. Here we restrict to the unramified case to simplify the exposition, which is also sufficient for all what we need.}  over $\Q_p$ and fix a reductive model $G_{\Z_p}$ of $G$ over $\Z_p$.  Let $S=\Spa(R,R^+)$ be an affinoid perfectoid space over $\F_p$, with an
 untilt $S^\sharp$ of $S$. Then there exists a closed embedding $S^\sharp\hookrightarrow \Y_{[0,\infty)}(R,R^+)$.
\begin{definition}[\cite{Sch2} Definition 11.4.1]
A local $G$-shtuka over $S$ with one paw $x: S^\sharp\ra \Y_{[0,\infty)}(R,R^+)$ is a pair $(\E, \phi_\E)$, where
\begin{itemize}
\item	$\E$ is a $G$-bundle over $\Y_{[0,\infty)}(R,R^+)$, 
\item $\phi_\E: \phi^\ast \E\ra \E$ is an isomorphism over $\Y_{[0,\infty)}(R,R^+)\setminus \Gamma_x$, such that along $\Gamma_x$ it is meromorphic. Here $\Gamma_x$ is the image of $x$.
\end{itemize}
\end{definition}
One can then generalize the above notion to define a local $G$-shtuka over a general perfectoid space over $\F_p$.

Let $C$ be a complete algebraically closed extension of $\Q_p$. We have the associated de Rham period ring $B_{dR}^+:=B_{dR,C}^+$ with a fixed uniformizer $\xi\in B_{dR}^+$. Let $B_{dR}=B_{dR}^+[\frac{1}{\xi}], A_{\inf}=W(\Ol_{C^\flat})$.
We have the following various descriptions of local $G$-shtukas with one paw at $C$, in the case $G=\GL_n$.
\begin{theorem}[\cite{Sch2} Proposition 20.1.1; see also \cite{F5}]\label{T:one paw}
	The following categories are equivalent.
	\begin{enumerate}
		\item Shtukas over $\Spa (C^\flat,\Ol_{C^\flat})$ with one paw at $C$.
		\item Pairs $(T, \Xi)$, where $T$ is a finite free $\Z_p$-module, and $\Xi\subset T\otimes B_{dR}$ is a $B_{dR}^+$-lattice.
		\item Breuil-Kisin-Fargues modules over $A_{\inf}$.
		\item Quadruples $(\Fc, \Fc',\beta,T)$, where $\Fc$ and $\Fc'$ are vector bundles on the Fargues-Fontaine curve $X=X_{C^\flat}$ , and $\beta: \Fc|_{X\setminus\{\infty\}}\st{\sim}{\ra}\Fc'_{X\setminus\{\infty\}}$ is an isomorphism, where $\Fc$ is trivial, and $T\subset H^0(X,\Fc)$ is a  $\Z_p$-lattice.  
	\end{enumerate}
	If the paw is minuscule, i.e. we have
	\[ \xi(T\otimes_{\Z_p} B_{dR}^+)\subset \Xi\subset T\otimes_{\Z_p}B_{dR}^+ ,\]then these categories are equivalent to the category of $p$-divisible groups over $\Ol_C$.
\end{theorem}

Recall that a Breuil-Kisin-Fargues module over $A_{\inf}$ is a pair $(M,\phi_M)$, where $M$ is a finite free $A_{\inf}$-module and $\phi_M: (\phi^\ast M)[\xi^{-1}]\st{\sim}{\ra} M[\xi^{-1}]$ is an isomorphism, cf. \cite{Sch2} Definition 11.4.2.

\subsection{Moduli of local $G$-shtukas in mixed characteristic}
We have the following generalizations of Definitions \ref{D:RV} and \ref{D:local Shimura morphism}.
\begin{definition}\label{D:local shtuka datum}
\begin{enumerate}
\item A local shtuka datum is a triple $(G,[b],\{\mu\})$, where
\begin{itemize}
	\item $G$ is a connected reductive group over $\Q_p$,
	\item $\{\mu\}$ is a conjugacy class of cocharacters $\mu: \G_m\ra G_{\ov{\Q}_p}$ over $\ov{\Q}_p$,
	\item $[b]\in B(G,\mu)\subset B(G)$.
\end{itemize}
\item Let $(G_1,[b_1],\{\mu_1\}), (G_2,[b_2],\{\mu_2\})$ be two local shtuka data. A morphism
 \[(G_1,[b_1],\{\mu_1\})\ra (G_2,[b_2],\{\mu_2\})\] is a homomorphism $f: G_1\ra G_2$ of algebraic groups sending $([b_1],\{\mu_1\})$ to $([b_2],\{\mu_2\})$.
\end{enumerate}
\end{definition}
\begin{remark}
\begin{enumerate}
\item By definition, a local Shimura datum $(G,[b],\{\mu\})$ is a local shtuka datum with $\{\mu\}$ minuscule. For a local shtuka datum $(G,[b],\{\mu\})$, the simple factors of $G^{ad}$ can be groups of arbitrary type.
\item In \cite{Sch2}, several $\{\mu\}$'s can be allowed, as in the classical function field case, cf. \cite{Var}.
\item In particular, if $(G,[b],\{\mu\})$ is a local shtuka datum, and $G\ra G'$ is a homomorphism of reductive groups over $\Q_p$, we get the induced $[b'], \{\mu'\}$ such that $(G', [b'], \{\mu'\})$ is also a local shtuka datum.
\item We refer the reader to \cite{HV} for local function field case, where $\{\mu\}$ is replaced by a bound $\hat{Z}$ in the sense of loc. cit. Definition 2.1 (b). 
\end{enumerate}
\end{remark}

Let $(G,[b],\{\mu\})$ be a local shtuka datum. As before, we have the associated local reflex field $E$, and the reductive group $J_b$ over $\Q_p$. Let $F$ be an algebraically closed perfectoid field of characteristic $p$. By Theorem \ref{T: G-bundles} we have a $G$-bundle on $X_{F}$, which is the same as a $\phi$-$G$-module $(\E_b, \phi_{\E_b})$ on $Y_F$, well defined up to isomorphism. We will use freely the notion of diamond in the following, cf. \cite{Sch2} for basic definitions and properties.
We define a functor on the category of perfectoid affinoid algebras over $\ov{\F}_p$ as follows.
\begin{definition}[\cite{Sch2} Definition 19.3.3]
Let $(R,R^+)$ be a perfectoid affinoid $\ov{\F}_p$-algebra together with a map $x: \Spa(R,R^+)^\diamond\ra \Spa(\breve{E})^\diamond$ (which is the same as giving an untilt of $R$ over $\breve{E}$). Let $\Sht (G,b,\mu)\ra \Spa(\breve{E})^\diamond$ be the functor such that for any $\big((R,R^+),x\big)$, 
\[\Sht (G,b,\mu)\big((R,R^+),x\big)=\{\big((\E,\phi_\E),\iota\big)\}/\simeq\]where
\begin{itemize}
	\item $(\E,\phi_\E)$ is a $G$-shtuka over $\Y_{[0,\infty)}(R,R^+)$ with one paw at $x$, such that $(\E,\phi_\E)$ is bounded by $\{\mu\}$.
	\item $\iota: (\E,\phi_\E)|_{[\rho,\infty)}\st{\sim}{\ra}(\E_b,\phi_{\E_b})|_{[\rho,\infty)}$ is an isomorphism for some sufficiently large $\rho$.
\end{itemize}
\end{definition}

The main theorem of \cite{Sch2} is 
\begin{theorem}[Scholze, \cite{Sch2} Theorem 20.3.1]\label{thm:scholze-shtuka}
The functor $\Sht (G,b,\mu)$ is represented by a diamond over $\Spa(\breve{E})^\diamond$.
\end{theorem}
(In \cite{Sch2} the theorem is proved for the case $G=\GL_n$, but one sees immediately that the proof given there works also for the general case. See also \cite{SW17}.)

We want to discuss period maps in this setting. Consider the $B_{dR}^+$-affine Grassmannian $Gr_{G}^{B_{dR}^+} $ over $\Q_p$. This is the functor associating to any perfectoid affinoid $\Q_p$-algebra $(R,R^+)$ the set \[Gr_{G}^{B_{dR}^+}(R,R^+)=\{ (\E, \beta)\}/\simeq \]where $\E$ is a
$G$-torsor over $\Spec B_{dR, R}^+$, and $\beta$ is a trivialization of $\E\otimes_{B_{dR, R}^+}B_{dR, R}$. One can check that $Gr_{G}^{B_{dR}^+}$ is the \'etale sheaf associated to the presheaf
\[(R,R^+)\mapsto G(B_{dR, R})/G(B_{dR, R}^+). \]
Consider the case $(C, C^+)$ with $C|\Q_p$ an algebraically closed perfectoid field. Then we have the Cartan decomposition
\[G(B_{dR, C})=\coprod_{\mu\in X_\ast(T)_+}G(B_{dR, C}^+)\mu(\xi)^{-1} G(B_{dR, C}^+), \]
where $T\subset B\subset G$ is a fixed choice of maximal torus inside a Borel subgroup $B$ of $G$, and $X_\ast(T)_+\subset X_\ast(T)$ is the associated set of dominant cocharacters. Fix a conjugacy class of cocharacters $\{\mu\}$ with the dominant representative $\mu$. Let $E$ be the field of definition of $\{\mu\}$. Consider $Gr_{G,\leq\mu}^{B_{dR}^+}\subset Gr_{G}^{B_{dR}^+}\otimes E$ the sub functor such that
\[ Gr_{G,\leq\mu}^{B_{dR}^+}(R,R^+)=\{ (\E, \xi)\in Gr_{G}^{B_{dR}^+}(R,R^+) |\, \Inv(\E_x, \E_{0x})\leq \mu^{-1}, \, \forall x \in \Spa(R,R^+) \}.\]
This is the analogue of the classical Schubert variety associated to $\{\mu\}$ in the setting of $B_{dR}^+$-affine Grassmannian $Gr_{G}^{B_{dR}^+} $.
There is an action of $J_b(\Q_p)$ on $Gr_{G,\leq \mu}^{B_{dR}^+}$. By abuse of notation, we still denote $Gr_{G,\leq \mu}^{B_{dR}^+}\ra \Spa (\breve{E})^\diamond$ the sheaf base changed over $\Spa(\breve{E})^\diamond$. By Theorem 21.3.6 of \cite{Sch2}, this is a diamond.

There exists an \'etale morphism of diamonds over $\Spa(\breve{E})^\diamond$ (cf. \cite{Sch2} 20.4)
 \[\pi_{dR}: \Sht (G,b,\mu)\lra Gr_{G,\leq\mu}^{B_{dR}^+}.\] When $G=\GL_n$, this morphism can be defined by using Theorem \ref{T:one paw} (4).  Let
 \[Gr_{G,\leq\mu}^{B_{dR}^+,adm} \subset Gr_{G,\leq \mu}^{B_{dR}^+} \] be the image of $\pi_{dR}$. This is an open sub-diamond, and we call it the admissible locus. We have the following description of the admissible locus.
 \begin{proposition}[\cite{Sch2} 20.5, \cite{KL} ]\label{P:admissible}
 Let $(\E,\beta)\in Gr_{G,\leq\mu}^{B_{dR}^+}(R,R^+)$. Then \[(\E,\beta)\in Gr_{G,\leq\mu}^{B_{dR}^+,adm}(R,R^+)\] if and only if one of the following equivalent conditions holds: for any representation $V\in \Rep G$ such that the center of $G$ is mapped into the center of $\GL(V)$, with the associated vector bundle $(\E_V,\beta_V)$,
 \begin{enumerate}
 \item $\forall x\in \Spa(R,R^+)$ the vector bundle $\E_{V,x}$ is semi-stable of slope 0;
 \item $\phi$-module of $\E_V$ is trivial;
 \item $\E_V$ extends to a $\phi$-module over $\wt{\mathcal{R}}_R^{int}$, where $\wt{\mathcal{R}}_R^{int}=\varinjlim_{r}H^0(\Y_{[0,r]},\Ol_\Y)$ is the integral Robba ring.
 \end{enumerate}
 \end{proposition}
 The action of $J_b(\Q_p)$ on $Gr_{G,\leq \mu}^{B_{dR}^+}$ stabilizes the open sub diamond $Gr_{G,\leq\mu}^{B_{dR}^+,adm}$. The period morphism \[\pi_{dR}: \Sht (G,b,\mu)\lra Gr_{G,\leq\mu}^{B_{dR}^+,adm}\] is then $J_b(\Q_p)$-equivariant.
 
 We have the following definition of local systems with additional structures on the diamond $Gr_{G,\leq\mu}^{B_{dR}^+,adm}$, similar to the classical situation.
\begin{definition}
Let $X$ be a diamond, and $G$ be a reductive group over $\Q_p$. Denote by $\Rep G$ the category of rational representations of $G$, and $\Q_p-Loc_{X}$ the category of $\Q_p$-local systems on $X$. Then a $\Q_p$-$G$-local system on $X$ is a tensor functor $\Rep G\ra \Q_p-Loc_{X}$. If $G$ is moreover unramified, and fix a reductive model (over $\Z_p$) $G_{\Z_p}$ of $G$,  then we can define similarly $\Z_p$-$G$-local systems (or better notion: $G_{\Z_p}$-local systems) on $X$.
\end{definition} 
By \cite{KL} Corollary 8.7.10, there exists a $J_b(\Q_p)$-equivariant $\Q_p$-$G$-local system $\mathbb{V}$ over $Gr_{G,\leq\mu}^{B_{dR}^+,adm}$, which realizes $\Sht (G,b,\mu)$ as the functor of the set of $\Z_p$-$G$-local systems in $\mathbb{V}$. In particular, there exists a $J_b(\Q_p)$-equivariant $\Z_p$-$G$-local system $\mathbb{L}$ over $\Sht (G,b,\mu)$. 

Scholze's theorem above (Theorem \ref{thm:scholze-shtuka}) in fact tells us more information.
More precisely, we get a tower of diamonds \[\Big(\Sht (G,b,\mu)_K\Big)_{K\subset G(\Z_p)}\] indexed by open compact subgroups $K\subset G(\Z_p)$ with $\Sht (G,b,\mu)_{G(\Z_p)}=\Sht (G,b,\mu)$, and the group $G(\Q_p)$ acts on this tower $\Big(\Sht (G,b,\mu)_K\Big)_{K\subset G(\Z_p)}$ by Hecke correspondences.  Let $(R,R^+)$ be a perfectoid affinoid $\ov{\F}_p$-algebra together with a map $x: \Spa(R,R^+)^\diamond\ra \Spa(\breve{E})^\diamond$. Then \[\Sht (G,b,\mu)_K\big((R,R^+),x\big)=\{\big((\E,\phi_\E),\iota, \alpha\big)\}/\simeq\]where
\begin{itemize}
	\item $(\E,\phi_\E)$ is a $G$-shtuka over $\Y_{[0,\infty)}(R,R^+)$ with one paw at $x$, such that $(\E,\phi_\E)$ is bounded by $\{\mu\}$.
	\item $\iota: (\E,\phi_\E)|_{[\rho,\infty)}\st{\sim}{\ra}(\E_b,\phi_{\E_b})|_{[\rho,\infty)}$ is an isomorphism for some sufficiently large $\rho$.
	\item $\alpha$ is a $K$-orbit of an isomorphism $\mathbb{L}(\E, \phi_\E)\simeq \mathbb{L}_0$, where $\mathbb{L}(\E, \phi_\E)$ is the $G$-local system associated to $(\E, \phi_\E)$, $\mathbb{L}_0$ is the trivial $G$-local system over $\Y_{[0,\infty)}(R,R^+)$.
\end{itemize}
As \[J_b(\Q_p)\subset Aut(\E_b,\phi_{\E_b}),\]cf. \cite{F} 2.5, 
$J_b(\Q_p)$ acts each $\Sht (G,b,\mu)_K$ by modifying $\iota$, and these actions are compatible when $K$ varies. When the context is clear, we will simply denote $\Sht (G,b,\mu)_K$ by $\Sht_K$.
The cover 
\[\pi_K: \Sht (G,b,\mu)_K\lra \Sht (G,b,\mu)\] is obtained by trivializing $K$-level structures, which is finite \'etale. By trivializing all of $\mathbb{L}$ we get a pro-\'etale cover
\[\pi_\infty: \Sht (G,b,\mu)_\infty\lra \Sht (G,b,\mu).\]We have the following moduli interpretation for $\Sht (G,b,\mu)_\infty$.
Let $(R,R^+)$ be a perfectoid affinoid $\ov{\F}_p$-algebra together with a map $x: \Spa(R,R^+)^\diamond\ra \Spa(\breve{E})^\diamond$. Then \[\Sht (G,b,\mu)_\infty\big((R,R^+),x\big)=\{\big((\E,\phi_\E),\iota, \alpha\big)\}/\simeq\]where
\begin{itemize}
	\item $(\E,\phi_\E)$ is a $G$-shtuka over $\Y_{[0,\infty)}(R,R^+)$ with one paw at $x$, such that $(\E,\phi_\E)$ is bounded by $\{\mu\}$.
	\item $\iota: (\E,\phi_\E)|_{[\rho,\infty)}\st{\sim}{\ra}(\E_b,\phi_{\E_b})|_{[\rho,\infty)}$ is an isomorphism for some sufficiently large $\rho$.
	\item $\alpha: \mathbb{L}(\E, \phi_\E)\simeq \mathbb{L}_0$ is an isomorphism, where as before $\mathbb{L}_0$ is the trivial $G$-local system over $\Y_{[0,\infty)}(R,R^+)$.
\end{itemize}
By construction,
we have an isomorphism of diamonds over $\Spa(\breve{E})^\diamond$
\[\Sht (G,b,\mu)_\infty/K\simeq \Sht (G,b,\mu)_K, \quad \Sht(G,b,\mu)_\infty=\varprojlim_{K}\Sht (G,b,\mu)_K.\]

For any open compact subgroup $K\subset G(\Q_p)$, we know that
the fibers of \[\Sht (G,b,\mu)_K(C, \Ol_C)\lra Gr_{G,\leq\mu}^{B_{dR}^+,adm}(C, \Ol_C)\] are in bijection with $G(\Q_p)/K$. We remark that it should be possible to define a notion of \'etale fundamental group for the diamond $Gr_{G,\leq\mu}^{B_{dR}^+,adm}$ as \cite{dJ1}, so that the $\Q_p$-$G$-local system $\mathbb{V}$ on $Gr_{G,\leq\mu}^{B_{dR}^+,adm}$ can be described in term of a collection of representations
\[\pi_1(Gr_{G,\leq\mu}^{B_{dR}^+,adm}, \ov{x})\lra G(\Q_p),\]for the geometric point $\ov{x}$ runs through each connected component of $Gr_{G,\leq\mu}^{B_{dR}^+,adm}$. See \cite{Sch3}.

At the infinite level, there exists a Hodge-Tate period map (cf. \cite{F} p.38; see also \cite{Han} Theorem 5.4)
\[\pi_{HT}: \Sht (G,b,\mu)_\infty\lra Gr_{G,\leq \mu^{-1}}^{B_{dR}^+},\] where $Gr_{G,\leq \mu^{-1}}^{B_{dR}^+}\subset Gr_{G}^{B_{dR}^+}\otimes E$ is the 
Schubert diamond associated to $\{\mu^{-1}\}$. We can make a little precise on the image of $\pi_{HT}$. By \cite{CS} Corollary 3.5.2 there is a natural map
\[\E:  Gr_{G}^{B_{dR}^+}(R,R^+)\lra  \Bun_{G, X_{R^\flat, R^{+\flat}}}.\]
Take $(R,R^+)=(C,\Ol_C)$ with $C|\Q_p$ complete and algebraically closed. By Theorem \ref{T: G-bundles},  we get a map $b(-): Gr_{G}^{B_{dR}^+}(C,\Ol_C)\ra B(G)$. By \cite{CS} Proposition 3.5.3, when restricting to $x\in Gr_{G,\leq \mu^{-1}}^{B_{dR}^+}(C, \Ol_C)$, one has \[b(x)\in B(G,\mu).\] Then for any $[b']\in B(G,\mu)$, we get a locally closed sub diamond \[Gr_{G,\leq \mu^{-1}}^{B_{dR}^+,b'}\subset Gr_{G,\leq \mu^{-1}}^{B_{dR}^+},\] such that the underling topological space $|Gr_{G,\leq \mu^{-1}}^{B_{dR}^+,b'}|$ is the fiber over $[b']$ under the above map $b(\cdot)$. Consider $[b']=[b]$ as in the local shtuka datum. Then by construction,  one has
\[ \pi_{HT}: \Sht (G,b,\mu)_\infty(C,\Ol_C)\lra Gr_{G,\leq \mu^{-1}}^{B_{dR}^+,b}(C,\Ol_C),\] 
for any $(C,\Ol_C)$ with $C|\Q_p$ complete and algebraically closed. That is, $\pi_{HT}$ factors through $Gr_{G,\leq \mu^{-1}}^{B_{dR}^+,b}$.

In summary, we get two period morphisms
\[\xymatrix{
	& \Sht (G,b,\mu)_\infty \ar[ld]_{\pi_{dR}} \ar[rd]^{\pi_{HT}} &\\
	Gr_{G,\leq\mu}^{B_{dR}^+,adm} & & Gr_{G,\leq \mu^{-1}}^{B_{dR}^+,b},
}\]
and the period morphism $\pi_{dR}$ factors through $\Sht(G,b,\mu)$.

\begin{remark}\label{R:moduli shtuka}
		In \cite{F} 8.2.1, there is an alternative construction of the diamond $\Sht (G,b,\mu)_\infty$. 
\end{remark}

By construction, a morphism
$(G_1,[b_1],\{\mu_1\})\ra (G_2,[b_2],\{\mu_2\})$ of local shtuka data induces a morphism of diamonds
\[\Sht(G_1,b_1,\mu_1)\lra \Sht(G_2,b_2,\mu_2).\]More generally, we have morphisms
\[\Sht(G_1,b_1,\mu_1)_{K_1}\lra \Sht(G_2,b_2,\mu_2)_{K_2}\]if $K_1$ is mapped into $K_2$ under $G_1\ra G_2$.

The above functoriality enables us to apply the Tannakian formalism. As before, we assume that $G$ is unramified over $\Q_p$. Consider now an embedding $G\hookrightarrow \GL_n$, then $([b], \{\mu\})$ induces $([b'], \{\mu'\})$, so that $(\GL_n, [b'], \{\mu'\})$ forms a local shtuka datum, and we get a morphism of local shtuka data $(G,[b],\{\mu\})\ra (\GL_n,[b'],\{\mu'\})$.
The following proposition is the local analogue of Deligne's theorem for Shimura varieties.
\begin{proposition}\label{P:embedding}
In the above setting, for any $K\subset G(\Z_p)$, there exists a $K'\subset \GL_n(\Z_p)$ such that there exists a natural closed embedding of diamonds
	\[\Sht(G,b,\mu)_K\hookrightarrow \Sht(\GL_n,b',\mu')_{K'}. \]
	The induced embedding of diamonds
	\[\Sht (G,b,\mu)_\infty \hookrightarrow \Sht (\GL_n,b',\mu')_\infty\] is compatible with the de Rham and Hodge-Tate period morphisms on both sides.
\end{proposition}
\begin{proof}
It suffices to prove that we have a closed embedding of diamonds
\[\Sht (G,b,\mu)_\infty \hookrightarrow \Sht (\GL_n,b',\mu')_\infty.\] This is clear from the construction above. Moreover, we have a closed embedding 
\[ Gr_{G,\leq\mu}^{B_{dR}^+,adm} \hookrightarrow Gr_{\GL_n,\leq\mu'}^{B_{dR}^+,adm},\] and the following diagram on de Rham period maps is commutative
\[\xymatrix{ \Sht (G,b,\mu)_\infty \ar[d]^{\pi_{dR}} \ar@{^{(}->}[r] & \Sht (\GL_n,b',\mu')_\infty \ar[d]^{\pi_{dR}}\\
	Gr_{G,\leq\mu}^{B_{dR}^+,adm} \ar@{^{(}->}[r] & Gr_{\GL_n,\leq\mu'}^{B_{dR}^+,adm}.
	}\]
We have also the following commutative diagram on Hodge-Tate period maps
\[\xymatrix{ \Sht (G,b,\mu)_\infty \ar[d]^{\pi_{HT}}\ar@{^{(}->}[r] & \Sht (\GL_n,b',\mu')_\infty \ar[d]^{\pi_{HT}}\\
	Gr_{G,\leq \mu^{-1}}^{B_{dR}^+,b} \ar@{^{(}->}[r] & Gr_{\GL_n,\leq (\mu')^{-1}}^{B_{dR}^+,b'}.
}\]
\end{proof}

\subsection{Moduli of local $G$-shtukas and affine Deligne-Lusztig varieties}\label{subsection: Moduli and ADLV}
Let $(G, [b], \{\mu\})$ be a local shtuka datum. Recall that we assume $G$ is unramified.
We want to compare the moduli space of local $G$-shtukas $\Sht(G,b,\mu)$ and the affine Deligne-Lusztig variety $X_{\leq \mu}^G(b)$ associated to $(G, [b], \{\mu\})$ as in section 2.

Let $(C,\Ol_C)$ be an affinoid perfectoid field of characteristic $p$ with an untilt $C^\sharp$ of $C$. Let $k$ be the residue field of $\Ol_C$.
We have a $J_b(\Q_p)$-equivariant morphism of sets
\[sp=sp_{(G,b,\mu)}:  \Sht (G,b,\mu)(C,\Ol_C)\lra X_{\leq\mu}^G(b)(k).\]Indeed, consider first the case $G=\GL_n$, we have
\[\Sht(G,b,\mu)(C,\Ol_C)=\{\big((\E,\phi_\E),\iota\big)\}/\simeq\] 
with $\big((\E,\phi_\E),\iota\big)$ a shtuka over $\Spa (C,\Ol_C)$ with one paw at $C^\sharp$. By Theorem \ref{T:one paw}, there exists a Breuil-Kisin module $(M,\phi)$ over $A_{\inf}=W(\Ol_C)$. Let \[(M\otimes_{A_{\inf}}W(k),\phi)\]
 be the associated Dieudonn\'e module. This defines a point in $X_{\leq\mu}^G(b)(k)$. This construction is compatible with the $J_b(\Q_p)$ actions on both sides. For the general case, we apply the Tannakian formalism: take any embedding $(G,[b],\{\mu\})\ra (\GL_n,[b'],\{\mu'\})$, then we have a commutative diagram
\[\xymatrix{\Sht(G,b,\mu)(C,\Ol_C)\ar[d]^{sp_{G,b.\mu}}\ar@{^{(}->}[r] & \Sht(\GL_n,b',\mu')(C,\Ol_C)\ar[d]^{sp_{\GL_n,b',\mu'}}\\
	X_{\leq\mu}^G(b)(k)\ar@{^{(}->}[r] & 	X_{\leq\mu'}^{\GL_n}(b')(k).
	}\]

Recall that we have the map  $\omega_G: G(L)/G(W)\ra \pi_1(G)$. In the rest of this subsection we will only consider $X_{\leq\mu}^G(b)$ as a subset of $G(L)/G(W)$. Restricting $\omega_G$ to $X_{\leq\mu}^G(b)$, it gives \[ \omega_G: X_{\leq\mu}^G(b)\lra c_{b,\mu}\pi_1(G)^\Gamma.\]
Recall that as in subsection 2.2, after replacing $b$ by $b'$ we may assume $c_{b,\mu}=1$. On the other hand,  restricting $\omega_G$ to $G(\Q_p)/G(\Z_p)$ we get \[\omega_G: G(\Q_p)/G(\Z_p)\lra  \pi_1(G)^\Gamma.\]We can regard the quotient set $G(\Q_p)/G(\Z_p)$ as certain subset of $\Sht(G,b,\mu)(C,\Ol_C)$, see below.
The above two maps are related by the following reduction map.
Recall that we have the period map $\pi_{dR}: \Sht(G,b,\mu)\ra Gr_{G,\leq \mu}^{B_{dR}^+,adm}$, which is $J_b(\Q_p)$-equivariant. Take any point $y\in Gr_{\leq \mu}^{B_{dR}^+,adm}(C,\Ol_C)$, then the fiber $\pi_{dR}^{-1}(y)$ is in bijection with $G(\Q_p)/G(\Z_p)$ once we fix a point $x\in \pi_{dR}^{-1}(y)$. For $g\in J_b(\Q_p)$, we take the point $gx\in \pi_{dR}^{-1}(gy)$ to identify $\pi_{dR}^{-1}(gy)$ with $G(\Q_p)/G(\Z_p)$. In this way we can define an action of $J_b(\Q_p)$ on $G(\Q_p)/G(\Z_p)$.
\begin{lemma}\label{L:specilization}
	There is a $J_b(\Q_p)$-equivariant map \[G(\Q_p)/G(\Z_p)\ra X_{\leq\mu}^G(b), \quad g\mapsto g_0,\]such that $\omega_G(g)=\omega_G(g_0)$.
\end{lemma}
\begin{proof}
Fix any point $x\in \Sht(G,b,\mu)(C,\Ol_C)$. Then we have an injection 
\[G(\Q_p)/G(\Z_p)\ra \Sht(G,b,\mu)(C,\Ol_C)\]which identifies $G(\Q_p)/G(\Z_p)$ with the Hecke orbit $\pi_{dR}^{-1}(\pi_{dR}(x))$ of $x$. For any $g\in J_b(\Q_p)$, the choice of the point $gx$ to identify $G(\Q_p)/G(\Z_p)$ with $\pi_{dR}^{-1}(\pi_{dR}(gx))$ shows that the injection $G(\Q_p)/G(\Z_p)\ra \Sht(G,b,\mu)(C,\Ol_C)$ is $J_b(\Q_p)$-equivariant. The composite
\[G(\Q_p)/G(\Z_p)\ra \Sht(G,b,\mu)(C,\Ol_C)\ra X_{\leq\mu}^G(b)\]gives the desired map, which is $J_b(\Q_p)$-equivariant, since by the above construction, the specialization map $\Sht(G,b,\mu)(C,\Ol_C)\ra X_{\leq\mu}^G(b)$ is $J_b(\Q_p)$-equivariant. 
The second assertion follows by the same argument as that in the proof of Lemma 1.2.18 of \cite{Ki2}, by applying Theorem \ref{T:one paw} (and Tannakian formalism) instead of subsection 1.1 of \cite{Ki2}.
\end{proof}

\begin{remark} Consider the composite map $G(\Q_p)/G(\Z_p)\ra X_{\leq\mu}^G(b)\ra \pi_1(G)^\Gamma$. Then this is surjective by Lemma \ref{L:surjective} (1). In \cite{Ki2} Proposition 1.2.23, Kisin proved a stronger result: the map \[G(\Q_p)/G(\Z_p)\lra \pi_0(X_{\leq\mu}^G(b))\]is surjective if $(G, [b],  \{\mu\})$ is an unramified local Shimura datum of Hodge type.
\end{remark}

The following is an analogue of Lemma 2.4.1 and Corollary 2.4.2 of \cite{CKV}, see also Proposition \ref{P:cartesian}.
\begin{proposition}\label{P:shtuka cart}
	Let $Z\subset Z_G$ be a central sub group and $G'=G/Z$, with the induced $[b']$ and $\{\mu'\}$ such that $(G',[b'], \{\mu'\})$ is a local shtuka datum. Then
we have a cartesian diagram
\[\xymatrix{\Sht (G,b,\mu)(C,\Ol_C)\ar[r]\ar[d]&\Sht (G',b',\mu')(C,\Ol_C)\ar[d] \\
X_{\leq\mu}^G(b)\ar[r]& X_{\leq\mu'}^{G'}(b').}\]
In particular, the induced diagram 
\[\xymatrix{\Sht (G,b,\mu)(C,\Ol_C)\ar[r]\ar[d]&\Sht (G',b',\mu')(C,\Ol_C)\ar[d] \\
\pi_1(G)^\Gamma	\ar[r]& \pi_1(G')^\Gamma}\] is also cartesian.
\end{proposition}
\begin{proof}
Firstly, we have the natural identification $Gr_{G,\leq \mu}^{B_{dR}^+,adm}=Gr_{G',\leq \mu'}^{B_{dR}^+,adm}$, since by construction $Gr_{G,\leq \mu}^{B_{dR}^+,adm}$ depends only on the adjoint local shtuka datum $(G^{ad}, [b^{ad}], \{\mu^{ad}\})$.
Now consider the following commutative diagram
\[\xymatrix{\Sht (G,b,\mu)(C,\Ol_C)\ar[rr]\ar[rd]_{\pi_{G,dR}} &  &\Sht (G',b',\mu')(C,\Ol_C)\ar[ld]^{\pi_{G',dR}} \\
	&Gr_{G,\leq \mu}^{B_{dR}^+,adm}(C,\Ol_C).&}
\]
For any point $x\in Gr_{G,\leq \mu}^{B_{dR}^+,adm}(C,\Ol_C)$, the above horizontal map induces a map on fibers
\[G(\Q_p)/G(\Z_p)\ra G'(\Q_p)/G'(\Z_p),\]thus it suffices to show that the following diagram is cartesian
\[\xymatrix{G(\Q_p)/G(\Z_p) \ar[r]\ar[d]&G'(\Q_p)/G'(\Z_p)\ar[d] \\
	X_{\leq\mu}^G(b)\ar[r]& X_{\leq\mu'}^{G'}(b'),}\]
where the vertical maps are those constructed in Lemma \ref{L:specilization}.
Consider the following diagram
\[\xymatrix{G(\Q_p)/G(\Z_p) \ar[r]\ar[d] & G'(\Q_p)/G'(\Z_p)\ar[d] \\
	X_{\leq\mu}^G(b)\ar[r]\ar[d]^{\omega_G} & X_{\leq\mu'}^{G'}(b')\ar[d]^{\omega_{G'}}\\
\pi_1(G)^\Gamma\ar[r]&\pi_1(G')^\Gamma.}
\]
We know that the lower square is cartesian, cf. Proposition \ref{P:cartesian}, and by Lemma \ref{L:cartesian} 
\[ \xymatrix{G(\Q_p)/G(\Z_p) \ar[r]\ar[d]&G'(\Q_p)/G'(\Z_p)\ar[d] \\
	\pi_1(G)^\Gamma\ar[r]&\pi_1(G')^\Gamma}
	\]is also cartesian. Therefore the upper square is cartesian.
\end{proof}

\subsection{Local Shimura varieties as moduli of local $G$-shtukas}\label{Section:diamond}
We return to the setting of Definition \ref{D:RV}.
The following strengthened version of Theorem \ref{thm:scholze-shtuka}, which may be viewed as a partial solution of Conjecture \ref{C:RV} (as we do not give information on the desired Weil descent datum), is implied by the results in \cite{CS, Sch2, Sch3}. See also \cite{SW17} Lecture 24 and \cite{CFS} section 3.
Recall that by \cite{Sch2} Proposition 10.2.8, there is a fully faithful functor $X\mapsto X^\diamond$ from the category of normal rigid analytic spaces over $k$ to the category of diamonds over $\Spa (k)^\diamond$ for any non-archimedean field $k$ of characteristic 0. 
\begin{theorem}\label{thm:local Shimura}
Let $(G,[b],\{\mu\})$ be a  local Shimura datum. Assume that $G$ is unramified. Then there exists a tower of rigid analytic spaces over $\tr{Sp} \breve{E}$
	\[(\M_K)_{K},\] where $K$ runs through all open compact subgroups of $G(\Z_p)$, with the following properties:
	\begin{enumerate}
		\item the group $J_b(\Q_p)$ acts on each space $\M_K$,
		\item the group $G(\Q_p)$ acts on the tower $(\M_K)_{K}$ as Hecke correspondences,
		\item there exists a compatible system of \'etale and partially proper period maps \[\pi_K: \M_K\ra \Fl\ell_{G, \mu}^{adm}\] which is equivariant for the action of $J_b(\Q_p)$, where $\Fl\ell_{G, \mu}^{adm}\subset \Fl\ell_{G, \mu}$ is the open subspace defined in \cite{Ra2} A.6 (see also Proposition \ref{P:admissible} and \cite{CFS}),
		\item for any $K$, we have an isomorphism of diamonds $\M_K^\diamond\simeq \Sht_K$.
	\end{enumerate}
\end{theorem}
\begin{proof}
Consider the Bialynicki-Birula morphism \[Gr_{G,\mu}^{B_{dR}^+}\lra \Fl\ell_{G,\mu},\] cf. \cite{CS} Proposition 3.4.3. Since $\mu$ is minuscule, it is an isomorphism, cf. \cite{CS} Theorem 3.4.5, which induces an isomorphism 
\[Gr_{G,\mu}^{B_{dR}^+,adm}\st{\sim}{\lra} \Fl\ell_{G,\mu}^{adm,\diamond}.\]
The tower $(\Sht_K)_K$ is constructed out of a $J_b(\Q_p)$-equivariant $\Q_p-G$-local system $\mathbb{V}$ over $Gr_{G,\mu}^{B_{dR}^+,adm}$, which realizes $\Sht (G,b,\mu)$ as the functor of the set of $G_{\Z_p}$-local systems in $\mathbb{V}$. Since $Gr_{G,\mu}^{B_{dR}^+,adm}\simeq\Fl\ell_{G,\mu}^{adm,\diamond}$, there exists a corresponding $J_b(\Q_p)$-equivariant $\Q_p$-$G$-local system over $\Fl\ell_{G,\mu}^{adm}$ which we still denote by $\mathbb{V}$. Here we use the fact that the categories of \'etale $\Z_p$-local systems and $\Q_p$-local systems on an adic space $X$ are eauivalent to the corresponding categories on the pro-\'etale site $X_{\tr{pro\'et}}$, cf. \cite{KL} Lemma 9.1.11. Therefore we get a tower of rigid analytic spaces $(\M_K)_K$ with the properties listed as in the theorem.
\end{proof}
\begin{remark}
In the above situation,
it is natural to conjecture that there exists a preperfectoid space $\M_\infty$ over $\breve{E}$ such that $\M_\infty\sim \varprojlim_K\M_K$ and $\M_\infty^\diamond=\Sht_\infty$. We will see that this is true if $(G,[b],\{\mu\})$  is unramified of abelian type, cf. Corollary \ref{C:local ab}. This is the local analogue of the fact that Shimura varieties of abelian type with infinite level at $p$ are perfectoid, cf. \cite{Sh1}.
\end{remark}

Finally, we return to Rapoport-Zink spaces of abelian type. In particular we assume $p>2$ in the rest of this section.
\begin{corollary}\label{C:local ab}
Let $(G,[b],\{\mu\})$ be an unramified local Shimura datum of abelian type. For any open compact subgroup $K\subset G(\Z_p)$, let $\M_K$ and $\M_K'$ be the rigid analytic spaces over $\breve{E}$ constructed in subsection \ref{subsection:generic} and Theorem \ref{thm:local Shimura} respectively. Then we have an isomorphism of rigid analytic spaces over $\breve{E}$ \[\M_K\simeq \M_K'.\] In particular, we get isomorphisms of  diamonds over $\Spa (\breve{E}) ^\diamond$
\[\M_K^\diamond \simeq \Sht_K, \]and
\[\M_\infty^\diamond \simeq \Sht_\infty,\]with compatible period morphisms on both sides. In particular, the Hodge-Tate period map $\pi_{HT}$ in Proposition \ref{P:local perfectoid}  factors through $\pi_{HT}: \M_{\infty}\lra \Fl\ell_{G,\mu^{-1}}^b$.
\end{corollary}
\begin{proof}
We first prove the case $(G,[b],\{\mu\})$ is of Hodge type. This follows exactly as the proof of \cite{Sch2} Theorem 19.4.5. Moreover, we have the following cartesian diagram
\[\xymatrix{\M(G,b,\mu)_K^\diamond \ar@{^{(}->}[d] \ar[r]^\sim & \Sht(G,b,\mu)_K\ar@{^{(}->}[d]\\
\M(\GL_n,b',\mu')_{K'}^\diamond \ar[r]^\sim	 & \Sht(\GL_n,b',\mu')_{K'}
	}\]

Now assume that $(G,[b],\{\mu\})$ is of abelian type.
We can apply Propositions \ref{P:local system} and \ref{P:shtuka cart}, and compare the construction of $\M(G,b,\mu)_K$ with that of $\Sht(G,b,\mu)_K$. Here as above, we use the fact that the categories of \'etale $\Z_p$-local systems and $\Q_p$-local systems on an adic space $X$ are eauivalent to the corresponding categories on the pro-\'etale site $X_{\tr{pro\'et}}$, cf. \cite{KL} Lemma 9.1.11.
\end{proof}

Let $(G,[b],\{\mu\})$ be a local Shimura datum with $G$ unramified. By Theorem \ref{thm:local Shimura}, there exists a tower of local Shimura varieties $\Big(\M(G,b,\mu)_K\Big)_K$ over $\tr{Sp} \breve{E}$ as conjectured by Rapoport-Viehmann. Take an embedding $G\hookrightarrow \GL_n$. Then we get an induced triple $(\GL_n,[b'],\{\mu'\})$, which is a local shtuka datum. The following corollary is now a consequence of Proposition \ref{P:embedding} and Theorem \ref{thm:local Shimura}. 
\begin{corollary}
For any $K\subset G(\Z_p)$, there exists a $K'\subset \GL_n(\Z_p)$ such that there exists a natural closed embedding of diamonds
\[\M(G,b,\mu)_K^\diamond\hookrightarrow \Sht(\GL_n,b',\mu')_{K'}. \]
\end{corollary}

\begin{remark}\label{Remarks}
	\begin{enumerate}
\item Let $(G,[b],\{\mu\})$ be an unramified local Shimura datum of Hodge type, with the associated Rapoport-Zink spaces $\M_K$ and the moduli spaces of local $G$-shtukas $\Sht_K$. The isomorphism of diamonds over $\Spa (\breve{E})^\diamond$ \[\M_K^\diamond \simeq \Sht_K\] indicates the magic ``switching $p$-divisible groups with additional structures to local $G$-shtukas''.
\item If $(G,[b],\{\mu\})$ is  a general local Shimura datum, e.g.  an unramified local Shimura datum of abelian type but not of Hodge type, then we do not have $p$-divisible groups any more. However, via $\M_K^\diamond \simeq \Sht_K$, the local Shimura varieties $\M_K$ can be viewed as moduli of local $G$-shtukas.
\item Corollary \ref{C:local ab} should be upgraded to the integral level as \cite{SW17} Theorem 25.1.2 and Corollary 25.1.3. Namely, we should get an isomorphisms of $v$-sheaves over $\tr{Spd} W$\[\breve{\M}^\diamond\simeq \M^{\tr{int}},\]where $\M^{\tr{int}}$ is the moduli functor introduced in \cite{SW17} Definition 25.1.1 for the unramified local Shimura datum $(G,[b],\{\mu\})$. This could be done similarly by the methods of \cite{SW17} Lecture XXV, and we will leave the details to the reader.
Therefore, at the end we would have a canonical moduli interpretation for general $\breve{\M}$, compared with subsection \ref{subsection moduli}.
\end{enumerate}
\end{remark}

\begin{remark}
We refer to \cite{RV} sections 6,7,8 and \cite{F} section 8 for the discussions on the conjectures on the realizations of local Langlands correspondences and local Jacquet-Langlands correspondences in the $\ell$-adic cohomology of the tower $(\M_K)_K$ or $(\Sht_K)_K$.
\end{remark}

\section{Rapoport-Zink uniformization for Shimura varieties of abelian type}
We discuss some global applications in this section.  As \cite{RZ} chapter 6 and \cite{Kim2}, we apply our construction of the formal schemes $\breve{\M}(G,b,\mu)$ to prove a uniformization theorem for Kisin's integral canonical models of Shimura varieties of abelian type \cite{Ki1}. Throughout this section, we assume $p>2$.

\subsection{Integral canonical models for Shimura varieties of abelian type}
Let $(G,X)$ be a Shimura datum of abelian type, i.e. there exists a Shimura datum of Hodge type $(G_1, X_1)$ together with a central isogeny $G_1^{der}\ra G^{der}$, such that it induces an isomorphism of the associated adjoint Shimura data $(G_1^{ad}, X_1^{ad})\simeq (G^{ad},X^{ad})$.
Fix a prime $p>2$. Assume that $G$ is unramified at $p$ from now on. By Lemma 3.4.13 of \cite{Ki1}, we can find a Shimura datum of Hodge type $(G_1, X_1)$ satisfying the above and $G_1$ is unramified at $p$. Let $E$ be the local reflex field of $(G,X)$ for some place over $p$. In the following we will only consider the open compact subgroups $K\subset G(\A_f)$ in the form $K=K_pK^p$ with $K_p=G(\Z_p)$.
\begin{theorem}[\cite{Ki1} Theorem 3.4.10, Corollary 3.4.14]
	With the above notation and assumption, for any sufficiently small open compact subgroup $K^p\subset G(\A_f^p)$, there exists an integral canonical smooth model \[S_K(G,X)\] of $\Sh_K(G,X)$ over $\Ol_E$. When $K^p$ varies, the prime to $p$ Hecke action on $\big(\Sh_K(G,X)\big)_K$ extends to \[\big(S_K(G,X)\big)_K.\]
\end{theorem}

It will be useful to review how these integral models are constructed, cf. \cite{Ki1} 2.3 and 3.4. 
\subsubsection{Case $(G,X)$ of Hodge type} Take an embedding of Shimura data $(G,X)\hookrightarrow (\GSp, S^\pm)$. Let $K=K_pK^p\subset G(\A_f)$ be an open compact subgroup with $K_p=G(\Z_p)$.
Take an open compact subgroup $K'=K'_pK^{'p}$ with $K'_p=\GSp(\Z_p)$, such that $K\subset K'$ and we have an closed immersion
\[\Sh_K(G,X)\hookrightarrow \Sh_{K'}(\GSp, S^\pm)_E, \] where $E$ is the local reflex field for $(G,X)$. For $\Sh_{K'}(\GSp, S^\pm)$ we have the integral canonical model $S_{K'}(\GSp, S^\pm)$. Consider the Zariski closure $S^-_K(G,X)$ of $\Sh_K(G,X)_E$ in $S_{K'}(\GSp, S^\pm)_{\Ol_E}$. Then $S_K(G,X)$ is defined as the normalization of $S^-_K(G,X)$. In particular we have a finite morphism
\[ S_K(G,X)\ra S^-_K(G,X)\subset  S_{K'}(\GSp, S^\pm)_{\Ol_E}.\]

It will be useful to review some further structures for the integral canonical model $S_K(G,X)$.
Let $T$ be a scheme over $\Ol_E$. Attached to each point $x\in S_{K}(G,X)(T)$ we have a triple \[(A_x, \lambda_x, \varepsilon_{x,K}^p),\] where $(A_x, \lambda_x)$ is the polarized abelian scheme up to prime to $p$ isogeny coming from pullback of the universal polarized abelian scheme over $S_{K'}(\GSp, S^\pm)$, and \[\varepsilon_{x,K}^p\in \Gamma(T, \Isom(V_{\A_f^p}, \wh{V}^p(A_x)_\Q)/K^p)\] is the (promoted) $K$-level structure coming from the $K'$-level structure $\varepsilon_{x,K'}^p$ on $A_x$, cf. \cite{Ki1} 3.4.2. The triple $(A_x, \lambda_x, \varepsilon_{x,K'}^p)$ can be viewed as the polarized abelian scheme with level structure attached to the $T$-point of $S_{K'}(\GSp, S^\pm)$ induced by $x$. Let $(s_\alpha)$ be a finite collection of tensors which cut off the inclusion $G\subset \GL(V)$. As explained in 1.3.6 of \cite{Ki2}, there exist de Rham tensors $s_{\alpha, dR, x}$ and $\ell$-adic \'etale tensors $(s_{\alpha, l, x})_{l\neq p}$ on the first relative de Rham cohomology and the first $\ell$-adic cohomology of $A_x$ respectively. The level $\varepsilon_{x,K}^p$ takes $s_\alpha$ to $(s_{\alpha, l, x})_{l\neq p}$. 

If $T=\Spec\, k$ where $k\subset\ov{\F}_p$ is a subfield containing the residue field $k_E$ of $\Ol_E$, then there exists crystalline Tate tensors $(s_{\alpha, 0, x})$ on the first crystalline cohomology of $A_x$. If $x$ is the specialization of a point $\wt{x}$ over $F$ with $F|E$ an extension, then there exist $p$-adic \'etale tensors $(s_{\alpha, p, \wt{x}})$ on the first $p$-adic \'etale cohomology of $A_{\wt{x}}$, and $(s_{\alpha, 0, x})$ and $(s_{\alpha, p, \wt{x}})$ are related by the $p$-adic comparison theorem, cf. Proposition 1.3.7 of \cite{Ki2}. By Corollary 1.3.11 of \cite{Ki2} the data \[(A_x, \lambda_x, \varepsilon_{x,K}^p, (s_{\alpha, 0, x}))\] uniquely determines the point $x\in S_{K}(G,X)(k)$. Sometimes we will write $s_{\alpha, 0, x}$ as $t_{\alpha,x}$ to be compatible with our previous notation on crystalline Tate tensors on $p$-divisible groups.

\subsubsection{Case $(G,X)$ of abelian type} Take a Shimura datum of Hodge type $(G_1,X_1)$ which is unramified at $p$, together with a central isogeny $G_1^{der}\ra G^{der}$, such that it induces an isomorphism of the associated adjoint Shimura datum $(G_1^{ad}, X_1^{ad})\simeq (G^{ad},X^{ad})$. Let $K=K_pK^p\subset G(\A_f)$ be an open compact subgroup with $K_p=G(\Z_p)$. The integral model $S_K(G,X)$ is constructed as the quotient \[S_{K_p}(G,X)/K^p,\] where $S_{K_p}(G,X)$ is an integral model over $\Ol_E$ of the pro-scheme \[\Sh_{K_p}(G,X)=\varprojlim_{K^p}\Sh_{K_pK^p}(G,X).\] The scheme $S_{K_p}(G,X)$ is constructed as follows. Fix a connected component $X^+\subset X$. We get the induced connected component $\Sh_K(G, X)_\C^+$ of the complex Shimura variety as usual. By \cite{Ki1} Proposition 2.2.4 it is defined over $L$. 
Consider the connected component \[\Sh_{K_{1p}}(G_1,X_1)^+=\varprojlim_{K^p_1}\Sh_{K_{1p}K^p_1}(G_1,X_1)^+\] of $\Sh_{K_{1p}}(G_1,X_1)=\varprojlim_{K^p_1}\Sh_{K_{1p}K^p_1}(G_1,X_1)$, where $K_{1p}=G_1(\Z_p)$. Let $S_{K_{1p}}(G_1,X_1)^+$ be the Zariski closure of $\Sh_{K_{1p}}(G_1,X_1)^+$ in $S_{K_{1p}}(G_1,X_1)$ over $W=W(\ov{\F}_p)$. Write $Z=Z_G$. The above integral model $S_{K_p}(G,X)$ of $\Sh_{K_p}(G,X)$ over $W$ is given by \[S_{K_p}(G,X)=[\Ac(G_{\Z_{(p)}})\times S_{K_{1p}}(G_1,X_1)^+ ]/ \Ac(G_{1\Z_{(p)}})^\circ,\]where
\[\Ac(G_{\Z_{(p)}})=G(\A_f^p)/Z(\Z_{(p)})^-\ast_{G(\Z_{(p)})_+/Z(\Z_{(p)})} G^{ad}(\Z_{(p)})^+\] and 
\[\Ac(G_{\Z_{(p)}})^\circ=G(\Z_{(p)})_+^-/Z(\Z_{(p)})^-\ast_{G(\Z_{(p)})_+/Z(\Z_{(p)})} G^{ad}(\Z_{(p)})^+;\] similarly we have $\Ac(G_{1\Z_{(p)}})$ and $\Ac(G_{1\Z_{(p)}})^\circ$, see \cite{Ki1} 3.3.2. The scheme $S_{K_p}(G,X)$ descends to $\Ol_E$ and gives the integral canonical model of $\Sh_{K_p}(G,X)=\varprojlim_{K^p}\Sh_{K_pK^p}(G,X)$, see the proof of loc. cit. Theorem 3.4.10.

\subsection{Newton stratification of the special fibers}\label{Section:newton}
We keep the notations as above. We will work over $\ov{\F}_p$ in this subsection. By abuse of notation, denote the special fiber of $S_K=S_K(G,X)$ over $\ov{\F}_p$ by $\ov{S}_K$ for simplicity. In this subsection, we will write an element of $B(G_{\Q_p})$ simply by $b$, and $B(G,\mu)=B(G_{\Q_p},\mu)$ as usual. In \cite{SZ}, we proved the following results.
\begin{theorem}
	\begin{enumerate}
		\item For any $b\in B(G,\mu)$, there exists a non empty locally closed subset $\ov{S}_K^b\subset \ov{S}_K$, which we view as a subscheme of $\ov{S}_K$ with its reduced structure, such that set theoretically we have \[ \ov{S}_K=\coprod_{b\in B(G,\mu)}\ov{S}_K^b.\]
		\item For any $b\in B(G,\mu)$, the Zariski closure of $\ov{S}_K^b$ in $\ov{S}_K$ is $\coprod_{b'\leq b}\ov{S}_K^{b'}$.
	\end{enumerate}
\end{theorem}

For $b\in B(G,\mu)$, we call the subschemes $\ov{S}_K^b$ as the Newton strata of $\ov{S}_K$. If $(G,X)$ is of Hodge type, then the existence of the Newton stratification is implied by \cite{RR}, see also \cite{Ham} 2.3 and \cite{W} 5.2. 

For later use, we briefly review the construction of the Newton stratification. If $(G,X)$ is of Hodge type, it is constructed by the associated $p$-divisible groups with crystalline Tate tensors. We now assume that $(G,X)$ is of abelian type. In this case, let $(G_1, X_1)$ be an unramified Shimura datum of Hodge type $(G_1,X_1)$, together with a central isogeny $G_1^{der}\ra G^{der}$, such that it induces an isomorphism of the associated adjoint Shimura data $(G_1^{ad}, X_1^{ad})\simeq (G^{ad},X^{ad})$. Then we have a canonical bijection $B(G_1,\mu_1)\simeq B(G,\mu)$. Consider the Newton stratification at level $K^p_1$,
\[\ov{S}_{K_{1p}K^p_1}(G_1,X_1)=\coprod_{b\in B(G_1,\mu_1)}\ov{S}_{K_{1p}K^p_1}(G_1,X_1)^b.\]When the level $K^p_1$ varies, the Newton stratifications are compatible.  Therefore, we get a Newton stratification  \[\ov{S}_{K_{1p}}(G_1,X_1)=\coprod_{b\in B(G_1,\mu_1)}\ov{S}_{K_{1p}}(G_1,X_1)^b\] by taking inverse limit over $K^p_1$.
As \cite{Ki2} 3.5.8, consider \[\pi(G_1):=G_1(\Q)^-_+\setminus G_1(\A_f)/G_1(\Z_p)=G_1(\Z_{(p)})_+^-\setminus G_1(\A_f^p),\]which is the set of geometric connected components of $S_{K_{1p}}(G_1,X_1)$. By \cite{SZ}, \[\ov{S}_{K_{1p}}(G_1,X_1)^b\subset  \ov{S}_{K_{1p}}(G_1,X_1)\] is stable under the action of $\Ac(G_{1\Z_{(p)}})$, and we have a surjective $\Ac(G_{1\Z_{(p)}})$-equivariant map
\[\ov{S}_{K_{1p}}(G_1,X_1)^b\ra  \pi(G_1).\]
Let $\ov{S}_{K_{1p}}(G_1,X_1)^{b,+}$ be the pullback of $\ov{S}_{K_{1p}}(G_1,X_1)^{b}$ under the inclusion $\ov{S}_{K_{1p}}(G_1,X_1)^+\hookrightarrow \ov{S}_{K_{1p}}(G_1,X_1)$. In other words, we consider the following commutative diagrams
\[\xymatrix{\ov{S}_{K_{1p}}(G_1,X_1)^{b,+}\ar@{^{(}->}[r]\ar@{^{(}->}[d]&\ov{S}_{K_{1p}}(G_1,X_1)^{+}\ar@{^{(}->}[d]\\
	\ov{S}_{K_{1p}}(G_1,X_1)^{b}\ar@{^{(}->}[r]\ar@{->>}[d]&\ov{S}_{K_{1p}}(G_1,X_1)\ar@{->>}[d]\\
	\pi(G_1)\ar@{=}[r]&\pi(G_1),
	}\]
where the above diagram is cartesian. The stabilizer of $\ov{S}_{K_{1p}}(G_1,X_1)^{b,+}\subset \ov{S}_{K_{1p}}(G_1,X_1)^{b}$ is $\Ac(G_{1\Z_{(p)}})^\circ$, and we have the identity
\[\ov{S}_{K_{1p}}(G_1,X_1)^b=[\Ac(G_{1\Z_{(p)}})\times \ov{S}_{K_{1p}}(G_1,X_1)^{b,+}]/\Ac(G_{1\Z_{(p)}})^\circ.\]
For more details we refer to \cite{SZ}. Now as
\[\ov{S}_{K_{p}}(G,X)=[\Ac(G_{\Z_{(p)}})\times \ov{S}_{K_{1p}}(G_1,X_1)^{+}]/\Ac(G_{1\Z_{(p)}})^\circ,\]
we get the Newton stratification 
  \[\ov{S}_{K_p}(G,X)=\coprod_{b\in B(G,\mu)}\ov{S}_{K_p}(G,X)^b, \] where for any $b\in B(G,\mu)$, the associated stratum
\[ \ov{S}_{K_p}(G,X)^b=[\Ac(G_{\Z_{(p)}})\times \ov{S}_{K_{1p}}(G_1,X_1)^{b,+}]/ \Ac(G_{1\Z_{(p)}})^\circ\hookrightarrow \ov{S}_{K_{p}}(G,X).\]
 For any sufficiently small open compact subgroup $K^p\subset G(\A_f^p)$, we define
 \[ \ov{S}_{K_pK^p}(G,X)^b=\ov{S}_{K_p}(G,X)^b/K^p.\]Therefore we get the Newton stratification at the finite level
 \[\ov{S}_{K_pK^p}(G,X)= \coprod_{b\in B(G,\mu)}\ov{S}_{K_pK^p}(G,X)^b.\]

\subsection{Rapoport-Zink uniformization}
The notations will be the same as the previous subsection. We will work over $W$ in the rest of this section. For simplicity, denote the base change of $S_K=S_K(G,X)$ over $W$ by the same notation.
Let $b\in B(G,\mu)$ (the same convention as the last subsection). We get an unramified local Shimura datum of abelian type $(G_{\Q_p},b,\{\mu\})$, thus a formal scheme $\breve{\M}=\breve{\M}(G,b,\mu)$ over $W$. Fix a point $x\in \ov{S}_K^b(\ov{\F}_p) $. 

\subsubsection{Case $(G,X)$ of Hodge type} We want to construct a morphism of formal schemes over $\Spf W$
\[\Theta=\Theta_x: \breve{\M}\times G(\A_f^p)/K^p\lra \wh{S}_K,\]where $\wh{S}_K$ is the formal completion of $S_K$ along its special fiber.  The morphism $\Theta$ is constructed in \cite{Kim2} Proposition 4.3 and Corollary 4.3.2. Let $(A_x, \lambda_x, \varepsilon_{x,K}^p, (t_{\alpha, x}))$ be the abelian variety with additional structures attached to $x$, and let $I_\phi(\Q)$ be the group of quasi-isogenies of $A_x$ preserving $(t_{\alpha, x})$. Then $I_\phi(\Q)$ is the group of $\Q$-points of a reductive group $I_\phi$ over $\Q$ (cf. \cite{Ki2} Corollary 2.3.1) which depends only on the isogeny class of $x$ (\cite{Ki2} 1.4.14).
In this case, $\Theta$ factors through the quotient by $I_\phi(\Q)$
\[\Theta:  I_\phi(\Q)\setminus\breve{\M}\times G(\A_f^p)/K^p\lra \wh{S}_K,\]
and the image $\Zm_{\phi,K^p}$ is contained in the stratum $\ov{S}_{K}^b$.

\subsubsection{Case $(G,X)$ of abelian type} We first work on the level of sets. By \cite{Ki2} Theorem 4.6.7, we have the following bijection
\[\ov{S}_{K_p}(G,X)^b(\ov{\F}_p)\st{\sim}{\lra}\coprod_{[\phi],b(\phi)=b}S(G,\phi), \]where
\[\phi: \mathfrak{Q}\lra  \mathfrak{G}_G\] runs through the set of admissible morphisms of Galois gerbs, $[\phi]$ is the associated equivalence class, cf. \cite{Ki2} 3.3, and
\[S(G,\phi)=\varprojlim_{K^p} I_{\phi}(\Q)\setminus \M_{red}(\ov{\F}_p)\times G(\A_f^p)/K^p,\]
where $\M_{red}$ is the reduced special fiber of the Rapoport-Zink space $\breve{\M}$ associated to $(G_{\Q_p}, b(\phi), \{\mu\})$.
 
\begin{remark}
In \cite{Ki2} 3.3, in fact one considers the set 
\[S(G,\phi)=\varprojlim_{K^p} I_\phi(\Q)\setminus X_p(\phi)\times X^p(\phi)/K^p,\]
where $X_p(\phi)$ and $X^p(\phi)$ are certain sets canonically associated to $\phi$, such that (cf. Lemma 3.3.4 of \cite{Ki2}) \[X_p(\phi) \simeq X_{\mu}^G(b)\simeq \M_{red}(\ov{\F}_p)\] and $X^p(\phi)$ is a $G(\A_f^p)$-torsor.
\end{remark}
 
Take an unramified Shimura datum of Hodge type $(G_1,X_1)$, together with a central isogeny $G_1^{der}\ra G^{der}$, such that it induces an isomorphism of the associated adjoint Shimura data $(G_1^{ad}, X_1^{ad})\simeq (G^{ad},X^{ad})$. Let 
\[\phi_1: \mathfrak{Q}\lra  \mathfrak{G}_{G_1}\] be an admissible morphism of Galois gerbs.
We note that \[\begin{split}
S(G_1,\phi_1)&=\varprojlim_{K^p_1} I_{\phi_1}(\Q)\setminus \M_{1red}(\ov{\F}_p)\times G_1(\A_f^p) /K^p_1\\
&= I_{\phi_1}(\Q)\setminus \M_{1red}(\ov{\F}_p)\times G_1(\A_f^p),
\end{split}
\]
where $\M_{1red}$ is the reduced special fiber of the Rapoport-Zink space $\breve{\M}_1$ associated to $(G_{1\Q_p}, b(\phi_1), \{\mu_1\})$.

Fix an admissible morphism $\phi_0: \mathfrak{Q}\ra  \mathfrak{G}_{G^{ad}}$. Consider
\[\begin{split} S(G,\phi_0)&=\coprod_{[\phi], \phi^{ad}=\phi_0} S(G, \phi)\\
&= \coprod_{[\phi], \phi^{ad}=\phi_0}\varprojlim_{K^p} I_{\phi}(\Q)\setminus \M_{red}(\ov{\F}_p)\times G(\A_f^p)/K^p. \end{split}\] By \cite{Ki2} Lemmas 3.7.2 and 3.7.4, there is an action of $\Ac(G_{\Z_{(p)}})$ on $S(G,\phi_0)$, together with an $\Ac(G_{\Z_{(p)}})$-equivariant surjective map \[c_G: S(G,\phi_0)\lra \pi(G).\] Recall that we have fixed a point $x\in \ov{S}_{K_p}(G,X)^b(\ov{\F}_p)$. We choose $\phi_0$ such that $x\in S(G,\phi_0)$ under the bijection $\ov{S}_{K_p}(G,X)^b(\ov{\F}_p)\st{\sim}{\lra}\coprod_{[\phi],b(\phi)=b}S(G,\phi)$. For the identity class $e\in \pi(G)$, consider the fiber \[S(G,\phi_0)^+=c_{G}^{-1}(e).\] Let $(G_1, X_1)$ be the unramified Shimura datum of Hodge type  as above.
Similarly we have $S(G_1, \phi_0)=\coprod_{[\phi_1], \phi^{ad}_1=\phi_0} S(G_1, \phi_1)$ and $S(G_1, \phi_0)^+$.

\begin{proposition}\label{P:quotient}
We have the following isomorphism of sets with $\Ac(G_{\Z_{(p)}})\times \lan \Phi \ran$-action
\[S(G,\phi_0)\simeq [ \Ac(G_{\Z_{(p)}})\times S(G_1,\phi_0)^+]/\Ac(G_{1\Z_{(p)}})^{\circ} .\]
\end{proposition}
\begin{proof}
This follows from Corollary 3.8.12 of \cite{Ki2}.
\end{proof}

Now we come back to Rapoport-Zink spaces.
If $K_1^{p'}\subset K_1^p$ is another open compact subgroup of $G_1(\A_f^p)$, then we have the following commutative diagram
\[\xymatrix{
	I_{\phi_1}(\Q)\setminus\breve{\M}_1\times G_1(\A_f^p)/K^{p'}_1\ar[r]\ar[d]^{\Theta_{1K^{p'}_1}}& I_{\phi_1}(\Q)\setminus\breve{\M}_1\times G_1(\A_f^p)/K^p_1\ar[d]^{\Theta_{1K^{p}_1}} \ar[d]\\
\wh{S}_{K_{1p}K^{p'}_1}(G_1,X_1) \ar[r]& \wh{S}_{K_{1p}K^p_1}(G_1,X_1)
}\]with horizontal maps finite. Therefore, if we set
\[\wh{S}(G_1,\phi_1):= \varprojlim_{K^p_1}I_{\phi_1}(\Q)\setminus\breve{\M}_1\times G_1(\A_f^p)/K^p_1,\]
then we get 
\[\Theta_1=\varprojlim_{K^p_1}\Theta_{1K^{p}_1}: \wh{S}(G_1,\phi_1)\lra \varprojlim_{K^p_1} \wh{S}_{K_{1p}K^p_1}(G_1,X_1), \] 
where both limits are taken in the category of formal schemes. 
\begin{lemma}
Let $\wh{S}_{K_{1p}}(G_1,X_1)$ be the formal completion of $S_{K_{1p}}(G_1,X_1)$ along its special fiber. Then we have
a canonical isomorphism of formal schemes \[\varprojlim_{K^p_1} \wh{S}_{K_{1p}K^p_1}(G_1,X_1)=\wh{S}_{K_{1p}}(G_1,X_1).\] 
\end{lemma}
\begin{proof}
This follows from the definition of inverse limit of formal schemes.
\end{proof}
We have thus $\Theta_1: \wh{S}(G_1,\phi_1)\lra \wh{S}_{K_{1p}}(G_1,X_1)$.  On the other hand, we have a surjective map \[c_{G_1}: S_{K_{1p}}(G_1,X_1)\lra \pi(G_1).\] Consider the fiber over $e$ of this map $c_{G_1}$, $S_{K_{1p}}(G_1,X_1)^+\subset S_{K_{1p}}(G_1,X_1)$,  and
let $\wh{S}_{K_{1p}}(G_1,X_1)^+$ be formal completion of $S_{K_{1p}}(G_1,X_1)^+$ along its special fiber.
Let \[\Theta_1^+: \wh{S}(G_1,\phi_1)^+:=\Big( \varprojlim_{K^p_1}I_{\phi_1}(\Q)\setminus\breve{\M}_1\times G_1(\A_f^p)/K^p_1\Big)^+\lra \wh{S}_{K_{1p}}(G_1,X_1)^+\]
be the pullback of 
\[\Theta_1: \wh{S}(G_1,\phi_1)= \varprojlim_{K^p_1}I_{\phi_1}(\Q)\setminus\breve{\M}_1\times G_1(\A_f^p)/K^p_1 \lra \wh{S}_{K_{1p}}(G_1,X_1)\] 
under the inclusion $\wh{S}_{K_{1p}}(G_1,X_1)^+\hookrightarrow \wh{S}_{K_{1p}}(G_1,X_1)$. The morphism 
$\Theta_1^+$ can be written as $\Theta_1^+=\varprojlim_{K^p_1}\Theta_{1K^{p}_1}^+$, with
\[\xymatrix{
\Big(I_{\phi_1}(\Q)\setminus\breve{\M}_1\times G_1(\A_f^p)/K^p_1\Big)^+\ar[r]\ar[d]^{\Theta_{1K^{p}_1}^+}& I_{\phi_1}(\Q)\setminus\breve{\M}_1\times G_1(\A_f^p)/K^p_1\ar[d]^{\Theta_{1K^{p}_1}}\\
\wh{S}_{K_{1p}K^p_1}(G_1,X_1)^+\ar[r]& \wh{S}_{K_{1p}K^p_1}(G_1,X_1).
}\] 
Define formal schemes
\[\begin{split}\wh{S}(G_1,\phi_0)^+&=\coprod_{[\phi_1],\phi^{ad}_1=\phi_0} \wh{S}(G_1,\phi_1)^+\\
&=\coprod_{[\phi_1],\phi^{ad}_1=\phi_0}\Big( \varprojlim_{K^p_1}I_{\phi_1}(\Q)\setminus\breve{\M}_1\times G_1(\A_f^p)/K^p_1\Big)^+\end{split}\]
and 
\[\wh{S}(G,\phi_0)=\coprod_{[\phi],\phi^{ad}=\phi_0} \varprojlim_{K^p}I_\phi(\Q)\setminus \breve{\M}\times G(\A_f^p)/K^p.\]
\begin{proposition}\label{P:formal quotient}
In the above situation, we have
\[ \wh{S}(G,\phi_0)\simeq [\Ac(G_{\Z_{(p)}})\times \wh{S}(G_1,\phi_0)^+ ]/\Ac(G_{1\Z_{(p)}})^{\circ}. \]
\end{proposition}
\begin{proof}
This is identical to the proof of Proposition \ref{P:quotient}.
\end{proof}

Let $\Zm_{\phi_1,K_1^p}$ (resp. $\Zm_{\phi_1,K_1^p}^+$ ) be the image of $\Theta_{1K^p_1}$ (resp. $\Theta_{1K^p_1}^+$). This exists a geometric structure on $\Zm_{\phi_1,K_1^p}$ as follows. We can write
\[\Zm_{\phi_1,K_1^p}=\bigcup _{j\in J_{K^p_1}}Z_{\phi_1,K_1^p}^j,\]
where $J_{K^p_1}$ is the $I_{\phi_1}(\Q)$-orbits of irreducible components of $\breve{\M}_1\times G_1(\A_f^p)/K^p$, and $Z_{\phi_1,K_1^p}^j$ is the image of the irreducible components under $\Theta_{1K^p_1}$ corresponding to $j\in J_{K^p_1}$. For each $j\in  J_{K^p_1}$, there exists only finitely many $j'\in J_{K^p_1}$ such that \[Z_{\phi_1,K_1^p}^j\bigcap Z_{\phi_1,K_1^p}^{j'}\neq \emptyset.\]
 Thus we get an induced geometric structure on $\Zm_{\phi_1,K_1^p}^+$ as
  \[ \Zm_{\phi_1,K_1^p}^+=\bigcup_{j\in J_{K^p_1}}Z_{\phi_1,K_1^p}^{j,+},\]
  where $Z_{\phi_1,K_1^p}^{j,+}$ is the pullback of $Z_{\phi_1K_1^p}^{j}$ to $\wh{S}_{K_{1p}K^p_1}(G_1,X_1)^+$.
When $K^p_1$ varies, $J_{K^p_1}$, $\Zm_{\phi_1,K_1^p}$, and $\Zm_{\phi_1,K_1^p}^+$  form inverse systems, and we set
\[\Zm_{\phi_1}=\varprojlim_{K^p_1}\Zm_{\phi_1,K_1^p}, \quad \Zm_{\phi_1}^+=\varprojlim_{K^p_1}\Zm_{\phi_1,K_1^p}^+. \]
Let $J_{1}$ be the $I_{\phi_1}(\Q)$-orbits of irreducible components of $\breve{\M}_1\times G_1(\A_f^p)$. For any $j\in J_1$, let $Z_{\phi_1}^j$ be the image of the irreducible components under $\Theta_{1}$ corresponding to $j$, then we can write
\[\Zm_{\phi_1}=\bigcup _{j\in J_{1}}Z_{\phi_1}^j\]and\[Z_{\phi_1}^j=\varprojlim_{K^p_1}Z_{\phi_1,K_1^p}^j,\]
where $Z_{\phi_1K_1^p}^j$ is the image of the irreducible components corresponding to $j$ under the composition 
\[\breve{\M}_1\times G_1(\A_f^p)\ra\breve{\M}_1\times G_1(\A_f^p)/K^p_1\ra  \wh{S}_{K_{1p}K^p_1}(G_1,X_1).\]
Similarly for $\Zm_{\phi_1}^+$.
By the proof of Proposition 4.6.2 of \cite{Ki2}, we have $\lan\Phi\ran \times Z_1(\Q_p)\times \Ac(G_{1\Z_{(p)}})^{I_{\phi_1}}$-equivariant bijection of sets (cf. Remark \ref{R:finer quotient})
\[\Zm_{\phi_1}(\ov{\F}_p)\simeq S(G_1,\phi_1), \quad \Zm_{\phi_1}^+(\ov{\F}_p)\simeq S(G_1,\phi_1)^+ .\]

We have (cf. \cite{Ki1} 3.4.11)
\[ [\Ac(G_{\Z_{(p)}})\times \wh{S}_{K_{1p}}(G_1,X_1)^+ ]/\Ac(G_{1\Z_{(p)}})^{\circ}= \wh{S}_{K_{p}}(G,X).\]
Recall that we fixed an admissible morphism $\phi_0: \mathfrak{Q}\ra  \mathfrak{G}_{G^{ad}}$. Set 
\[ \Zm_{G_1,\phi_0}^+=\coprod_{[\phi_1],\phi^{ad}_1=\phi_0}\Zm_{\phi_1}^+.\]
Applying the functor $[\Ac(G_{\Z_{(p)}})\times - ]/\Ac(G_{1\Z_{(p)}})^{\circ}$ to $\Zm_{G_1,\phi_0}^+$, we get a subset $\Zm_{\phi_0}(=\Zm_{G,\phi_0})\subset \wh{S}_{K_p}=\wh{S}_{K_{p}}(G,X)$. Let $\Zm_{\phi_0, K^p}$ be the image of $\Zm_{\phi_0}$ under the projection $\wh{S}_{K_p}\ra \wh{S}_K=\wh{S}_{K_pK^p}$. Then we can define the formal completion of $\wh{S}_K$ along $\Zm_{\phi_0,K^p}$ as \cite{RZ} chapter 6 and \cite{Kim2} Definition 4.6.
\begin{theorem}\label{Thm:uniformize}
We have an isomorphism of formal schemes over $W$
\[ \Theta: \coprod_{[\phi],\phi^{ad}=\phi_0} I_\phi(\Q)\setminus \breve{\M}\times G(\A_f^p)/K^p\st{\sim}{\lra}  \wh{S_K}_{/\Zm_{\phi_0,K^p}}.  \]
\end{theorem}
\begin{proof}
If $(G,X)$ is of Hodge type, this is proved in \cite{Kim2} Theorem 4.7. Assume that we are in the general case. By the above notation, it suffices to prove that \[ \coprod_{[\phi],\phi^{ad}=\phi_0} \varprojlim_{K^p}I_\phi(\Q)\setminus \breve{\M}\times G(\A_f^p)/K^p\simeq [\Ac(G_{\Z_{(p)}})\times \wh{S}(G_1,\phi_0)^+ ]/\Ac(G_{1\Z_{(p)}})^{\circ}. \]This is given by Proposition \ref{P:formal quotient}.
\end{proof}

\begin{remark}\label{R:finer quotient}
Denote by $G^{ad}_1(\Z_{(p)})^{+,I_{\phi_1}}$ the kernel of the composite of
\[G^{ad}_1(\Z_{(p)})^+ \hookrightarrow G^{ad}_1(\Z_{(p)})\ra H^1(\Q,Z_{1})\ra H^1(\Q,I_{\phi_1}),\]where $Z_1$ is the center of $G_1$.
Similarly we define $G^{ad}(\Z_{(p)})^{+,I_{\phi}}$.  Following \cite{Ki2} 4.3.4, we define
\[\begin{split} 
\Ac(G_{1\Z_{(p)}})^{I_{\phi_1}}&=G_1(\A_f^p)/Z_1(\Z_{(p)})^-\ast_{G_1(\Z_{(p)})_+/Z_1(\Z_{(p)})} G^{ad}_1(\Z_{(p)})^{+,I_{\phi_1}}\\
\Ac(G_{1\Z_{(p)}})^{I_{\phi_1},\circ}&= G_1(\Z_{(p)})_+^-/Z_1(\Z_{(p)})^-\ast_{G_1(\Z_{(p)})_+/Z_1(\Z_{(p)})} G^{ad}_1(\Z_{(p)})^{+,I_{\phi_1}}.
\end{split}
\]
Similarly we define $\Ac(G_{\Z_{(p)}})^{I_{\phi}}$ and $\Ac(G_{\Z_{(p)}})^{I_{\phi},\circ}$.
The group $\Ac(G_{1\Z_{(p)}})^{I_{\phi_1}}$ acts on $ S(G_1,\phi_1)$, cf. \cite{Ki2} Lemma 4.5.9.
By construction,  we have an $\Ac(G_{1\Z_{(p)}})^{I_{\phi_1}}$-equivariant map 
\[c_{G_1}: S(G_1,\phi_1)\ra \pi(G_1),\]which is surjective since $G_1(\A_f^p)$ (and thus $\Ac(G_{1\Z_{(p)}})^{I_{\phi_1}})$ acts transitively on $\pi(G_1)$. 
For the identity class $e\in \pi(G_1)$, consider the fiber \[S(G_1,\phi_1)^+=c_{G_1}^{-1}(e).\]We have then $S(G_1,\phi_0)^+=\coprod_{[\phi_1],\phi^{ad}_1=\phi_0} S(G_1,\phi_1)^+$.
The stabilizer of $S(G_1,\phi_1)^+\subset S(G_1,\phi_1)$ is
\[ \Ac(G_{1\Z_{(p)}})^{I_{\phi_1},\circ}\subset \Ac(G_{1\Z_{(p)}})^{I_{\phi_1}}.\] 
We have
\[S(G_1,\phi_1)=[\Ac(G_{1\Z_{(p)}})^{I_{\phi_1}}\times S(G_1,\phi_1)^+]/\Ac(G_{1\Z_{(p)}})^{I_{\phi_1},\circ}.  \]
Take any $\phi_1: \mathfrak{Q}\ra  \mathfrak{G}_{G_1}$, such that \[\phi^{ad}=\phi^{ad}_1: \mathfrak{Q}\ra  \mathfrak{G}_{G^{ad}}.\] 
It should be possible that the strategy of \cite{Ki2} 3.8 enables us to prove the following refinement of Proposition \ref{P:quotient}
\[S(G,\phi)=[\Ac(G_{1\Z_{(p)}})^{I_{\phi_1}}\times S(G_1,\phi_1)^+]/\Ac(G_{\Z_{(p)}})^{I_{\phi},\circ}.  \]
Once this is done, the same argument as above shows that there is an isomorphism of formal schemes over $W$
\[  I_\phi(\Q)\setminus \breve{\M}\times G(\A_f^p)/K^p\st{\sim}{\lra}  \wh{S_K}_{/\Zm_{\phi,K^p}}, \]where $\Zm_{\phi,K^p}$ is the image under 
the projection $\wh{S}_{K_p}\ra \wh{S}_{K_pK^p}$ of \[\Zm_{\phi}:=[\Ac(G_{1\Z_{(p)}})^{I_{\phi_1}}\times \Zm_{\phi_1}^+]/\Ac(G_{\Z_{(p)}})^{I_{\phi},\circ}.\]
\end{remark}

\begin{remark}
In the special cases of Shimura curves associated to quaternion algebras over a totally real field, see \cite{BZ} for a construction of the uniformization by Drinfeld spaces.
\end{remark}

Let $\Sh_K(\phi_0)=(\wh{S_K}_{/\Zm_{\phi_0,K^p}})^{ad}_\eta$. We get a natural morphism of adic spaces $\Sh_K(\phi_0)\ra \Sh_K^{ad}$. For any open compact subgroup $K_p'\subset G(\Q_p)$, let $\Sh_{K_p'K^p}(\phi_0)\ra \Sh_{K_p'K^p}^{ad}$ be the pullback of $\Sh_K(\phi_0)\ra \Sh_K^{ad}$ under the projection $\Sh_{K_p'K^p}^{ad}\ra \Sh_{K_pK^p}^{ad}$.
We get the following corollary from Theorem \ref{Thm:uniformize}.
\begin{corollary}\label{C:rigid uniformize}
	With the above notations, 
	$\Theta$ induces an isomorphism of rigid analytic spaces over $L$
	\[ \Theta: \coprod_{[\phi],\phi^{ad}=\phi_0} I_\phi(\Q)\setminus \M_{K_p'}\times G(\A_f^p)/K^p\st{\sim}{\lra} \Sh_{K_p'K^p}(\phi_0).  \]
\end{corollary}

We fix a morphism \[\pi: \M\lra \Sh_{K}^{ad}\] coming from the above uniformization isomorphism, which factors through the good reduction locus \[(\wh{S}_K)^{ad}_{\eta}\subset \Sh_K^{ad}.\] By \cite{LZ}, the universal $\Q_p-G$-local system $\mathcal{L}_K$ on $\Sh_K^{ad}$ is de Rham (which can be proved directly for the abelian type case; moreover we assume that $G=G^c$ for the notation $G^c$ of \cite{LZ}). When restricting to $(\wh{S}_K)^{ad}_{\eta}$, it is even crystalline. Recall by Proposition \ref{P:local system}, we have the universal $\Q_p$-$G$-local system $\mathbb{V}$ on $\M$. We have the natural local-global compatibility identity \[\mathbb{V}=\pi^\ast\mathcal{L}_K.\]

Recall that in \cite{Sh1} we have proved that there exists a perfectoid space $\S_{K^p}$ over $\C_p$ such that
\[\S_{K^p}\sim \varprojlim_{K_p'}\Sh_{K_p'K^p}(G,X)^{ad}.\]
On the other hand, by Proposition \ref{P:local perfectoid}, we get a perfectoid space $\M_\infty$ over $\C_p$ such that
\[\M_\infty\sim \varprojlim_{K_p'}\M_{K_p',\C_p}. \]
From the above Corollary \ref{C:rigid uniformize} we get
\begin{corollary}\label{C:perf uniformize}
There exists a perfectoid space $\S_{K^p}(\phi_0)$ together with a map $\S_{K^p}(\phi_0)\ra\S_{K^p}$, such that
\[ \S_{K^p}(\phi_0)\simeq \coprod_{[\phi],\phi^{ad}=\phi_0} I_\phi(\Q)\setminus \M_{\infty}\times G(\A_f^p)/K^p.\]
\end{corollary}

\begin{remark}
For the $b\in B(G,\mu)$ we fixed in this subsection, we can define the Newton stratum $\S_{K^p}^b\subset \S_{K^p}$, which is a locally closed subspace, cf. \cite{CS} subsection 4.3. Then we have $\S_{K^p}(\phi_0)\ra \S_{K^p}$ factors through $\S_{K^p}^b$. In the case that $b$ is basic, we will have $\S_{K^p}(\phi_0)= \S_{K^p}^b$, cf. the next subsection. In the general case, the image of $\S_{K^p}(\phi_0)\ra \S_{K^p}^b$ is a strict subspace, and to under understand the whole stratum $\S_{K^p}^b$, one should introduce Igusa varieties, cf. \cite{CS} section 4 in the PEL case.
\end{remark}

\subsection{The case of basic strata}

Let the notations be as in the last subsection. Assume now that $b=b_0$ is the basic element. Note that up to equivalence there is only one $\phi$ such that $b(\phi)=b_0$.
\begin{theorem}\label{T:uniformize basic}
In the setting above, $\Zm_{\phi_0,K^p}=\ov{S}^b_K$. Thus we have an isomorphism of formal schemes over $W$
\[  \Theta: I_\phi(\Q)\setminus \breve{\M}\times G(\A_f^p)/K^p \st{\sim}{\lra} \wh{S_K}_{/\ov{S}^b_K}.\]
\end{theorem}
\begin{proof}
In the case that $(G,X)$ is of Hodge type, this is proved in Theorem 4.11 of \cite{Kim2}. We note that by \cite{XZ} Lemma 7.2.14, $I_\phi$ is an inner form of $G$ with $I_{\phi,p}=J_b$.
The general case follows from the Hodge type case by construction.
\end{proof}

Corresponding to Corollaries \ref{C:rigid uniformize} and \ref{C:perf uniformize}, we have
\begin{corollary}For any open compact subgroup $K_p'\subset G(\Q_p)$,
	$\Theta$ induces an isomorphism of rigid analytic spaces over $L$
	\[ \Theta: I_\phi(\Q)\setminus \M_{K_p'}\times G(\A_f^p)/K^p \st{\sim}{\lra} \Sh_{K_p'K^p}^b, \]and an isomorphism of perfectoid spaces over $\C_p$
		\[ \Theta: I_\phi(\Q)\setminus \M_{\infty}\times G(\A_f^p)/K^p \st{\sim}{\lra} \S_{K^p}^b. \]
\end{corollary}

\section{Ekedahl-Oort stratifications for good reductions of Rapoport-Zink spaces}
In this section, we will construct and study Ekedahl-Oort stratifications for the special fibers of our Rapoport-Zink spaces of abelian type, motivated by the study of Artin invariants of K3 surfaces. Then we will discuss some special interesting cases, namely, the fully Hodge-Newton decomposable cases, cf. \cite{GoHeNi}. In the next section we will further specialize to orthogonal groups. As before, we assume $p>2$ in this section.

\subsection{Ekedahl-Oort stratifications for special fibers  of Rapoport-Zink spaces}
Let $(G,[b],\{\mu\})$ be an unramified local Shimura datum of abelian type, $b\in G(L)$ be a representative of $[b]$, and $\breve{\M}=\breve{\M}(G,b,\mu)$ be the associated Rapoport-Zink space by Theorem \ref{T:ab}.  Consider the special fiber $\ov{\M}$ over $\ov{\F}_p$ of $\breve{\M}$ and the associated reduced special fiber $\M_{red}$ of $\breve{\M}$, which is by definition the reduced subscheme of $\ov{\M}$. 

Since our local Shimura datum is unramified, we can consider $G\tr{-Zip}^{\mu}$, the stack of $G$-zips of type $\mu$ over $\ov{\F}_p$ (we refer to \cite{PWZ} and \cite{Zh} 1.2 for some basic facts about $G$-zips and the stack $G\tr{-Zip}^{\mu}$). The underling set of geometric points of $G\tr{-Zip}^{\mu}$ is in canonical bijection with a subset $^J\W$ of the Weyl group $\W$ of $G$ (for a fixed choice of maximal torus). More precisely, let $J$ be the type of the parabolic subgroup of $G$ attached to $\{\mu\}$ in the usual way, and $\W_J$ be the associated subgroup of $\W$, then $^J\W\subset \W$ is the set of elements $w\in\W$ that are of minimal length in their coset $\W_Jw$. There is a partial order $\preceq$ on $^J\W$ making which into a topological space, cf. \cite{Zh} 3.1 or \cite{W} 5.3.  In fact we have isomorphisms of topological spaces \[|G\tr{-Zip}^{\mu}|\simeq\, ^J\W,\] cf. \cite{Zh} Theorem 3.1.5 and \cite{W} Proposition 5.12.

\begin{theorem}\label{T:RZ EO}
	There exists a formally smooth morphism
\[ \zeta: \ov{\M}\lra G\tr{-Zip}^{\mu},\] which induces a decomposition 
\[ \ov{\M}=\coprod_{w\in ^J\W}\ov{\M}_w,\]
where $\ov{\M}_w\subset \ov{\M}$ is locally closed ( could be empty).
\end{theorem} 
\begin{proof}
Assume first that $(G,[b],\{\mu\})$  is of Hodge type. Then the universal $p$-divisible group with crystalline Tate tensors gives rise to a $G$-zip of type $\mu$ $(I, I_+, I_-, \iota)$ on $\ov{\M}$ : the arguments of \cite{Zh} Theorem 2.4.1 apply to our local setting. Thus we get a morphism \[ \zeta: \ov{\M}\lra G\tr{-Zip}^{\mu}.\] This morphism is $J_b(\Q_p)$-invariant. To show this morphism is formally smooth, one can apply the arguments for the proof of \cite{Zh} Theorem 3.1.2. More precisely, from the datum $(I_+, I_-, \iota)$ one can construct an $E_{G,\mu}$-torsor $\ov{\M}^\sharp$ over $\ov{\M}$, which sits in the following cartesian diagram:
\[\xymatrix{
\ov{\M}^\sharp\ar[r]^{\zeta^\sharp}\ar[d]& G_{\ov{\F}_p}\ar[d]\\
\ov{\M}\ar[r]&[E_{G,\mu}\setminus G_{\ov{\F}_p}]=G\tr{-Zip}^{\mu}.
}\]
Here the algebraic group $E_{G,\mu}$ is as that in \cite{PWZ} section 3 (see also \cite{Zh} subsection 1.2), considered as a group over $\ov{\F}_p$.
It suffices to show that $\zeta^\sharp$ is formally smooth. By \cite{EGA4.4} Prop. 17.1.6, formal smoothness is local in both the sources and targets, thus we can reduce to the algebra side: take any affine subschemes $\Spec B\subset \ov{\M}^\sharp, \Spec A\subset G_{\ov{\F}_p}$ over $\ov{\F}_p$ with induced $\zeta^\sharp: \Spec B\ra \Spec A$, we need to prove that $B$ is formally smooth over $A$. Since $\ov{\M}^\sharp$ is formally smooth (and formally locally of finite type) over $\ov{\F}_p$, $G_{\ov{\F}_p}$ is smooth over $\ov{\F}_p$, by \cite{Mats} Theorem 65 and Theorem 66, the fact that $B$ is formally smooth over $A$ is equivalent to the map \[\Omega_{A/\ov{\F}_p}\otimes_AB/J\ra \Omega_{B/\ov{\F}_p}\otimes B/J\]is left invertible, where $J\subset B$ is an ideal of definition. Note that $\Omega_{B/\ov{\F}_p}\otimes B/J$ is a projective $B/J$-module. By \cite{EGA4.1} Cor. 19.1.12, the last statement is equivalent to for any closed points $x\in \Spec B, y=\zeta^\sharp (x)\in \Spec A$, the induced maps on the tangent spaces $T_x\Spec B\ra T_y\Spec A$ are surjective. One can then conclude by the arguments in the proof of \cite{Zh} Theorem 3.1.2 to show the surjectivity on tangent spaces.

Now assume that $(G,[b],\{\mu\})$ is unramified of abelian type. Take any unramified local Shimura datum of Hodge type $(G_1,[b_1],\{\mu_1\})$ such that $(G^{ad}, [b^{ad}],  \{\mu^{ad}\})\simeq (G_1^{ad}, [b_1^{ad}],  \{\mu_1^{ad}\})$. Let $\ov{\M}$ and $\ov{\M}_1$ be the special fibers of Rapoport-Zink spaces attached to $(G,[b],\{\mu\})$  and $(G_1,[b_1],\{\mu_1\})$ respectively. By construction after fixing $x_0\in \pi_1(G_1)^\Gamma$ we have $\ov{\M}^+=\ov{\M}_1^+$. Consider the restriction $\zeta_1^+: \ov{\M}_1^+\lra G_1\tr{-Zip}^{\mu_1}$. As $|G_1\tr{-Zip}^{\mu_1}|\simeq |G\tr{-Zip}^{\mu}|$, we get a formally smooth morphism $\zeta^+: \ov{\M}^+\lra G\tr{-Zip}^{\mu}$. Applying the $J_b(\Q_p)$-action, we get a formally smooth $J_b(\Q_p)$-invariant morphism \[\zeta: \ov{\M}\lra G\tr{-Zip}^{\mu},\] as desired.
\end{proof}
We note that in the EL/PEL cases, Wedhorn and Lau proved the above proposition previously, cf. \cite{Wed}, \cite{PWZ} Example 9.21 and \cite{VW} Theorem 3.2. If $(G, [b],  \{\mu\})\hookrightarrow (\GL_n, [b'],  \{\mu'\})$ is an embedding of unramified local Shimura data of Hodge type, by construction, we have the commutative diagram
\[\xymatrix{
\ov{\M}(G, b, \mu)\ar@{^{(}->}[r]\ar[d]^{\zeta_G} &\ov{\M}(\GL_n, b', \mu')\ar[d]^{\zeta_{\GL_n}}\\
	G\tr{-Zip}^{\mu}\ar[r] & \GL_n\tr{-Zip}^{\mu'}.
}
\]

Let $^J\W^b\subset \,^J\W$ be the subset defined by the image of $\zeta$. For each $w\in \,^J\W^b$, we call $\ov{\M}_{w}$ the Ekedahl-Oort stratum of $\ov{\M}$ attached to $w$. We get a stratification \[\ov{\M}=\coprod_{w\in ^J\W^b}\ov{\M}_{w}.\] We also get an induced stratification \[\M_{red}=\coprod_{w\in ^J\W^b}\M_{w},\]
where $\M_w\subset \M_{red}$ is a locally closed subscheme of $\M_{red}$, which we call the Ekedahl-Oort stratum of $\M_{red}$ associated to $w$. For a locally closed subscheme $X\subset Y$, we write $X^{cl}$ the (Zariski) closure of $X$ in $Y$. By construction, we have the closure relations
\[\ov{\M}_w^{cl}=\coprod_{w'\preceq w, w'\in ^J\W^b} \ov{\M}_{w'}, \quad \M_w^{cl}=\coprod_{w'\preceq w, w'\in ^J\W^b} \M_{w'}.\]

In \cite{GoHeNi} 1.4 (and \cite{GoHe} 3.4), there exists a decomposition
\[X^G_\mu(b)=\coprod_{w\in \Adm(\mu)\cap ^K\wt{W}} X_{K, w}(b), \]
where in our case $K=G(\Z_p)$ is the fixed hyperspecial group. Here are some explanations about the notations: $\wt{W}$ is the Iwahori Weyl group of $G$, $\Adm(\mu)\subset \wt{W}$ is the (finite) subset defined as (2.1) of \cite{HR}, $^K\wt{W}\subset \wt{W}$ is the set of minimal length elements in the coset $W_K\setminus\wt{W}$, with $W_K\subset \wt{W}$ the subgroup corresponding to $K=G(\Z_p)$, our fixed hyperspecial group. The above decomposition holds in the general setting of arbitrary local Shimura data.
 On the other hand, if we set $\Adm^K(\mu)=W_K\Adm(\mu)W_K$, then we have 
\[\Adm(\mu)\cap\, ^K\wt{W}= \Adm^K(\mu)\cap\, ^K\wt{W}\st{\sim}{\lra} \,^J\W,\]
where the first equality follows from \cite{HR} Theorem 6.10, and the second bijection is induced by the projection $\wt{W}\ra \W$ by \cite{Vieh} Theorem 1.1 (1). Moreover, this bijection preserves the order $\preceq_{K,\sigma}$ on $^K\wt{W}$ (cf. \cite{HR} 6.5 and \cite{GoHe} 3.3) and the order $\preceq$ on $^J\W$.
Therefore we can rewrite the above decomposition in the hyperspecial level as \[X^G_\mu(b)=\coprod_{w\in ^J\W } X_{K, w}(b).\]

Recall that by Theorem \ref{T:ab} we have \[\ov{\M}^{perf}\simeq X^G_\mu(b).\]
\begin{proposition}\label{P:EO}
For $w\in\, ^J\W$ the strata  $X_{K, w}(b)\neq \emptyset$ if and only if $w\in\, ^J\W^b$, in which case 
we have \[\ov{\M}_{w}^{perf}\simeq X_{K,w}(b).\] In particular, $\ov{\M}_{w}$ is of dimension $\dim X_{K,w}(b)$ if it is non-empty.
\end{proposition}
\begin{proof}
It suffices to prove $\ov{\M}_{w}^{perf}\simeq X_{K,w}(b)$ for any $w\in\, ^J\W$.
We first consider the Hodge type case. By the proof of \cite{Zhu} Proposition 3.11, we have two morphism $f: \ov{\M}^{perf}\ra X_{\mu}^G(b)$ and $f^{-1}: X_{\mu}^G(b) \ra \ov{\M}^{perf}$, inverse to each other, by using Dieudonn\'e theory over perfect rings. It suffices to check that $f$ (resp. $f^{-1}$) induces $f: \ov{\M}_{w}^{perf}\ra X_{K,w}(b)$ (resp. $f^{-1}:  X_{K,w}(b)\ra \ov{\M}_{w}^{perf}$). This follows from \cite{Vieh} Theorem 1.1 (see loc. cit. section 7 for some discussion in the global setting of Shimura varieties), see also Remark 6.5 (2) of \cite{HR}. Indeed, the above decomposition $X^G_\mu(b)=\coprod_{w\in \Adm(\mu)\cap ^K\wt{W}} X_{K, w}(b)$ is given by truncations of level 1 of elements in the Witt loop group. Since $G$ is unramified, by \cite{Vieh} Theorem 1.1, the set of $\sigma-\breve{K}$-conjugacy classes in $\breve{K}_1\setminus \breve{K}\mu(p)\breve{K}/\breve{K}_1$ (together with the partial order $\preceq_{K,\sigma}$ is identified with the underlying topological space of the algebraic stack $G\tr{-Zip}^{\mu}$. Here $\breve{K}_1$ is the kernel of the projection $\breve{K}=G(W)\ra G(\ov{\F}_p)$.
Thus the above decomposition is equivalent to a morphism of algebraic stacks
\[X_\mu^G(b)\ra G\tr{-Zip}^{\mu,perf},\]which is given by $\zeta^{perf}\circ f^{-1}$ since by construction $f$ and $f^{-1}$ preserve the universal $G$-Zips.

Now assume that $(G,[b],\{\mu\})$ is unramified of abelian type. Take any unramified local Shimura datum of Hodge type $(G_1,[b_1],\{\mu_1\})$ such that $G^{ad}, [b^{ad}],  \{\mu^{ad}\})\simeq (G_1^{ad}, [b_1^{ad}],  \{\mu_1^{ad}\})$. Let $\ov{\M}$ and $\ov{\M}_1$ be the special fibers of Rapoport-Zink spaces attached to $(G,[b],\{\mu\})$  and $(G_1,[b_1],\{\mu_1\})$ respectively. As always after fixing a point $x_0\in \pi_1(G_1)^\Gamma$ we have $\ov{\M}^+=\ov{\M}_1^+$. The restriction of $f$ induces an isomorphism $f_1^+: \ov{\M}_{1w}^{+,perf}\simeq  X_{K_1,w}(b_1)^+$.  Applying the $J_b(\Q_p)$ action, we get \[\ov{\M}_{w}^{perf}\simeq X_{K,w}(b)\] as desired.
\end{proof}

\begin{remark} Let $(G,[b],\{\mu\})$ be a local Shimura datum such that $G$ is unramified. Fix a representative $b\in G(L)$ of $[b]$.
\begin{enumerate}
	\item The closure relation for the decomposition $X^G_\mu(b)=\coprod_{w\in ^J\W } X_{K, w}(b)$ can be proved similarly as \cite{HR} Theorem 6.15. See also \cite{GoHeNi} 4.11 and \cite{GoHe} section 7.
	\item If we were working in the equal characteristic setting, then a formula for $\dim X_{K,w}(b)$ is known by combining \cite{GoHe} Theorem 4.1.2 (2) and \cite{He} Theorem 6.1. In our mixed characteristic setting, it should be possible to prove the same formula holds by applying \cite{HR} Proposition 6.20 and the Witt vector affine flag varieties in 1.4 of \cite{Zhu} and the method in 3.1 of loc. cit..
\end{enumerate}
\end{remark}

\subsection{Rapoport-Zink spaces for a fully Hodge-Newton decomposable pair $(G,\{\mu\})$}
We discuss some special Rapoport-Zink spaces in this subsection. Our motivation here is the observation that the list in the classification of \cite{GoHeNi} Theorem 2.5 (a posteriori) lies in our class of local abelian type (for minuscule $\mu$).

Let $G$ be a connected reductive group over $\Q_p$ and $\{\mu\}$ be a conjugacy class of cocharacters $\mu: \G_m\ra G_{\ov{\Q}_p}$. Assume that $G$ is quasi-split.
Recall the following definition
\begin{definition}[\cite{GoHeNi} Definition 2.1, \cite{CFS} Lemma 4.11]
	The pair $(G,\{\mu\})$ (or the set $B(G,\mu)$) is called fully Hodge-Newton decomposable if for any non basic $[b']\in B(G,\mu)$, the pair $([b'], \{\mu\})$ is Hodge-Newton decomposable, i.e. there exists a proper Levi subgroup $G\supsetneq M\supset M_{b'}$ such that $\kappa_M(b')=\mu^\sharp$ in $\pi_1(M)_\Gamma$.
\end{definition}
Recall $M_{b'}=M_{\nu_{b'}}$ is the Levi subgroup of $G$ defined as the centralizer of $\nu_{[b']}$.
In \cite{GoHeNi} Theorem 2.5 there is a purely group theoretical classification of all the fully Hodge-Newton decomposable pairs $(G,\{\mu\})$, and in loc. cit. Theorem 2.3 one can find further equivalent conditions ( those in (2)-(6) of the statement of the theorem) for $(G,\{\mu\})$ being fully Hodge-Newton decomposable.

\begin{theorem}\label{T:FHN}
Let $(G,[b],\{\mu\})$  be an unramified local Shimura datum of abelian type, $\M_{red}$ be the reduced Rapoport-Zink space associated to $(G,[b],\{\mu\})$. Suppose that $(G, \{\mu\})$ is fully Hodge-Newton decomposable. Then we have
\begin{enumerate}
	\item if $[b]$ is non basic, then $\dim \M_{red}=0$;
	\item if $[b]$ is basic, then the perfection of each Ekedahl-Oort stratum $\ov{\M}_w\subset \ov{\M}$ is a disjoint union of perfections of classical Deligne-Lusztig varieties;
	\item for each $w\in\, ^J\W$, there exists a unique $[b']\in B(G, \mu)$ such that $\M_w\neq \emptyset$, where $\M_w$ is an Ekedahl-Oort stratum of $\M_{red}'$, the reduced Rapoport-Zink space associated to $(G,[b'],\{\mu\})$. In particular we get a decomposition \[^J\W=\coprod_{[b']\in B(G,\mu)} \,^J\W^{b'}.\]
\end{enumerate}
Conversely, if $(G, \{\mu\})$ is part of any unramified local Shimura datum of abelian type with one of the above three conditions holds, then $(G, \{\mu\})$ is fully Hodge-Newton decomposable. 
\end{theorem}
\begin{proof}
This follows from \cite{GoHeNi} Theorem 2.3 (in the hyperspecial level case), our Theorem \ref{T:ab} and Proposition \ref{P:EO}.
\end{proof}

\begin{remark}
Let $(G,[b],\{\mu\})$  be an unramified local Shimura datum of abelian type, with $\ov{\M}$ the special fiber of the associated Rapoport-Zink space. Suppose that $[b]$ is non basic and the pair $(b, \{\mu\})$ is Hodge-Newton decomposable. With our Theorem \ref{T:ab} at hand, we refer the reader to \cite{GoHeNi} Theorems 3.16 and 6.2 (in the unramified case) to write down an isomorphism between $\ov{\M}^{perf}$ and $\ov{\M}(M, b_M, \mu_M)^{perf}$,  the perfection of the special fiber of some Rapoport-Zink space of abelian type attached to a Levi subgroup $M$ of $G$.
\end{remark}

\section{Applications to moduli spaces of K3 surfaces in mixed characteristic}
In this section, we discuss some applications to K3 surfaces and their moduli in mixed characteristic.  We will firstly discuss some examples of Rapoport-Zink spaces of orthogonal type, following the line of the previous section. Then we will move to orthogonal Shimura varieties and moduli spaces of K3 surfaces.  Finally, we will apply our constructions of Rapoport-Zink uniformization and Ekedahl-Oort stratifications to moduli spaces of K3 surfaces. Again, we assume $p>2$ in this section.

\subsection{Rapoport-Zink spaces of orthogonal type}
In this and the next subsection, we will discuss an example of Rapoport-Zink space for a fully Hodge-Newton decomposable pair $(G,\{\mu\})$.

Let $(L, Q)$ be a non degenerate self dual quadratic lattice of rank $n+2$ over $\Z_p$, where $n\geq 1$ is an integer. We write $(V,Q)$ as the induced quadratic space over $\Q_p$. Let $G=\SO(V, Q), G_1=\GSpin(V, Q)$ be the associated special orthogonal and spinor similitudes groups over $\Q_p$. By our assumption that $L$ is self dual, both $G$ and $G_1$ are unramified. We have an exact sequence of groups 
\[1\ra \G_m\ra G_1\ra G\ra 1,\] which is in fact defined over $\Z_p$.

As in \cite{HP} subsection 4.2, there is a natural choice of minuscule cocharacter $\mu_1$ of $G_1$. Take any $[b_1]\in B(G_1,\mu_1)$. Then $(G_1,[ b_1],\{\mu_1\})$ is a local Shimura datum of Hodge type.
We get a local Shimura datum $(G,[b],\{\mu\})$ by taking $[b],\{\mu\}$ as the image of $[b_1],\{\mu_1\}$ under the map $G_1\ra G$. By construction $(G,[b],\{\mu\})$ is unramified of abelian type.  We get the associated Rapoport-Zink spaces $\breve{\M}_1=\breve{\M}(G_1,b_1,\mu_1)$ and $\breve{\M}=\breve{\M}(G,b,\mu)$. The pairs $(G_1, \{\mu_1\})$ and $(G, \{\mu\})$ are fully Hodge-Newton decomposable by Theorem 2.5 of \cite{GoHeNi} (or one can compute the sets $B(G_1,\mu_1)$  and $B(G,\mu)$ directly to see they are fully Hodge-Newton decomposable).

Let $X_0$ be the $p$-divisible group over $\ov{\F}_p$ with contravariant Dieudonn\'e module $(C(V)^\vee\otimes W, b_1\sigma)$, where $C(V)$ is the Clifford algebra attached to $V$ and $C(V)^\vee$ is its dual. Fix any $\delta\in C(V)^\times$ with $\delta^\ast=\delta$ where $\ast$ is the canonical involution on $C(V)$. Then $\psi_\delta(c_1,c_2)=\tr{Tr}(c_1\delta c_2^\ast)$ is a perfect symplectic form on $C(V)$. Here $\tr{Tr}: C(V)\ra \Z_p$ is the reduced trace map. The perfect symplectic form $\psi_\delta$ on $C(V)$ induces a principal polarization $\lambda_0: X_0\ra X_0^\vee$. There exists a finite collection tensors $(s_\alpha)_{\alpha\in I}$ which includes $\psi_\delta$, such that $G_1\subset \GL(C(V))$ is cut out by $(s_\alpha)_{\alpha\in I}$.
Recall that $\breve{\M}_1$ has the following moduli interpretation. For any $R\in\Nilp_W^{sm}$, $\breve{\M}_1(R)=\{(X,(t_\alpha)_{\alpha\in I}, \rho)\}/\simeq$, where
	\begin{itemize}
		\item $X$ is a $p$-divisible group on $\Spec R$,
		\item $(t_\alpha)_{\alpha\in I}$ is a collection of crystalline Tate tensors of $X$,
		\item $\rho: X_0\otimes R/J\ra X\otimes R/J$ is a quasi-isogeny which sends $s_\alpha\otimes 1$ to $t_\alpha$ for $\alpha\in I$, where $J$ is some ideal of definition of $R$,
	\end{itemize}
	such that the following condition holds:\\
	the $R$-scheme 
	\[ \Isom\Big(\big(\D(X)_R, (t_\alpha), \Fil^\bullet(\D(X)_R) \big), \big(\Lambda\otimes R, (s_\alpha\otimes1), \Fil^\bullet\Lambda\otimes R\big)\Big) \]
	that classifies the isomorphisms between locally free sheaves $\D(X)_R$ and $\Lambda\otimes R$ on $\Spec R$ preserving the tensors and the filtrations is a $P_{\mu_1}\otimes R$-torsor.

The exact sequence $1\ra \G_m\ra G_1\ra G\ra 1$ induces a long exact sequence (cf. \cite{Bor} Lemma 1.5)
\[1\ra \pi_1(\G_m)^\Gamma\ra \pi_1(G_1)^\Gamma\ra \pi_1(G)^\Gamma\ra H^1(\Gamma, \pi_1(\G_m))\ra \cdots\]
We have the following isomorphisms \[\pi_1(\G_m)^\Gamma=\pi_1(\G_m)\simeq X_\ast(\G_m)\simeq \Z.\] Since $G_1^{der}=\Spin(V)$ is simply connected and we have the exact sequence \[1\ra \Spin(V)\ra \GSpin(V)\ra \G_m\ra 1,\] we get (\cite{Bor} 1.6) \[\pi_1(G_1)^\Gamma\simeq\pi_1(\G_m)^\Gamma\simeq \Z.\] On the other hand, since 
\[1\ra \mu_2\ra \Spin(V)\ra \SO(V)\ra 1\] is exact, we get \[\pi_1(G)=\mu_2(-1)=\Z/2\Z.\]
\begin{lemma}
We have $\pi_1(G)^\Gamma\simeq \Z/2\Z$ and the map $\pi_1(G_1)^\Gamma\ra \pi_1(G)^\Gamma$ is surjective.
\end{lemma}
\begin{proof}
As $\mu_2\subset \G_m$, and $\Gamma$ acts trivially on the later, we get $\pi_1(G)^\Gamma=\mu_2(-1)^\Gamma=\mu_2(-1)=\Z/2\Z$. For the second assertion, note that \[ \pi_1(G_1)^\Gamma/2\pi_1(G_1)^\Gamma =\Z/2\Z\subset \textrm{Im}(\pi_1(G_1)^\Gamma\ra \pi_1(G)^\Gamma).\]Thus the image is $\pi_1(G)^\Gamma$.
\end{proof}
\begin{corollary}\label{C:RZ SO}
	We have an isomorphism of formal schemes $\breve{\M}\simeq \breve{\M}_1/p^\Z$. 
\end{corollary}
\begin{proof}
 By the above lemma $\pi_1(G_1)^\Gamma\ra \pi_1(G)^\Gamma$ is surjective.Thus $ \breve{\M}\simeq \breve{\M}_1/p^\Z$ as the proof of (1) of Proposition \ref{P:moduli}.
\end{proof}

As the pairs $(G_1, \{\mu_1\})$ and $(G, \{\mu\})$ are fully Hodge-Newton decomposable, by Theorem \ref{T:FHN} we get
\begin{corollary}
Assume that $[b_1]$ (and hence $[b]$) is non basic. Then we have $\dim \M_{red}=\dim \M_{1red}=0$
\end{corollary}

\subsection{Ekedahl-Oort and Howard-Pappas stratifications for basic Rapoport-Zink spaces of orthogonal type}\label{subsection:HPEO}
Let the notations be as in the previous subsection. Now assume that $[b_1]$ (thus $[b]$) is basic.

In \cite{HP}, Howard and Pappas introduced a stratification\footnote{In \cite{HP} 6.5 it is called the Bruhat-Tits stratification, and our $\M_{1\Lambda}^\circ$ is denoted by $\tr{BT}_\Lambda$ there. } for the reduced special fiber $\M_{1red}$ of $\breve{\M}_1$:
\[\M_{1red}=\coprod_{\Lambda}\M_{1\Lambda}^\circ,\]where $\Lambda$ runs through the set of vertex lattices. 
By definition (cf. \cite{HP} section 5), a vertex $\Lambda$ lattice is a $\Z_p$-lattice in $V_L^{\Phi}$, such that \[p\Lambda \subset \Lambda^\vee\subset \Lambda\subset V_L^\Phi.\] Here $L=W(\ov{\F}_p)_\Q, \Phi=b\sigma$ is the Frobenius, $V_L^{\Phi}$ admits a quadratic form induced from $V_L$, so that this quadratic space $V_L^{\Phi}$ has the same dimension and determinant as $V$, but has Hasse invariant -1. Associated to a vertex, we have the type \[t(\Lambda):=\dim \Lambda/\Lambda^\vee,\] which is an even integer with $2\leq t(\Lambda)\leq t_{max}$, where
\[t_{max}=\begin{cases}
n+1,& n \quad\tr{odd},\\
n+2,& n\quad \tr{even},\quad \det V\neq (-1)^{\frac{n}{2}},\\
n,& n\quad \tr{even},\quad \det V= (-1)^{\frac{n}{2}}.
\end{cases}
\]
Recall that we have the inclusion \[V_L^\Phi\subset End(X_0)_\Q,\] so each vertex lattice $\Lambda\subset V_L^\Phi$ can be viewed as a set of self quasi-isogenies of $X_0$.
For each vertex lattice $\Lambda$, the associated Howard-Pappas stratum\footnote{We refer the reader to \cite{LiZhu} Definition 2.9.1 and Theorem 4.2.11 for a variant.} \[\M_{1\Lambda}^\circ\subset\M_{1red}\] is the locus $(X, (t_\alpha),\rho)$ where \[\rho\circ \Lambda^\vee\circ \rho\subset End(X)\] and this does not hold for any smaller vertex lattice $\Lambda'\subsetneq \Lambda$.  Let $\M_{1\Lambda}\subset\M_{1red}$ be its Zariski closure.
In \cite{HP} 4.3.3 and 6.4.1, Howard and Pappas proved that there exists a decomposition 
\[\M_{1red}=\coprod_{j\in\Z}\M_1^{(j)},\]
such that each $\M_1^{(j)}$ is a connected component of $\M_{1red}$. Accordingly, we get a decomposition for each stratum
\[\M_{1\Lambda}^\circ=\coprod_{j\in\Z}\M_{1\Lambda}^{(j),\circ}.\]
By \cite{HP} Theorem 6.5.6, each connected stratum $\M_{1\Lambda}^{(j),\circ}$ is isomorphic to a Deligne-Lusztig variety $X_B(w)$ for the group $\SO(\Lambda_W/\Lambda_W^\vee)$.

 As 
\[\breve{\M}\simeq \breve{\M}_1/p^\Z\simeq \breve{\M}^{(0)}\coprod \breve{\M}^{(1)},\]
we get an induced Howard-Pappas stratification for $\M_{red}$
\[\M_{red}=\coprod_{\Lambda}\M_\Lambda^\circ.\]
In fact, in \cite{HP} sections 5 and 6 Howard and Pappas studied the geometric structures of $\M_{1red}$ by passing to the quotient space $\M_{red}=\M_{1red}/p^\Z$ first.

Recall that $W=W(\ov{\F}_p), L=W_\Q$. Following \cite{HP}, we can describe the sets $\M_{red}(\ov{\F}_p), \M_{\Lambda}(\ov{\F}_p)$ and $\M_{\Lambda}^\circ(\ov{\F}_p)$ in terms of special lattices of $V_L$ as follows. By definition (\cite{HP} Definition 5.2.1) a special lattice $\Ll\subset V_L$ is a self-dual $W$-lattice such that \[(\Ll+\Phi_\ast(\Ll))/\Ll\simeq W/pW,\] where $\Phi_\ast(\Ll)$ is the $W$-submodule generated by $\Phi(\Ll)$.
By Proposition 6.2.2 of \cite{HP}, we have a bijection
\[\M_{red}(\ov{\F}_p) \simeq \{\tr{special lattices}\,\Ll\subset V_L\}.\]By loc. cit. 5.3.1 and Theorem 6.3.1 we have bijections
\[ \begin{split} \M_{\Lambda}(\ov{\F}_p)&\simeq \{\tr{Lagrangians}\, \Ll\subset \Omega: \dim(\Ll+\Phi(\Ll))=d+1\}\\
&\simeq \{\tr{special lattices}\,\Ll\subset V_L: \Lambda_W^\vee\subset\Ll\subset \Lambda_W\}\\
&= \{\tr{special lattices}\,\Ll\subset V_L: \Lambda(\Ll)\subset\Lambda\},
\end{split}\]
where $\Omega=\Lambda_W/\Lambda_W^\vee, \Lambda(\Ll)=(\Ll^{(d)})^{\Phi}, d=\frac{t(\Lambda)}{2}$, and $\Ll^{(d)}=\Ll+\Phi(\Ll)+\cdots+\Phi^{d}(\Ll)$.  Under the above description, we have the bijection
\[\begin{split} \M_{\Lambda}^\circ(\ov{\F}_p)&\simeq \{\tr{special lattices}\,\Ll\subset V_L:  \Ll^{(d)}=\Lambda_W\}\\
&= \{\tr{special lattices}\,\Ll\subset V_L: \Lambda(\Ll)=\Lambda\}.
\end{split}\]
In fact the above descriptions are true for any finitely generated field extension $k'|\ov{\F}_p$ (cf. \cite{HP}).

Let $G_1\tr{-Zip}^{\mu_1}$ be the stack of $G_1$-zips of type $\mu_1$ over $\ov{\F}_p$. The universal $p$-divisible group with crystalline Tate tensors on  $\breve{\M}_1$ defines a morphism
\[\zeta: \ov{\M}_{1}\lra G_1\tr{-Zip}^{\mu_1}.\]
The underling set of geometric points of $G_1\tr{-Zip}^{\mu_1}$ is in canonical bijection with the subset $^J\W$ of the Weyl group of $G_1$. In fact we have isomorphisms of topological spaces \[|G_1\tr{-Zip}^{\mu_1}|\simeq |G\tr{-Zip}^{\mu}|\simeq\, ^J\W.\] Let $^J\W^b\subset \,^J\W$ be the subset defined by the image of $\zeta$. For each $w\in \,^J\W^b$, recall we have the Ekedahl-Oort stratum of $\M_{1red}$ associated to $w$: \[\M_{1w}=\zeta^{-1}(w)_{red}.\] We get the Ekedahl-Oort stratification \[\M_{1red}=\coprod_{w\in ^J\W^b}\M_{1w}.\] We get also the induced Ekedahl-Oort stratification for $\M_{red}$.

Let $m\geq 1$ be such that $2m=n+1$ if $n$ is odd, and $2m=n+2$ if $n$ is even. Then there is a bijection (cf. \cite{Zh} subsection 4.4)
\[^J\W \st{\sim}{\lra} \begin{cases} \{0,1,\cdots, 2m-1\},  & n=2m-1 \quad\tr{odd}\\
\{0,1,\cdots, m-2, m-1, m-1', m, \cdots, 2m-2\},   & n=2m-2\quad\tr{even}\end{cases}
\] 
induced by the length function $ w\mapsto \ell(w)$, where we use the symbols $m-1', m-1$ to distinguish 
the two elements with the same length $m-1$.  Under the above bijection, the subset $^J\W^b\subset\, ^J\W$ can be described as
\[^J\W^b \st{\sim}{\lra} \begin{cases}
\{m,\cdots, 2m-1\},& n=2m-1 \quad\tr{odd},\\
\{m,\cdots, 2m-2\},& n=2m-2\quad \tr{even},\quad \det V= (-1)^{\frac{n}{2}},\\
\{m-1, m-1', m, \cdots, 2m-2\},& n=2m-2\quad \tr{even},\quad \det V\neq (-1)^{\frac{n}{2}}.
\end{cases}
\]
For each $i\neq m-1'$ on the right hand side, we denote the corresponding element of the left hand side as $w_i$. The element corresponding to $m-1'$ will be denoted by $w_{m-1}'$.

We can describe the map $i\mapsto w_i$ in more details. Assume first that $n$ is odd. The simple reflections are \[\begin{cases} s_i=(i,i+1)(2m+1-i,2m+2-i), &1\leq i\leq m-1\\
s_m=(m,m+2), &i=m, \end{cases}\] 
and we have \[ w_i=\begin{cases} s_1\cdots s_i,&0\leq i\leq m\\ s_1\cdots s_{m-1}s_ms_{m-1}\cdots s_{2m-i}, &m+1\leq i\leq 2m-1.
 \end{cases}\]
Now assume that $n$ is even. The simple reflections are \[\begin{cases} s_i=(i,i+1)(2m-i,2m+1-i), &1\leq i\leq m-1\\
s_m=(m-1,m+1)(m,m+2), &i=m, \end{cases}\] 
and we have \[ w_i=\begin{cases} s_1\cdots s_i,&0\leq i\leq m\\s_1\cdots s_ms_{m-2}\cdots s_{2m-1-i},& m+1\leq i\leq 2m-2, \end{cases}\]  and \[w_{m-1}'=s_1\cdots s_{m-2}s_m.\] 

Let $\ov{V}=L_W\otimes \ov{\F}_p$ be the quadratic space over $\ov{\F}_p$. For each $w_i\in \,^J\W$ we will attach to it an orthogonal $F$-zip (also called a $\SO(V)$-zip) as follows. Fix a basis $e_1,\dots, e_{n+2}$ of $L$ such that the quadratic form $Q$ has the form $x_1x_{n+2}+x_2x_{n+1}+\cdots+x_{m}x_{m+2}+x_{m+1}^2$ (cf. \cite{Zh} the proof of Proposition 4.4.1). By abuse of notation we still denote by $e_1,\dots, e_{n+2}$  the induced basis of $(\ov{V},Q)$. For each $w\in \, ^J\W$, let $M_w$ be the orthogonal $F$-zip $(\ov{V}, Q, C^\bullet, D_\bullet, \phi_\bullet)$ where
\begin{itemize}
	\item $C^\bullet$ is the descending filtration $\ov{V}\supset \lan e_2,e_3,\dots,e_{n+2}\ran\supset \lan e_{n+2}\ran \supset 0$, denoted by $C^0\supset C^1\supset C^2\supset C^3$,
	\item $D_\bullet$ is the ascending filtration $0\subset \lan w(e_1)\ran\subset \lan w(e_1), w(e_2),\dots,w(e_{n+1})\ran\subset \ov{V}$, denoted by $D_0\subset D_1\subset D_2\subset D_3$,
	\item $\varphi_\bullet$ is the collections of isomorphisms $\varphi_0: (C^0/C^1)^{(p)}\st{\sim}{\ra} D_1,\quad \varphi_1: (C^1/C^2)^{(p)}\st{\sim}{\ra} D_2/D_1,\quad \varphi_2: (C^2/C^3)^{(p)}\st{\sim}{\ra} D_3/D_2$.
\end{itemize}
We remark that the above construction is not the standard isomorphism $^J\W \simeq | G-\tr{Zip}^{\mu}|$ of Pink-Wedhorn-Ziegler (for example as in Theorem 3.1.5 of \cite{Zh}): the standard association is the twist \[w\longmapsto M_{w_0w}\] of ours, where $w_0$ is the maximal length element of $^J\W$. In particular $\ell(w_0w)=n-\ell(w)$.

\begin{theorem}\label{Thm:EO}
Each stratum $\M_{1w}$ is some union of Howard-Pappas strata of $\M_{1red}$.
\end{theorem}
\begin{proof}
	By the methods of \cite{HP}, it suffices to prove the following assertion first.
\end{proof}

\begin{corollary}\label{C:EO for SO}
Each stratum $\M_{w}$ is some union of Howard-Pappas strata of $\M_{red}$.
\end{corollary}
\begin{proof}
	
	We first prove the equalities for the sets of $k$-points, where $k$ is an algebraically closed field of characteristic $p$. This follows from \cite{HP} Theorem 6.5.6 and \cite{GoHe} Corollary 4.1.3. 
	
	Indeed, by \cite{HP} Theorem 6.5.6, we have an isomorphism \[\M_{\Lambda}^\circ\simeq X_B(w^+)\coprod X_B(w^-),\]where $X_B(w^+)$ and $X_B(w^-)$ are the Deligne-Lusztig varieties associated to the elements $w^+$ and $w^-$ of $\W_\Omega$, the Weyl group of $\SO(\Omega)$, where as before $\Omega=\Lambda_W/\Lambda^\vee_W$. As in \cite{HP} 6.5.4, $w^{\pm}$ are Coxeter elements. Write $w(\Lambda)=w^+$, and consider it as an element in $\W$ under the inclusion $\W_\Omega\hookrightarrow\W$. Then by \cite{GoHe} Corollary 4.1.3, we have \[\M_w(k)=\coprod_{\Lambda, w(\Lambda)=w}\M_{\Lambda}^\circ(k).\]
	
	To prove the identities on the level of schemes, we argue as in the proof of Corollary 4.10 of \cite{VW}. That is, it suffices to show that $\M_{\Lambda}^\circ$ is open and closed in $\M_{w}$. This follows from the facts that $\M_{\Lambda}^\circ$ is open  in $\M_{\Lambda}$,  $\M_{\Lambda}\cap \M_{w}=\M_{\Lambda}^\circ$,
	and the above identities on the level of points.
	
\end{proof}

Consider the case $k=\ov{\F}_p$.
For any vertex lattice $\Lambda$ and any point $x\in \M_{\Lambda}^\circ(\ov{\F}_p)$, we have the associated special lattice $\Ll_x$. Reduction modulo $p$, we get an orthogonal $F$-zip $M_x$, which we write it as $M_{w_0w_x}$ attached to $w_0w_x\in \,^J\W^b$ for some $w_x\in\, ^J\W^b$. Then by definition $x\in \M_{w_0w_x}$. By the above corollary, we have the equality \[d-1=\ell(w_0w_x)\]where $d=\frac{t(\Lambda)}{2}$. 
The following corollaries are coarser versions of Theorem \ref{Thm:EO} and Corollary \ref{C:EO for SO}. However, they are more explicit in terms of types.
\begin{corollary}
\begin{enumerate}
		\item If $n$ is odd, or $n$ is even with $\det (V)=(-1)^{\frac{n}{2}}$, then we have the following identity
\[ \M_{1w_i}=\coprod_{\Lambda, t(\Lambda)=2(n-i+1)}\M_{1\Lambda}^\circ. \]
\item If $n$ is even with $\det (V)\neq (-1)^{\frac{n}{2}}$, then 
\begin{enumerate}
	\item if $m\leq i\leq 2m-1$, \[ \M_{1w_i}=\coprod_{\Lambda, t(\Lambda)=2(n-i+1)}\M_{1\Lambda}^\circ. \]
	\item if $i=m-1$, \[ \M_{1w_{m-1}}\coprod \M_{1w_{m-1}'}=\coprod_{\Lambda, t(\Lambda)=2m}\M_{1\Lambda}^\circ. \]
\end{enumerate}

\end{enumerate}
\end{corollary}

\begin{corollary}
\begin{enumerate}
	\item If $n$ is odd, or $n$ is even with $\det (V)=(-1)^{\frac{n}{2}}$, then we have the following identity
	\[ \M_{w_i}=\coprod_{\Lambda, t(\Lambda)=2(n-i+1)}\M_{\Lambda}^\circ. \]
	\item If $n$ is even with $\det (V)\neq (-1)^{\frac{n}{2}}$, then 
	\begin{enumerate}
		\item if $m\leq i\leq 2m-1$, \[ \M_{w_i}=\coprod_{\Lambda, t(\Lambda)=2(n-i+1)}\M_{\Lambda}^\circ. \]
		\item if $i=m-1$, \[ \M_{w_{m-1}}\coprod \M_{w_{m-1}'}=\coprod_{\Lambda, t(\Lambda)=2m}\M_{\Lambda}^\circ. \]
	\end{enumerate}
	
\end{enumerate}
\end{corollary}

\subsection{Moduli spaces of polarized K3 surfaces with level structures and the integral Kuga-Satake map}\label{subsection:mod K3}
In this and the next subsection, we will turn to moduli spaces of polarized K3 surfaces, with the involved Shimura varieties, cf. \cite{MP2} sections 2 and 4, \cite{Ri2} section 6. 

Let $U$ be the hyperbolic lattice over $\Z$ of rank 2, and $E_8$ be the positive quadratic lattice associated to the Dynkin diagram of type $E_8$. Set $N=U^{\oplus3}\oplus E_8^{\oplus2}$, which is a self-dual lattice. Let $d\geq 1$ be an integer. Choose a basis $e, f$ for the first copy of $U$ in $N$ and set
\[L_d=\lan e-df\ran^\bot\subset N.\]This is a quadratic lattice over $\Z$ of discriminant $2d$ and rank 21 (in \cite{Ri2} it is denoted by $L_{2d}$ ). Let $V_d=L_d\otimes\Q$ and $L_d^\vee\subset V_d$ be the dual lattice. Set \[G=\SO(V_d),\] which is isomorphic to the special orthogonal group over $\Q$ of signature $(2, 19)$. Let $K\subset G(\A_f)$ be an open compact subgroup which stabilizes $L_{d,\wh{\Z}}$ and acts trivially on $L_d^\vee/L_d$. Such compact opens are called \emph{admissible}. We fix a prime $p>2$ such that $p\nmid d$ from now on. Then as $L$ is self dual at $p$, the local reductive group $G_{\Q_p}$ is unramified. Let $K_p=G(\Z_p)$ be the hyperspecial group. We only consider open compact subgroups $K^p\subset G(\A_f^p)$ which is contained in the discriminant kernel of $L_{d,\wh{\Z}^p}$ with finite index. In particular, $K=K_pK^p$ is admissible, cf. \cite{Ri2} 5.3. For the reductive group $G$, we have the associated Shimura varieties $\Sh_{K_pK^p}$, which are defined over $\Q$. By \cite{Ki1}, there exists an integral smooth canonical model $S_{K_pK^p}$ of $\Sh_{K_pK^p}$ over $\Z_p$.

Let $\Mm_{2d}$ (resp. $\Mm^\ast_{2d}$) be the moduli spaces of K3 surfaces $f: X\ra S$ together with a primitive polarization $\xi$ (resp. quasi-polarization) of degree $2d$ over $\Z_p$ (in \cite{MP2} section 2, these spaces are denoted by $\Mm_{2d}^\circ$ and $\Mm_{2d}$ respectively). These are Deligne-Mumford stacks of finite type over $\Z_p$. The natural map $\Mm_{2d}\ra \Mm_{2d}^\ast$ is an open immersion. Moreover, $\Mm_{2d}$ is separated and smooth of dimension 19 over $\Z_p$, cf. \cite{Ri2} Theorem 4.3.3, Proposition 4.3.11 and \cite{MP2} Proposition 2.2.

Let $(f: \mathcal{X}\ra \Mm_{2d}, \xi)$ be the universal object over $\Mm_{2d}$. For any prime $\ell$, we consider the second relative \'etale cohomology $H_\ell^2$ of $\mathcal{X}$ over $\Mm_{2d}$. This is a lisse $\Z_\ell$-sheaf of rank 22 equipped with a perfect symmetric Poincar\'e pairing $\lan,\ran: H_\ell^2\times H_\ell^2\ra \Z_\ell(-2)$. The $\ell$-adic Chern class $\ch_\ell(\xi)$ of $\xi$ is a global section of the Tate twist $H^2_\ell(1)$ that satisfies $\lan \ch_\ell(\xi),\ch_\ell(\xi)\ran=2d$.  The product \[H^2_{\wh{\Z}}=\prod_{\ell}H_\ell^2\] is a lisse $\wh{\Z}$-sheaf, and the Chern classes of $\xi$ can be put together to get the Chern class $\ch_{\wh{\Z}}(\xi)$ in $H^2_{\wh{\Z}}(1)$. Recall that we have the quadratic lattice $N$ of rank 22 over $\Z$.
\begin{definition}\label{D:level}
	Consider the \'etale sheaf over $\Mm_{2d}$ whose sections over any scheme $T\ra \Mm_{2d}$ are given by
	\[ I(T)=\{ \eta: N\otimes\wh{\Z}\st{\sim}{\ra} H^2_{\wh{\Z},T}(1)\,\tr{isometries},\,\tr{with}\,\eta(e-df)=\ch_{\wh{\Z}}(\xi) \}.\] Let $K=K_pK^p\subset K_{L_{\wh{\Z}^p}}$ be an admissible open compact subgroup. Then $I$ admits a natural action by the constant sheaf of groups $K$. A section $\ov{\eta}\in H^0(T,I/K)$ is called a $K$-level structure over $T$ (in \cite{Ri2} 5.3 it is called a full $K$-level structure). 
\end{definition}
Let $\Mm_{2d,K}$ (resp. $\Mm_{2d,K}^\ast$) be the relative moduli problem over $\Mm_{2d}$ (resp. $\Mm_{2d}^\ast$) which parametrizes $K$-level structures. For $K^p$ (thus $K$) small enough, these are smooth algebraic spaces. Moreover, the maps 
\[\Mm_{2d, K}\ra \Mm_{2d},\quad \Mm_{2d,K}^\ast\ra \Mm_{2d}^\ast\] are finite \'etale. For another admissible $K'=K_pK^{p'}\subset K=K_pK^p$, we have natural finite \'etale projections \[\Mm_{2d,K^{'}}\ra \Mm_{2d,K},\quad \Mm^{\ast}_{2d,K^{'}}\ra \Mm^\ast_{2d,K}\] as algebraic spaces over $\Mm_{2d}, \Mm_{2d}^\ast$ respectively.  When $K^{p'}$ is a normal subgroup of $K^p$, these projections are Galois with Galois group $K^p/K^{p'}$.

For any prime $\ell$, we have the primitive cohomology sheaf \[P_\ell=\lan \ch_\ell(\xi)\ran^\bot \subset H_\ell^2.\]Let $H_B^2$ and $H_{dR}^2$ be the second relative Betti and de Rham cohomology respectively of the universal K3 surface $\mathcal{X}\ra \Mm_{2d,K,\C}^\ast$. We have also the primitive cohomology sheaves
\[P_B=\lan \ch_B(\xi)\ran^\bot \subset H_B^2,\quad P_{dR}=\lan \ch_{dR}(\xi)\ran^\bot\subset H_{dR}^2.\]
Consider $\wt{\Mm}_{2d,K}^\ast\ra \Mm_{2d,K}^\ast$,  the two-fold finite \'etale cover parameterizing isometric trivializations $\det(L_d)\otimes\Z_2\st{\sim}{\ra} \det(P_2)$  of the determinant of the primitive 2-adic cohomology of the universal quasi-polarized K3 surface. We can identify $\wt{\Mm}_{2d,K}^\ast$ with the the space of isometric trivializations $\det(L_d)\st{\sim}{\ra} \det(P_B)$
of the determinant of the primitive Betti cohomology. There is a Hodge-de Rham filtration $F^\bullet P_{dR}$ on $P_{dR}$,  and we have a natural isometric trivialization $\eta: \tr{disc}(L_d)\st{\sim}{\ra} \tr{disc}(P_B)$ and the the tautological trivialization $\beta: \det(L_d)\st{\sim}{\ra}\det(P_B)$. The tuple $(P_B,F^\bullet P_{dR},\eta,\beta)$ gives rise to a natural period map
\[\wt{\Mm}_{2d,K,\C}^\ast\lra \Sh_{K,\C},\] cf. \cite{MP2} Propositions 4.2 and 3.3. There is a section map $\Mm_{2d,K,\C}\subset \Mm_{2d,K,\C}^\ast\ra \wt{\Mm}_{2d,K,\C}^\ast$, whose composition with the above period map gives us the Kuga-Satake period map
 \[\iota_\C: \Mm_{2d,K,\C}\lra \Sh_{K,\C}.\] It follows from \cite{Ri1} Theorem 3.9.1, this map is defined over $\Q$. 
Therefore we get the map over $\Q_p$ \[\iota_{\Q_p}: \Mm_{2d,K,\Q_p}\lra \Sh_{K,\Q_p}.\]
 As $S_{K}$ is the integral canonical model of $\Sh_{K}$, by extension property of $S_{K}$, the Kuga-Satake map extends to a map over $\Z_p$
 \[\iota: \Mm_{2d,K}\lra S_{K}.\]
\begin{theorem}[\cite{MP2} Corollary 5.15]\label{T:MP}
	The integral Kuga-Satake period map
	\[\iota: \Mm_{2d,K}\lra S_{K}\] is an open immersion.
\end{theorem}
When $K^p_1\subset K^p$ is another open compact subgroup, we note that the following diagram is cartesian:
\[\xymatrix{ \Mm_{2d, K_1}\ar[r]\ar[d]& S_{K_1}\ar[d]\\
	\Mm_{2d, K}\ar[r]& S_{K}.}\]
As a corollary, we see that for $K^p$ small enough, $\Mm_{2d, K}$ is a scheme.

\subsection{Newton and Ekedahl-Oort stratifications of the moduli spaces of K3 surfaces}
In the rest of this section we will work over $W$. As before we simply denote by the same notation for an object base changed to $W$.
Let $\ov{\Mm}_{2d, K}$ be the special fiber of $\Mm_{2d, K}$, which can be viewed as an open subspace of the special fiber $\ov{S}_{K}$ of $S_{K}$ by Theorem \ref{T:MP}. For the good reduction of Shimura varieties of abelian type, in \cite{SZ} we have introduced the Newton and Ekedahl-Oort stratifications for the special fibers. In subsection \ref{Section:newton} we have seen the Newton stratification. In the cases of GSpin and SO Shimura varieties, we can compare the Newton and Ekedahl-Oort stratifications as follows. These are in the list of Shimura varieties of coxeter type studied in \cite{GoHe} (comp. \cite{GoHeNi}).
\begin{theorem}[\cite{SZ}]
Assume that $n$ is odd\footnote{When $n$ is even, we have a similar but more delicate statement that each Newton stratum is a disjoint union of some Ekedahl-Oort strata, cf. \cite{SZ}.}.
\begin{enumerate}
	\item We have
	\[ \ov{S}_{K}=\coprod_{b\in B(G,\mu)} \ov{S}_{K}^b, \quad  \ov{S}_{K}=\coprod_{w\in ^J\W}\ov{S}_{K}^w,\]with each stratum in the two stratifications non empty. 
	\item Let $b_0$ be the unique basic element in $B(G,\mu)$. We have
	\begin{itemize}
		\item for $b\neq b_0$, there exists a unique $w_b\in\, ^J\W$ such that $ \ov{S}_{K}^b=\ov{S}_{K}^{w_b}$
		\item for $b_0$, $\ov{S}_{K}^{b_0}=\coprod_{
		w\in  ^J\W^{b_0}}\ov{S}_{K}^w$
	\end{itemize}
\end{enumerate}
\end{theorem}
Note that the subset $^J\W^{b}=\{w_b\}$ for any $b\neq b_0$.
In fact these statements are just the global analogue of Theorem \ref{T:FHN} in the setting of Shimura varieties of abelian type, cf. \cite{SZ} section 7 (see also \cite{GoHeNi} section 6, where the authors there assume that the axioms of \cite{HP} are verified).

We return to the case $n=19$.
Consider the Kuga-Satake map \[\ov{\iota}: \ov{\Mm}_{2d, K}\hookrightarrow \ov{S}_{K},\] which is an open immersion by Theorem \ref{T:MP}. The above stratifications of $\ov{S}_{K}$ in turn induce stratifications of $\ov{\Mm}_{2d, K}$
\[ \ov{\Mm}_{2d,K}=\coprod_{b\in B(G,\mu)} \ov{\Mm}_{2d,K}^b, \quad  \ov{\Mm}_{2d,K}=\coprod_{w\in\, ^J\W}\ov{\Mm}_{2d,K}^w,\]
where $\ov{\Mm}_{2d,K}^b$ and $\ov{\Mm}_{2d,K}^w$ are the pullbacks of the corresponding strata $\ov{S}_{K}^b$ and $\ov{S}_{K}^w$ under the open immersion $\ov{\iota}: \ov{\Mm}_{2d, K}\hookrightarrow\ov{S}_{K}$. We have the similar relation 
\begin{itemize}
	\item for $b\neq b_0$, there exists a unique $w_b\in\, ^J\W$ such that $ \ov{\Mm}_{2d,K}^b=\ov{\Mm}_{2d,K}^{w_b}$,
	\item for $b_0$, $\ov{\Mm}_{2d,K}^{b_0}=\coprod_{w\in\, ^J\W^{b_0}}\ov{\Mm}_{2d,K}^w$. We will also write  $\ov{\Mm}_{2d,K}^{b_0}$ as $\ov{\Mm}_{2d,K}^{ss}$ to indicate that it is the supersingular locus of $\ov{\Mm}_{2d, K}$.
\end{itemize}
We will investigate these stratifications in some more classical terms, which appeared already in the literature, see \cite{Og} for example.
\subsubsection{Newton stratification vs. height stratification}
Let $X$ be a K3 surface over a field $k$ of characteristic $p$. Consider the functor on local Artinian $k$-algebras with residue field $k$ defined by
\[\begin{split}
\Phi^2_{X/k}: (Art/k)&\ra (\textrm{Abelian groups})\\
R &\mapsto \ker \big(H^2_{et}(X\times \Spec R, \G_m)\ra H^2_{et}(X, \G_m)\big).
\end{split}
\]
It is pro-representable by a one-dimensional formal group $\wh{\textrm{Br}}(X)$, the so called formal Brauer group. The height $h$ of this formal Brauer group of the K3 surface $X$ satisfies $1\leq h\leq 10$ or $h=\infty$. 

The Newton slopes of the $F$-crystal $H^2_{cris}(X/W)$ are equal to $(1-\frac{1}{h}, 1, 1+\frac{1}{h})$. Thus the set $B(G,\mu)$ is in bijection with the set $\{1,\dots,10,\infty\}$. The basic element $b_0$ corresponds to $\infty$. We write $B(G,\mu)=\{b_1,\dots, b_{10}, b_{11}=b_0\}$. The Newton stratification of $\ov{\Mm}_{2d, K}$ is just the classical height stratification. By \cite{EvG}, for each $b\in B(G,\mu)$, the Newton stratum $\ov{\Mm}_{2d,K}^b$ is non empty.
\subsubsection{Ekedahl-Oort stratification vs. Artin invariant stratification}
Thanks to the recent proof of the Tate conjecture for K3 surfaces, we know that for a K3 surface $X$ over $\ov{\F}_p$,  $h=\infty$ if and only if its Picard rank $\rho=22$, i.e. it is Artin supersingular if and only if it is Shioda supersingular, cf. \cite{L} Theorem 2.3. We simply call $X$ supersingular in this case. Let $X$ be a supersingular K3 surface over $\ov{\F}_p$, then the discriminant of its N\'eron-Severi lattice is equal to \[-p^{2\sigma_0(X)}\] for some integer $1\leq \sigma_0(X)\leq 10$. The integer $\sigma_0(X)$ is called the Artin invariant of $X$.

By \cite{EvG}, we have an explicit description of the set $^J\W$ as \[\{w_1, \dots, w_{20}\},\] with $w_i$ corresponds to $b_i$ for $1\leq i\leq 10$, and for $11\leq i\leq 20$ the elements $w_i$ are basic.  The K3 surfaces in the stratum $\ov{\Mm}_{2d,K}^{w_i}$ have Artin invariant $21-i$. In particular, we note that the index $i$ in the description of the set $^J\W$ in subsection \ref{subsection:HPEO} (where $0\leq i\leq 19$ in our case) is shifted to $i+1$ here. By \cite{EvG}, for each $w\in\, ^J\W$, the Ekedahl-Oort stratum $\ov{\Mm}_{2d,K}^w$ is non empty.

\subsection{Rapoport-Zink type uniformization and Artin invariants}\label{subsection: K3 app}
In this final subsection, we make the link between Rapoport-Zink spaces and moduli spaces of K3 surfaces.

Let $\wh{\Mm}_{2d, K}$ and $\wh{S}_{K}$ be the formal completion of $\Mm_{2d,K}$ and $S_{K}$ along their special fibers respectively. Then the integral Kuga-Satake period map in Theorem \ref{T:MP} induces an open immersion of formal schemes
\[\wh{\iota}: \wh{\Mm}_{2d, K}\lra  \wh{S}_{K}.\]

Let $x_0\in \ov{\Mm}_{2d, K}$ be any point in the special fiber $\ov{\Mm}_{2d, K}$ of $\Mm_{2d, K}$, and $x=\ov{\iota}(x_0)$ be its image in $\ov{S}_K$. Let $b\in B(G,\mu)$ be the Newton point associated to $x$ and consider the corresponding formal Rapoport-Zink space $\breve{\M}=\breve{\M}_b$ for the group $\SO(V)$. The choice of the point $x$ determines a morphism of formal schemes
\[\Theta_x:  \breve{\M}\lra \wh{S}_{K}.\]
Denote by $\breve{\N}$ the pullback of $\breve{\M}$ under $\wh{\iota}: \wh{\Mm}_{2d, K}\ra  \wh{S}_{K}$. In other words, we get a cartesian diagram
\[\xymatrix{\breve{\N}\ar[r] \ar[d]&\breve{\M}\ar[d]^{\Theta_x}\\
	\wh{\Mm}_{2d, K}\ar[r]^{\wh{\iota}}&\wh{S}_{K},
	}\]
	with the upper horizontal map $\breve{\N}\ra \breve{\M}$ is an open immersion. By the moduli description of $\breve{\M}$, we get the following description of $\breve{\N}$: for any $R\in \Nilp_W^{sm}$, 
	\[\breve{\N}(R)=\{(X, (t_\alpha),\ov{\rho})\in \breve{\M}(R)\}\] where
	\begin{itemize}
	\item $(X, (t_\alpha),\rho)\in \breve{\M}_1(R)$, with $X=KS(Y)[p^\infty]$, where $Y$ is a K3 surface over $R$, $KS(Y)$ is the Kuga-Satake abelian scheme attached to $Y$ (cf. Theorem \ref{T:MP} and \cite{MP2} section 5),
	\item $\ov{\rho}$ is the $p^\Z$-orbit of $\rho$.
	 \end{itemize}
In particular, $\breve{\N}$ is stable under the action of $J_b(\Q_p)$ on $\breve{\M}$.
\begin{remark}
By construction, we have an open subspace $\breve{\N}_1\subset   \breve{\M}_1$, such that for any $R\in \Nilp_W^{sm}$, \[\breve{\N}_1(R)=\{(X,(t_\alpha),\rho)\}\] with $(X, (t_\alpha),\rho)\in \breve{\M}_1(R)$ as above. The space $\breve{\N}$ is given by $\breve{\N}=\breve{\N}_1/p^\Z$. On the level of affine Deligne-Lusztig varieties, we get subsets \[\N_{red}(\ov{\F}_p)\subset \M_{red}(\ov{\F}_p)=X_{\mu}^G(b),\quad \N_{1red}(\ov{\F}_p)\subset \M_{1red}(\ov{\F}_p)=X_{\mu_1}^{G_1}(b_1).\] In the case that $b$ is basic, it will be interesting to describe the above subsets by special lattices as in \cite{HP} section 5.
\end{remark}

We can apply the Rapoport-Zink uniformization theorem for $S_{K}$ to deduce a similar uniformization for $\Mm_{2d, K}$. Recall that as $\dim V=21$ is odd, the group $G=\SO(V)$ is adjoint.
\begin{corollary}\label{C:K3 RZ}
	Let $J_\phi$ be the pullback of $\Zm_{\phi,K^p}$ under the open immersion $\ov{\iota}: \ov{\Mm}_{2d, K}\hookrightarrow \ov{S}_{K}$. Then we have the following identity \[\wh{\Mm_{2d, K}}_{/J_\phi} =\coprod_{j\in I}\breve{\N}/\Gamma_j, \]
	where $I$ is certain countable set, and for each $j\in I$, $\Gamma_j\subset J_b(\Q_p)$ is some discrete subgroup (constructed as usual from the uniformization theorem of the last section). If moreover $b=b_0$ is basic, then $J_\phi=\ov{\Mm}_{2d, K}^{ss}$ which is the supersingular locus in $\ov{\Mm}_{2d, K}$, and the above disjoint union is finite.
\end{corollary}
Recently there has been a definition of isogeny between two K3 surfaces in characteristic $p$, cf. \cite{Yang}. One can cheek that the locus $J_\phi$ parametrizes an isogeny class of polarized K3 surfaces.
\begin{remark}
If the open compact subgroup $K=K_pK^p\subset G(\A_f)$ ($K_p=G(\Z_p)$) is the image of some open compact subgroup $K_1=K_{1p}K_1^p\subset G_1(\A_f)$ ($K_{1p}=G_1(\Z_p)$), then it will be much easier to prove the uniformization theorem for $S_K$: one can work directly on the finite level and take a finite \'etale quotient from the corresponding Rapoport-Zink uniformization for $G_1$,  cf. \cite{Sh1} section 4 for example.
\end{remark}

Assume that $b=b_0$ is basic. Let $\N_{red}$ be the reduced special fiber of  $\breve{\N}$. Then the Howard-Pappas stratification of the reduced special fiber $\M_{red}$ of $\breve{\M}$ induces a similar stratification of the open subspace $\N_{red}$:
\[\N_{red}=\coprod_{\Lambda}\N_{\Lambda}^\circ, \]where $\N_{\Lambda}^\circ\subset \N_{red}$ is the pullback of the stratum $\M_{\Lambda}^\circ\subset \M_{red}$. For each  $w_i\in\, ^J\W^b$, consider the corresponding Ekedahl-Oort stratum
\[\M_{w_i}=\coprod_{\Lambda, t(\Lambda)=2(21-i)}\M_\Lambda^\circ,\quad \N_{w_i}=\coprod_{\Lambda, t(\Lambda)=2(21-i)}\N_\Lambda^\circ.\]For each $11\leq i\leq 20$, the image of $\N_{w_i}$ under the uniformization morphism gives us the corresponding Ekedahl-Oort stratum $\ov{\Mm}_{2d,K}^{w_i}$ in supersingular locus. 

For $(X,\xi)\in  \ov{\Mm}_{2d, K}^{ss}(\ov{\F}_p)$, consider \[\Ll=\lan \ch_{cris}(\xi)\ran ^\bot\subset H^2_{cris}(X/W).\]This is a special lattice in the sense of Definition 5.2.1 of \cite{HP}. Then we can apply Proposition 5.2.2 of loc. cit. to produce a vertex lattice $\Lambda(\Ll)$. For any integer $r\geq 0$ define 
\[\Ll^{(r)}=\Ll+\Phi(\Ll)+\cdots+\Phi^r(\Ll). \]Then there is a unique integer $1\leq d\leq 10$ such that
\[\Ll=\Ll^{(0)}\subsetneq \Ll^{(1)}\subsetneq\cdots\subsetneq \Ll^{(d)}=\Ll^{(d+1)}.\]The vertex lattice $\Lambda(\Ll)$ is defined by
\[\Lambda(\Ll)=(\Ll^{(d)})^{\Phi}. \]It has type \[t\big(\Lambda(\Ll)\big)=2d\] and $\Lambda(\Ll)^\vee=\Ll^\Phi$.
The following corollary follows from the above uniformization and Corollary \ref{C:EO for SO}.
\begin{corollary}\label{C:K3 Artin}
	Under the uniformization identity \[ \ov{\Mm}_{2d, K}^{ss}=\coprod_{j\in I}\N_{red}/\Gamma_j, \] the Ekedahl-Oort stratum $\ov{\Mm}_{2d, K}^{w_i}$ for each $11\leq i\leq 20$ is the image of $\N_{w_i}$. In particular, if $x\in\ov{\Mm}_{2d, K}^{ss}(\ov{\F}_p)$, let $X_x$ be the associated supersingular K3 surface over $\ov{\F}_p$, then we have the identity between the Artin invariant $\sigma_0(X_x)$ and the type $t(\Lambda_x)$ \[\sigma_0(X_x)=\frac{t(\Lambda_x)}{2},\]where $\Lambda_x=\Lambda(\Ll_x)$ is the vertex lattice attached to the special lattice associated to $(X_x,\xi_x)$ as above.
\end{corollary}

\appendix

\section{Admissibility and weakly admissibility in the basic orthogonal case}

In this appendix, we investigate the $p$-adic period domains $\Fl\ell_{G,\mu}^{adm}$ and $\Fl\ell_{G,\mu}^{wa}$ in the case $b$ is basic and $G=\SO$. Although the following Theorem \ref{T: wa=a SO} appears as a special case of our more recent work \cite{CFS}, we still would like to present it here, since it provides some concrete computations, which can be served as a good example-based introduction to our proof for Theorem 6.1 in \cite{CFS} in the direction $(1)\Rightarrow (2)$.  All the following materials are taken from \cite{F4}. We thank Fargues sincerely for kindly allowing us to include it here. 

Let $V=\Q_p^n$ equipped with the quadratic form $Q$ with matrix $\left( \begin{array}{ccc}
& & 1\\
& \iddots& \\
1& &
\end{array} \right)$. 
Let $G=\SO(V, Q)$ and consider the minuscule cocharacter $\mu: \G_m\ra G_{\ov{\Q}_p}$ given by $\mu(z)=\tr{diag}(z, 1, \cdots, 1, z^{-1})$. Then the basic class in $B(G,\mu)$ is $[b]=[1]$ and thus $J_b=G$. One checks easily that any non basic Newton polygon has a non trivial contact point with the Hodge polygon, i.e. $(G, \{\mu\})$ is fully Hodge-Newton decomposable in the sense of \cite{GoHeNi} Definition 2.1.

For simplicity, we write $\Fl\ell=\Fl\ell_{G,\mu}$ as the $p$-adic flag variety,  $\Fl\ell^{wa}=\Fl\ell_{G,\mu}^{wa}$, and $\Fl\ell^{adm}=\Fl\ell_{G,\mu}^{adm}$. We first describe the weakly admissible locus $\Fl\ell^{wa}$. The associated isocrystal is $\breve{\Q}_p^n$ with Frobenius $\sigma^{\oplus n}$. The sub-isocrystals are in bijection with the sub $\Q_p$-vector space of $V$. Let $C$ be a complete and algebraically closed extension of $\breve{\Q}_p$. Then we have
\[\Fl\ell(C,\Ol_C)=\{\tr{Lagrangian lines}\, D\subset V_C \} .\]
It follows that $\Fl\ell\subset \mathbb{P}^n_{\breve{\Q}_p}$ is the quadric defined by the equation $\sum_{i=1}^nx_ix_{n-i+1}=0$. Let $\Q_p^{[\frac{n}{2}]}\oplus (0)\subset V$ be a Lagrangian subspace with associated parabolic subgroup $P\subset G$. For any line $D\in \Fl\ell(C,\Ol_C)$ we attach to it the following Hodge filtration
\[0\subset \Fil^1=D\subset \Fil^0=D^\perp\subset \Fil^{-1}=V_C. \]Then 
\[\Fl\ell^{wa}(C, \Ol_C)=\{D\in \Fl\ell(C,\Ol_C)|\, D\cap W_C=0, \, \forall \, \tr{totally isotropic subspace} \,W\subset V \}.\]
Therefore, we get
\begin{proposition}
\[ \Fl\ell^{wa}=\Fl\ell\setminus G(\Q_p) S^{ad}, \]
where $S^{ad}$ is the adic space associated to the Schubert variety attached to $P$ ($S$ is defined by the locus $x_{[\frac{n}{2}]+1}=\cdots=x_n=0$ inside $\Fl\ell$).
\end{proposition}

Now we look at the admissible locus $\Fl\ell^{adm}$ (cf. \cite{Ra2} Definition A.6 or \cite{CFS} Definition 3.1). We have the following
\begin{theorem}\label{T: wa=a SO}
$\Fl\ell^{adm}=\Fl\ell^{wa}$.
\end{theorem}
\begin{proof}
For any point $x\in \Fl\ell^{wa}(C, \Ol_C)$, let $\E_x$ be the associated modification of $\Ol_X^n$ such that the relative position of $(B_{dR}^+)^n$ and $\wh{\E}_{x,\infty}$ is bounded by $\mu$. Here $X$ is the Fargues-Fontaine curve over $\Q_p$ associated to the perfectoid field $C^\flat$, and $\infty=x_C\in X$ is the point defined by $C$. We need to show this weakly admissible modification is in fact an admissible modification (i.e. $\E_x$ is semi-stable of slope 0).

By \cite{Ra2} Proposition A. 9, we have either
\[\E_x\simeq \Ol_X(\frac{1}{r})\oplus \Ol_X^{n-2r}\oplus\Ol_X(-\frac{1}{r}) \]
for some integer $1\leq r\leq [\frac{n}{2}]$, or
\[\E_x\simeq \Ol_X^n.\]The second case is admissible. We have to show this is always the case. Suppose that we are in the first case: we will find a contradiction. The perfect quadratic form on $\E_x$ is such that for any $\lambda\in\Q$, we have $(\E_x^{\geq \lambda})^\perp =\E_x^{>-\lambda}$, where $\E_x^\lambda\subset \E_x$ is a step in the Harder-Narasimhan filtration of $\E_x$. Therefore, we get
\[\Ol_X(\frac{1}{r})^{\perp}=\Ol_X(\frac{1}{r})\oplus \Ol_X^{n-2r}\]
and $\Ol_X(\frac{1}{r})$ is totally isotropic. It follows that there exists a unique sub vector bundle $\Fl\subset \Ol_X^n$ which is a locally direct summand, such that the modification $\E_x|_{X\setminus\infty}\st{\sim}{\ra}\Ol_{X\setminus\infty}^n$ induces a modification
\[\Ol_X(\frac{1}{r})|_{X\setminus\infty}\st{\sim}{\lra}\Fl|_{X\setminus\infty}.\]
In particular, $\Fl$ is totally isotropic in $\Ol_X^n$. Such a modification is necessarily of one of the following types:
\begin{enumerate}
	\item $(-1,0,\dots, 0)$,
	\item $(0, \dots, 0, 1)$,
	\item $(0, \dots, 0)$.
\end{enumerate}
Indeed, it suffices to look at for all the sub $B_{dR}$-vector spaces $E$ of $B_{dR}^n$, the relative positions of the lattices $E\cap (B_{dR}^+)^n$ and $E\cap \lan te_1, e_2, \dots, e_{n-1}, t^{-1}e_n\ran$, where $e_1,\dots, e_n$ is a basis of $V$. As $\Ol_X^n$ is semi-stable, we have $\deg(\Fl)\leq 0$. By looking at the above three cases, we get that $\Fl$ is a degree $-1$ modification of $\Ol_X(\frac{1}{r})$. Thus, 
\[\Fl\simeq \Ol_X^r,\]that is $\Fl=W\otimes\Ol_X$ for some totally isotropic subspace $W\subset \Q_p^n$ of dimension $r$. This implies that our modification $\E_x|_{X\setminus\infty}\st{\sim}{\ra}\Ol_{X\setminus\infty}^n$ is not weakly admissible. Thus we get a contradiction.
\end{proof}

\end{document}